\documentclass[11pt,a4paper]{article}

\usepackage[all]{xy}
\usepackage{latexsym}
\usepackage[margin=3cm]{geometry}
\usepackage[dvipdfmx]{graphicx}

\usepackage[colorlinks=true,linkcolor=blue,citecolor=red]{hyperref}
\usepackage{mymacro}
\usepackage{tikz}
\usetikzlibrary{decorations,decorations.pathmorphing, intersections, arrows.meta}
\usepackage{mathtools}

\renewcommand{\labelenumi}{$\mathrm{(\roman{enumi})}$}
\renewcommand{\labelenumii}{$\mathrm{(\alph{enumii})}$}

\SelectTips{cm}{}

\makeatletter

\@addtoreset{equation}{section}
\makeatother

\title{Sheaf quantization and intersection of \\ rational Lagrangian immersions}
\author{Tomohiro Asano \and Yuichi Ike}
\date{\today}
\AtEndDocument{{
	\noindent Tomohiro Asano: 
    Graduate School of Mathematical Sciences, The University of Tokyo, 3-8-1 Komaba, Meguro-ku, Tokyo, 153-8914, Japan.

	\noindent \textit{E-mail address}: \texttt{tasano@ms.u-tokyo.ac.jp}, \texttt{tomoh.asano@gmail.com}

	\medskip

	\noindent Yuichi Ike:
	Graduate School of Information Science and Technology, The University of Tokyo, 7-3-1 Hongo, Bunkyo-ku, Tokyo 113-8656, Japan.

	\noindent
	\textit{E-mail address}: \texttt{ike@mist.i.u-tokyo.ac.jp}, \texttt{yuichi.ike.1990@gmail.com}
}}

\begin{document}
\maketitle

\begin{abstract}
    We study rational Lagrangian immersions in a cotangent bundle, based on the microlocal theory of sheaves. 
    We construct a sheaf quantization of a rational Lagrangian immersion and investigate its properties in Tamarkin category.
    Using the sheaf quantization, we give an explicit bound for the displacement energy and a Betti/cup-length estimate for the number of the intersection points of the immersion and its Hamiltonian image by a purely sheaf-theoretic method.
\end{abstract}

\section{Introduction}\label{section:introduction}

\subsection{Sheaf-theoretic bound for displacement energy}

The microlocal theory of sheaves due to Kashiwara--Schapira~\cite{KS90} has been effectively applied to symplectic geometry for a decade. 
After the pioneering works by Nadler--Zaslow~\cite{NZ,Nad} and Tamarkin~\cite{Tamarkin}, numerous theorems related to symplectic geometry have been proved by sheaf-theoretic methods (for example, see~\cite{GKS,Chiu17,Gu19}).
Now the theory is considered to be a powerful tool other than Floer theory for the study of symplectic geometry. 

In~\cite{AI20}, the authors gave a purely sheaf-theoretic bound for the displacement energy of compact subsets in a cotangent bundle.
Let $M$ be a connected manifold without boundary and denote by $T^*M$ its cotangent bundle equipped with the canonical symplectic form $\omega$. 
For a compactly supported $C^\infty$-function $H=(H_s)_{s \in [0,1]} \colon T^*M \times [0,1] \to \bR$, we define
$\| H \| \coloneqq \int_0^1 \left(\max_p H_s(p) - \min_p H_s(p) \right) ds$ and let $\phi^H=(\phi^H_s)_s \colon T^*M \times [0,1] \to T^*M$ denote the generated Hamiltonian isotopy.
For given compact subsets $A$ and $B$ of $T^*M$, their displacement energy is the infimum of $\|H\|$ such that $A \cap \phi^H_1(B)=\emptyset$.
A sheaf-theoretic tool to estimate displacement energy is Tamarkin category~\cite{Tamarkin}.
We denote by $\cD(M)$ the Tamarkin category of $M$, which is defined as a quotient category of the bounded derived category $\Db(M \times \bR)$ of sheaves of vector spaces over the field $\bF_2=\bZ/2\bZ$ on $M \times \bR$.
For a compact subset $A$ of $T^*M$, $\cD_A(M)$ denotes the full subcategory of $\cD(M)$ consisting of objects whose microsupports are contained in the cone of $A$ in $T^*(M \times \bR)$.
We also denote by $\cHom^\star \colon \cD(M)^{\op} \times \cD(M) \to \cD(M)$ the canonical internal Hom functor.
For an object $F \in \cD(M)$ and $c \in \bR_{\ge 0}$ there exists a canonical morphism $\tau_{0,c}(F) \colon F \to {T_c}_*F$, where $T_c \colon M \times \bR \to M \times \bR$ is the translation by $c$ to the $\bR$-direction.

\begin{theorem}[{\cite[Thm.~4.18)]{AI20}}]\label{theorem:intro-previous}
    Denote by $q \colon M \times \bR \to \bR$ the projection and let $A,B$ be compact subsets of $T^*M$.
    For $F \in \cD_{A}(M), \allowbreak G \in \cD_{B}(M)$, if 
    \begin{equation}
        \| H \| 
        < 
        \inf \{ c \in \bR_{\ge 0} \mid \tau_{0,c}(Rq_* \cHom^\star(F,G))=0 \},
    \end{equation}
    then $A \cap \phi_1^H(B) \neq \emptyset$.
\end{theorem}

\cref{theorem:intro-previous} asserts that if we find sheaves associated with given compact subsets, we can estimate their displacement energy using these sheaves.
However, it says nothing about the existence of such objects and we could only use sheaves associated with some concrete examples or compact exact Lagrangian submanifolds~\cite{Gu19}.
In this paper, for a certain class of Lagrangian immersions, we construct such objects that give an explicit bound of the displacement energy, based on sheaf quantization.

\subsection{Sheaf quantization}\label{subsec:intro-quantization}

For a subset of a cotangent bundle (resp.\ a cosphere bundle), in particular a Lagrangian (resp.\ Legendrian) submanifold, a \emph{sheaf quantization} is a sheaf whose microsupport coincides with the subset.
The process to associate a sheaf quantization with a given subset is also called sheaf quantization.
Since the microsupport of a sheaf is always conic, for a non-conic subset of $T^*M$, we conify it adding one variable $\bR$.
In this way, we obtain a conic subset of $T^*(M \times \bR)$ or equivalently a subset of the cosphere bundle $ST^*(M \times \bR)$ and we could construct a sheaf quantization of the subset.

Guillermou~\cite{Gu12,Gu19} constructed a sheaf quantization of a compact exact Lagrangian submanifold of $T^*M$, whose conification is a Legendrian submanifold of $ST^*(M \times \bR)$ with no Reeb chords.
He first constructed a sheaf on $M \times (0,+\infty)\times \bR$, which can be regarded as a family of objects of $\Db(M \times \bR)$ parametrized by $(0,+\infty)$, and then obtained a sheaf quantization on $M \times \bR$ as a limit of the family at $+\infty$.

In this paper, we construct a sheaf quantization of a general compact Legendrian submanifold of $ST^*(M \times \bR/\theta \bZ)$ with some $\theta \in \bR_{\ge 0}$ (see \cref{theorem:existence-quantization,remark:legendrian}).
Such a Legendrian is a conification of a strongly rational Lagrangian immersion (see \cref{definition:strongly-rational} and \cref{remark:legendrian}).
For the construction, we follow the idea of Guillermou~\cite{Gu19}.
In our setting, his construction is obstructed by the existence of Reeb chords and it only gives a family parametrized by $(0,a)$, where $a$ is a positive real number less than the shortest length of the Reeb chords. 
In this way, we obtain an object $G_{(0,a)}$ of the category $\Db_{/[1]}(M \times (0,a) \times \bR/\theta\bZ)$, where $\Db_{/[1]}(X)$ denotes the triangulated orbit category of sheaves on a manifold $X$ (see \cref{section:preliminaries}).
We also call this object $G_{(0,a)}$ a sheaf quantization.

\begin{remark}\label{remark:intro-modification}
    The reasons why we construct a sheaf quantization as an object of $\Db_{/[1]}(M \times (0,a) \times \bR/\theta\bZ)$ is the following threefold:
    \begin{enumerate}
        \item The appearance of $\bR/\theta \bZ$ comes from the fact that a primitive of the Liouville $1$-form on a strongly rational Lagrangian takes value only modulo $\theta$.
        \item Using a sheaf on $M \times (0,a) \times \bR/\theta\bZ$ rather than one on $M \times \bR/\theta\bZ$ is the essential idea to obtain better energy estimates as in \cref{theorem:introbetti,theorem:introcuplength} below.
        The restriction of the sheaf to $M \times \{u\} \times \bR/\theta\bZ$ for some $u \in (0,a)$ can also give an energy estimate but it is worse in general.
        \item We can only construct a sheaf quantization as an object of the triangulated orbit category $\Db_{/[1]}$ because of the existence of an obstruction class, which is related to the Maslov class (see \cite[\S10.3]{Gu19}).
        This is why we use the orbit category instead of the usual derived category.
    \end{enumerate}
\end{remark}

\subsection{Intersection of rational Lagrangian immersions}

Based on \cref{theorem:intro-previous} and sheaf quantization introduced in \cref{subsec:intro-quantization}, 
we give an explicit bound for the displacement energy of a rational Lagrangian immersion with a purely sheaf-theoretic method.
Not only a bound for the energy, we also give an estimate of the number of intersection points by the total Betti number and the cup-length of the Lagrangian.

\begin{definition}
    \begin{enumerate}
        \item A Lagrangian immersion $\iota\colon L\to T^*M$ is said to be \emph{rational} if there exists $\sigma(\iota) \in \bR_{\ge 0}$ such that
        \begin{equation}
            \left\{ \int_{D^2}v^*\omega \; \middle| \; (v,\bar{v}) \in \Sigma(\iota)
        \right\}
        =
        \sigma(\iota) \cdot \bZ,
        \end{equation}
        where
        \begin{equation}
        \Sigma(\iota)
         \coloneqq 
        \left\{ (v,\bar{v})
        \; \middle| \;
        \begin{aligned}
        & v \colon D^2 \to T^*M, \bar{v} \colon
        \partial D^2 \to L, \\
        & v|_{\partial D^2}=\iota \circ \bar{v}
        \end{aligned}
        \right\}. 
        \end{equation}
        \item For a rational Lagrangian immersion $\iota \colon L \to T^*M$, one defines
        \begin{equation}
        e(\iota)
         \coloneqq 
        \inf \left( \left\{ \int_{D^2}v^*\omega \; \middle| \; (v,\bar{v}) \in E(\iota) \amalg \Sigma(\iota) \right\} \cap \bR_{>0} \right),
        \end{equation}
        where
        \begin{equation}
        E(\iota)
         \coloneqq 
        \left\{ (v,\bar{v})
        \; \middle| \;
        \begin{aligned}
        & v \colon D^2 \to T^*M, \bar{v} \colon
        [0,1] \to L, \\
        & \bar{v}(0) \neq \bar{v}(1),
        \iota \circ \bar{v}(0) = \iota \circ \bar{v}(1), \\
        & v|_{\partial D^2}\circ \exp(2\pi \sqrt{-1} (-))=\iota \circ \bar{v}
        \end{aligned}
        \right\}.
        \end{equation}
    \end{enumerate}
\end{definition}

Here we put the following additional assumption. 

\begin{assumption}\label{assumption:intro}
	There exists no $(v,\bar{v}) \in E(\iota)$ with $\int_{D^2}v^*\omega =0$.
\end{assumption}

Our explicit bounds are the following.

\begin{theorem}[{see \cref{theorem:betti-estimate}}]\label{theorem:introbetti}
	Let $\iota \colon L \to T^*M$ be a compact rational Lagrangian
	immersion satisfying \cref{assumption:intro}.
	If $\|H\| <e(\iota)$ and $\iota \colon L \to T^*M$ intersects $\phi^H_1
	\circ \iota \colon L \to T^*M$ transversally, then
	\begin{equation}
		\# \left\{ (y,y') \in L \times L \; \middle| \; \iota(y)=\phi^H_1 \circ
		\iota(y') \right\}
		\ge
		\sum_{i=0}^{\dim L} b_i(L).
	\end{equation}
\end{theorem}

\begin{theorem}[{see \cref{theorem:cuplength}}]\label{theorem:introcuplength}
	Let $\iota \colon L \to T^*M$ be a compact rational Lagrangian
	immersion satisfying \cref{assumption:intro}.
	If a Hamiltonian function $H$ satisfies $\|H\| < \min \left( \{e(\iota)\} \cup (\{ \sigma(\iota)/2 \}\cap \bR_{>0})\right)$, then
	\begin{equation}
		\# \left\{ (y,y') \in L \times L \; \middle| \; \iota(y)=\phi^H_1 \circ
		\iota(y') \right\}
		\ge
		\cl (L)+1,
	\end{equation}
	where
	$\cl (L)$ denotes the cup-length of $L$ over $\bF_2$.
\end{theorem}

In particular, \cref{theorem:introbetti} gives a bound for the displacement energy of the image of a rational Lagrangian immersion.
\smallskip

In what follows, we give the outline of our proof.

First, we can reduce the problems to the case of a strongly rational Lagrangian immersion, where \cref{assumption:intro} guarantees that the associated Legendrian is a submanifold of $ST^*(M \times \bR/\theta \bZ)$.
Let $a \in \bR_{>0}$ be less than the shortest length of the Reeb chords of the Legendrian, which is related to $e(\iota)$.
As mentioned in \cref{subsec:intro-quantization}, we can associate a sheaf quantization $G_{(0,a)} \in \Db_{/[1]}(M \times (0,a) \times \bR/\theta\bZ)$ with the Legendrian submanifold.
Using the quantization object, we define two objects $F_{(0,a)} \coloneqq R{j_{a}}_!G_{(0,a)}, F_{[0,a]} \coloneqq R{j_{a}}_*G_{(0,a)} \in \Db_{/[1]}(M \times \bR \times \bR/\theta\bZ)$, where $j_a$ is the inclusion $M \times (0,a) \times \bR/\theta\bZ \to M \times \bR \times \bR/\theta\bZ$.

To use the objects $F_{(0,a)}$ and $F_{[0,a]}$ effectively, we introduce a modified version of Tamarkin category $\cD^{P}(M)_{\theta}$ parametrized by a manifold $P$ with period $\theta$, which is defined as a quotient category of $\Db_{/[1]}(M \times P \times \bR/\theta\bZ)$.
In the case $P=\pt$ and $\theta=0$, the category recovers (an orbit version of) the usual Tamarkin category $\cD(M)$. 
Similarly to the case of $\cD(M)$, we can define a canonical internal Hom functor $\cHom^\star \colon \cD^P(M)^{\op}_\theta \times \cD^P(M)_\theta \to \cD^P(M)_\theta$.
For $c \in \bR$ we denote by $T_c$ the translation by $c$ on $\bR/\theta\bZ$ modulo $\theta$.
In this setting, we can show that the results in~\cite{AI20} including \cref{theorem:intro-previous} also hold (see \cref{section:Tamarkincat} and \Cref{section:tamarkin-precise}).
This modification corresponds to the family version of the previous results and gives a better energy estimate in some cases.
Indeed, $F_{(0,a)}$ and $F_{[0,a]}$ define objects of $\cD^\bR(M)_\theta$, and we can prove that $\tau_{0,c}(Rq_*\cHom^\star(F_{(0,a)},F_{[0,a]}))$ is non-zero for any $0\le c <a$, where $q \colon M \times \bR \times \bR/\theta \bZ \to \bR/\theta \bZ$ is the projection.
We cannot obtain such estimate with $G_{(0,a)}|_{M \times \{u\} \times \bR/\theta\bZ} \in \cD^{\pt}(M)_\theta$ for some $u \in (0,a)$ and this is why we use the parametrized version.

For the bounds for the number of the intersection points, similarly to~\cite{Ike19} we study the object $\cHom^\star(F_{(0,a)}, F_{[0,a]}^H)$, where $F_{[0,a]}^H$ denotes the Hamiltonian deformation of $F_{[0,a]}$ associated with $\phi^H_1$.
We find that its microsupport is related to the intersection $\# \left\{ (y,y') \in L \times L \; \middle| \; \iota(y)=\phi^H_1 \circ \iota(y') \right\}$ and for any $c \in \bR$
\begin{equation}\label{eq:intro-microlocal-stalk}
    H^*\RG_{[c,+\infty)}(\ell^! Rq_*\cHom^\star(F_{(0,a)},F^H_{[0,a]}))_{c}
    \simeq
    H^*\RG (\Omega_+; \mu hom(F_{(0,a)}, {T_{-c}}_*F^H_{[0,a]})),
\end{equation}
where $\ell\colon \bR\to \bR/ \theta \bZ$ is the quotient map and $\Omega_+=\{ \tau>0 \} \subset T^*(M \times \bR \times \bR/\theta \bZ)$ with $(t;\tau)$ being the homogeneous coordinate on $T^*(\bR/\theta \bZ)$.
We study an action of $H^*(L)=\bigoplus_{i \in \bZ}H^i(L;\bF_2)$ on \eqref{eq:intro-microlocal-stalk} and even compute the right-hand side explicitly in the case where the intersection is transverse. 
These computations are more difficult than previous works such as~\cite{Gu19,Ike19} because of the singularity of the microsupports of $F_{(0,a)}$ and $F_{[0,a]}$.
Combining the computations with the Morse inequality for sheaves and some properties of an algebraic counterpart of cup-length, we obtain the theorems. 
A benefit of the microlocal theory of sheaves especially appears in the proof of the cup-length bound.
We prove the triviality of $H^*(L)$-action on the each contribution by showing that the action factors through a microlocal category.
This microlocal sheaf-theoretic proof is more straightforward than that via Floer theory, which needs an unusual construction of suitable moduli spaces~\cite{Liu, AH16}. 

\subsection*{Related topics}

Sheaf quantization has been studied in several situations.
Guillermou--Kashiwara--Schapira \cite{GKS} constructed a sheaf quantization of the graph of a Hamiltonian isotopy.
Guillermou \cite{Gu12,Gu19} constructed a sheaf quantization of a compact exact Lagrangian submanifold of a cotangent bundle and applied it to the study of the topology of the Lagrangian.
Note that Viterbo~\cite{Vi19} also constructed a sheaf quantization of a compact exact Lagrangian submanifold, based on Floer theory. 
Jin--Treumann~\cite{JT17} studied the relation between sheaf quantization and brane structures.

The microlocal theory of sheaves is also applied to quantitative problems in symplectic geometry. 
Indeed, Tamarkin~\cite{Tamarkin} already mentioned an action of Novikov ring on Tamarkin category (see also~\cite[Rem.~4.21]{AI20}).
Chiu~\cite{Chiu17} proved a non-squeezing result based on a sheaf-theoretic method. 
See also the recent textbook by Zhang~\cite{Zhang20} for the quantitative aspect.

Results similar to \cref{theorem:introbetti,theorem:introcuplength} were also proved by Floer-theoretic methods for a compact symplectic manifold, without \cref{assumption:intro}.
Chekanov~\cite{Chekanov98} proved \cref{theorem:introbetti} for a rational embedding $\iota$ with $\sigma(\iota)>0$ and Akaho~\cite{Akaho} proved \cref{theorem:introbetti} for an exact immersion $\iota$, which corresponds to the condition $\sigma(\iota)=0$.
Liu~\cite{Liu} gave a proof of \cref{theorem:introcuplength} for a rational Lagrangian embedding with a better bound for $\|H\|$.
Floer-theoretic approach could give better estimates in the cases where a bounding cochain exists~\cite{FOOO09,FOOO092,FOOO13,AJ10}.
See also \cref{remark:comparisontoFloer,remark:comparisontoFOOO}. 

Both of the constructions of Floer homology groups and sheaf quantizations are obstructed by Reeb chords. 
Augmentations or bounding cochains were originally introduced to resolve the obstruction to construct Legendrian contact homology or Lagrangian Floer homology. 
Augmentations are also related to sheaf quantization of Legendrians \cite{NRSSZ20,ABS19,RS19}. 
In this work, we construct sheaf quantizations in a more general setting without assuming the existence of augmentations or bounding cochains, though our quantizations have less information than the quantizations in \cite{NRSSZ20,ABS19,RS19}.
Combining our argument in this paper with sheaf quantization with augmentations, we expect that we could obtain a better estimate for the displacement energy.

\subsection*{Organization}

The structure of the paper is as follows.
In \cref{section:preliminaries}, we give some results of the microlocal theory of sheaves in the triangulated orbit category.
In \cref{section:Tamarkincat}, we introduce the modified version of Tamarkin category and give some refined versions of the results of Asano--Ike~\cite{AI20}.
In \cref{section:constrution}, we construct a sheaf quantization of a strongly rational Lagrangian immersion in a cotangent bundle.
In \cref{section:rational-immersions}, we prove \cref{theorem:introbetti,theorem:introcuplength} based on the results obtained in the previous sections.
In \Cref{section:tamarkin-precise}, we give details on the modified version of Tamarkin category.

\subsection*{Acknoledgments}	
The authors would like to express their gratitude to St{\'e}phane Guillermou and Takuro Mochizuki.
Indeed, the essential idea of this paper was obtained through the discussion they organized for the authors.
The authors also thank Tatsuki Kuwagaki for drawing their attention to sheaf quantization of Lagrangian immersions, Manabu Akaho, Takahiro Saito, and Pierre Schapira for helpful discussions, and Mikio Furuta for his encouragement and helpful advice. 
They thank Wenyuan Li for pointing out some errors in the previous version of this paper.
They also thank the anonymous referee for the careful reading of the paper and the constructive suggestions.
TA was supported by Innovative Areas Discrete Geometric Analysis for Materials Design (Grant No.~17H06461).
YI was supported by JSPS KAKENHI (Grant No.~21K13801) and ACT-X, Japan Science and Technology Agency (Grant No.~JPMJAX1903). 

\section{Preliminaries on microlocal sheaf	theory}\label{section:preliminaries}

In this paper, we assume that all manifolds are real manifolds of class $C^\infty$ without boundary.
Throughout this paper, let $\bfk$ be the field $\bF_2=\bZ/2\bZ$.

In this section, we recall some definitions and results from \cite{KS90,Gu19} and prepare some notions. 
We mainly follow the notation in \cite{KS90}.
Until the end of this section, let $X$ be a manifold.

\subsection{Geometric notions}\label{subsection:geometric}

For a locally closed subset $Z$ of $X$, we denote by $\overline{Z}$
its closure.
We also denote by $\Delta_X$ the diagonal of $X \times X$.
We denote by $TX$ the tangent bundle and by $T^*X$ the cotangent bundle of $X$, and write $\pi_X \colon T^*X \to X$ or simply $\pi$ for the projection.
For a submanifold $M$ of $X$, we denote by $T^*_MX$ the conormal bundle to $M$ in $X$.
In particular, $T^*_XX$ denotes the zero-section of $T^*X$.
We set $\rT X \coloneqq T^*X \setminus T^*_XX$.
For two subsets $S_1$ and $S_2$ of $X$, we denote by $C(S_1,S_2)
\subset TX$ the normal cone of the pair $(S_1,S_2)$.

With a morphism of manifolds $f \colon X \to Y$, we associate the following morphisms and commutative diagram:
\begin{equation}\label{diag:fpifd}
\begin{aligned}
\xymatrix{
	T^*X \ar[d]_-{\pi_X} & X \times_Y T^*Y \ar[d]^-\pi \ar[l]_-{f_d}
	\ar[r]^-{f_\pi} & T^*Y \ar[d]^-{\pi_Y} \\
	X \ar@{=}[r] & X \ar[r]_-f & Y,
}
\end{aligned}
\end{equation}
where $f_\pi$ is the projection and $f_d$ is induced by the transpose
of the tangent map $f' \colon TX \to X \times_Y TY$.

We denote by $(x;\xi)$ a local homogeneous coordinate system on
$T^*X$.
The cotangent bundle $T^*X$ is an exact symplectic manifold with the
Liouville 1-form $\alpha=\langle \xi, dx \rangle$.
Thus the symplectic form on $T^*X$ is defined to be $\omega=d\alpha$.
We denote by $a \colon T^*X \to T^*X,(x;\xi) \mapsto (x;-\xi)$ the
antipodal map.
For a subset $A$ of $T^*X$, $A^a$ denotes its image under the antipodal map $a$.
We also denote by $\bfh \colon T^*T^*X \simto TT^*X$ the Hamiltonian
isomorphism given in local coordinates by
$\bfh(dx_i)=-\partial / \partial\xi_i$ and $\bfh(d\xi_i)=\partial / \partial x_i$. 
We will identify $T^*T^*X$ and $TT^*X$ by $-\bfh$. 

\begin{notation}\label{notation:Px}
	For notational simplicity, we sometimes write $\{P(x)\}_X$ for $\{x\in X\mid P(x)\}$ if there is no risk of confusion. 
\end{notation}

\subsection{Microsupports of objects in orbit category}

We denote by $\bfk_X$ the constant sheaf with stalk $\bfk$ and by $\Module(\bfk_X)$ the abelian category of sheaves of $\bfk$-vector spaces on $X$.
Moreover, we denote by $\Db(\bfk_X)$ the bounded derived category of sheaves of $\bfk$-vector spaces.
One can define Grothendieck's six operations $\cRHom,\allowbreak
\otimes, \allowbreak Rf_*,\allowbreak f^{-1},\allowbreak
Rf_!,\allowbreak f^!$ for a continuous map $f \colon X \to Y$ with suitable conditions.
For a locally closed subset $Z$ of $X$, we denote by $\bfk_Z$ the
zero-extension of the constant sheaf with stalk $\bfk$ on $Z$ to $X$,
extended by $0$ on $X \setminus Z$.
Moreover, for a locally closed subset $Z$ of $X$ and $F \in \Db(\bfk_X)$,
we define $F_Z, \RG_Z(F) \in \Db(\bfk_X)$ by
\begin{equation}
F_Z \coloneqq F \otimes \bfk_Z, \quad \RG_Z(F) \coloneqq \cRHom(\bfk_Z,F).
\end{equation}

Let us recall the definition of the \emph{microsupport} $\MS(F)$ of an
object $F \in \Db(\bfk_X)$. 
Remark that we can define the mircosupport of an object over any commutative ring, which we will use below for $\bK=\bfk[\varepsilon]/(\varepsilon^2)$.

\begin{definition}[{\cite[Def.~5.1.2]{KS90}}]\label{definition:microsupport}
	Let $F \in \Db(\bfk_X)$ and $p \in T^*X$.
	One says that $p \not\in \MS(F)$ if there is a neighborhood $U$ of
	$p$ in $T^*X$ such that for any $x_0 \in X$ and any
	$C^\infty$-function $\varphi$ on $X$ (defined on a neighborhood of
	$x_0$) satisfying $d\varphi(x_0) \in U$, one has
	$\RG_{\{ x \in X \mid \varphi(x) \ge \varphi(x_0)\}}(F)_{x_0} \simeq 0$.
\end{definition}

In this paper, we will work in the triangulated orbit category $\Db_{/[1]}(X)$ for sheaves studied in \cite{Gu12,Gu19}, which was originally defined in Keller~\cite{Keller05}. 
Here we recall its definition and properties.  
See \cite{Gu12,Gu19} for more details. 

Let $\bK$ be the $\bfk$-algebra $\bfk[\varepsilon]/(\varepsilon^2)$
and $\perf(\bK_X)$ be the full triangulated subcategory
of $\Db(\bK_X)$ generated by the image of
the functor $\mathfrak{e}\colon \Db(\bfk_X)\to \Db(\bK_X), F\mapsto
\bK_X\otimes_{\bfk_X} F$.
We denote by $\Db_{/[1]}(X)$ or $\Db_{/[1]}(\bfk_X)$ the quotient
category $\Db(\bK_X)/\perf(\bK_X)$. 
For any $F\in \Db_{/[1]}(X)$, $F[1]$ is isomorphic to $F$. 
We also denote by $\fraki$ the composite functor $\fraki \colon \Db(\bfk_X)\to \Db(\bK_X)\to \Db_{/[1]}(X)$, where
the former functor is induced by the natural ring homomorphism $\bK \to \bfk$ corresponding to the trivial $\varepsilon$-action and the latter is the quotient functor. 
The Grothendieck's six operations are defined also on the orbit categories and commute with $\fraki$. 
Adjunctions, natural transformations and natural isomorphisms between composites of the operations exist as in the usual case. 
A cohomological functor $H^* \colon \Db_{/[1]}(X)\to \Module (\bfk_X)$ is defined so that $H^*(F)$ is the sheafification of the presheaf $(U \mapsto \Hom_{\Db_{/[1]}(U)}(\bfk_U, F|_U)) $ on $X$. 
This functor satisfies $H^*(\fraki (F))= \bigoplus_{n \in \bZ}H^n(F)$ for $F$ of $\Db(\bfk_X)$. 
The functor $H^*$ for $X=\pt$ gives an equivalence between $\Db_{/[1]}(\pt)$ and $\Module (\bfk)$. 
Note also that 
\begin{equation}
    \Hom_{\Db_{/[1]}(X)}(\fraki(F),\fraki(G)) \simeq \bigoplus_{n \in \bZ} \Hom_{\Db(\bfk_X)}(F,G[n]) \quad \text{for $F,G \in \Db(\bfk_X)$}.
\end{equation}

Guillermou~\cite{Gu12,Gu19} also introduced the microsupport $\MS(F) \subset T^*X$ of an object $F \in \Db_{/[1]}(X)$. 
We define $\MS(F) \coloneqq \bigcap_{F'\simeq F} \MS_{\bK}(F')$, where $F'$ runs over objects of $\Db(\bK_X)$ that are isomorphic to $F$ in $\Db_{/[1]}(X)$ and $\MS_{\bK}(F')$ denotes the usual microsupport of $F'$ as an object of $\Db(\bK_X)$. 
We also set $\rMS(F) \coloneqq \MS(F) \cap \rT X=\MS(F) \setminus T^*_XX$.
One can check the following properties.

\begin{proposition}\label{proposition:properties-ms}
    \begin{enumerate}
    	\item The microsupport of an object in $\Db_{/[1]}(X)$ is a conic (i.e., 
    	invariant under the action of $\bR_{>0}$ on $T^*X$) closed subset of
    	$T^*X$.
    	\item The microsupports satisfy the triangle inequality: if $F_1 \to
    	F_2 \to F_3 \toone$ is an exact triangle in $\Db_{/[1]}(X)$, then
    	$\MS(F_i) \subset \MS(F_j) \cup \MS(F_k)$ for $j \neq k$.
    	\item For $F\in \Db_{/[1]}(\bR^n)$ with $\MS (F) \subset T^*_{\bR^n}\bR^n$, there exists $L \in\Module(\bfk)$ such that $F\simeq L_{\bR^n}$. 
    \end{enumerate}
\end{proposition}

Using microsupports, we can microlocalize the category $\Db_{/[1]}(X)$ as follows.

\begin{definition}
    \begin{enumerate}
        \item For a subset $A$ of $\rT X$, one defines $\Db_{/[1],A}(X)$ as the triangulated subcategory of $\Db_{/[1]}(X)$ consisting of $F$ with $\rMS(F) \subset A$.
        \item For a subset $\Omega$ of $T^*X$, one defines $\Db_{/[1]}(X;\Omega)$ as the categorical localization of $\Db_{/[1]}(X)$ by $\Db_{/[1], \rT X \setminus \Omega}(X)$. 
        That is, $\Db_{/[1]}(X;\Omega) \coloneqq \Db_{/[1]}(X)/\Db_{/[1], \rT X \setminus \Omega}(X)$.
        \item For a subset $B$ of $\Omega$, $\Db_{/[1],B}(X;\Omega)$ denotes the full triangulated subcategory of $\Db_{/[1]}(X;\Omega)$ consisting of $F$ with $\rMS(F) \cap \Omega \subset B$.
        \item For a subset $B$ of $\rT X$, $\Db_{/[1], (B)}(X)$ denotes the full triangulated subcategory of $\Db_{/[1]}(X)$ consisting of $F$ for which there exists a neighborhood $U$ of $B$ such that $\rMS(F)\cap U \subset B$.
    \end{enumerate}
\end{definition}

The fact that $\Db_{/[1],A}(X)$ is a triangulated subcategory follows from the triangle inequality (\cref{proposition:properties-ms}(ii)).
Note that $\Db_{/[1],A}(X)$ contains locally constant sheaves on $X$.
Remark also that our notation is the same as in \cite{Gu12,Gu19} and slightly differs from that of \cite{KS90}.

The definition below is derived through the discussion with S.~Guillermou.

\begin{definition}\label{definition:microlocally-tame}
    Let $\mathcal{U}=\{U_\alpha \}_\alpha$ be an open covering of $X$ and $\cV=\{V_\alpha \}_\alpha$ be an open covering of an open subset $\Omega$ of $T^*X$. 
	\begin{enumerate}
	\item An object $F\in \Db_{/[1]}(X;\Omega)$ is said to be \emph{locally tame with respect to $\mathcal{U}$} if $F|_{U_\alpha}$ is isomorphic to some $\fraki(G_\alpha)$ as an object of $\Db_{/[1]}(U_\alpha;\Omega\cap T^*U_\alpha)$ for each $U_\alpha\in \mathcal{U}$.
	An object of $\Db_{/[1]}(X;\Omega)$ is said to be \emph{locally tame} if it is locally tame with respect to some open covering of $X$. 
	\item An object $F\in \Db_{/[1]}(X;\Omega)$ is said to be \emph{microlocally tame with respect to $\cV$} if $F$ is isomorphic to some $\fraki (G_\alpha)$ as an object of $\Db_{/[1]}(X;V_\alpha)$ for each $V_\alpha\in \cV$.
	An object of $\Db_{/[1]}(X;\Omega)$ is said to be \emph{microlocally tame} if it is microlocally tame with respect to some open covering of $\Omega$. 
	\item Let $B$ be a subset of $\rT X$.
	One defines $\BD_{/[1],(B)}^{\mathrm{b},\mathrm{mt}}(X)$ to be the full subcategory of $\Db_{/[1]}(X)$ consisting of $F$ for which there exists a neighborhood $\Omega$ of $B$ such that $\rMS(F) \cap \Omega \subset B$ and $F \in \Db_{/[1]}(X;\Omega)$ is microlocally tame. 
	\end{enumerate}
\end{definition}

\begin{remark}
    The reason why we introduce the notions of tame and microlocally tame objects is that we can deduce the statements in this section directly from the corresponding ones in \cite{KS90}.
\end{remark}

\subsection{Functorial operations}

We consider the behavior of the microsupports with respect to functorial operations.

\begin{proposition}[{cf.\ \cite[Prop.~5.4.4, 
		Prop.~5.4.13, and Prop.~5.4.5]{KS90}}]\label{proposition:SSpushpull}
	Let $f \colon X \to Y$ be a morphism of manifolds, $F \in \Db_{/[1]}(X)$,
	and $G \in \Db_{/[1]}(Y)$.
	\begin{enumerate}
		\item Assume that $f$ is proper on $\Supp(F)$.
		Then $\MS(f_*F) \subset f_\pi f_d^{-1}(\MS(F))$.
		\item Assume that $f$ is non-characteristic for $\MS(G)$ (see \cite[Def.~5.4.12]{KS90} for the definition).
		Then $\MS(f^{-1}G) \cup \MS(f^!G) \subset
		f_d f_\pi^{-1}(\MS(G))$.
		\item Assume that $f \colon X \to Y$ is a vector bundle over $Y$ and there exists an open covering $\cU$ of $Y$ such that $F$ is locally tame with respect to $f^{-1}\cU \coloneqq \{f^{-1}(U) \mid U\in\cU\}$. 
		Then $\MS(F) \subset f_d(X \times_{Y} T^*Y)$ if and only if the morphism $ f^{-1}Rf_*F\to F$ is an isomorphism. 
	\end{enumerate}
\end{proposition}

The following proposition is called the microlocal Morse lemma, which follows from \cref{proposition:properties-ms}(iii) and \cref{proposition:SSpushpull}(i). 

\begin{proposition}[{cf.\ \cite[Prop.~5.4.17]{KS90}}]\label{proposition:microlocalmorse}
	Let $F \in \Db_{/[1]}(X)$ and $\varphi \colon X \to \bR$ be a
	$C^\infty$-function.
    Moreover, let $a,b \in \bR$ with $a<b$.
	Assume
	\begin{enumerate}
		\renewcommand{\labelenumi}{$\mathrm{(\arabic{enumi})}$}
		\item $\varphi$ is proper on $\Supp(F)$,
		\item  $d\varphi(x) \not\in \MS(F)$ for any $x \in
		\varphi^{-1}([a,b))$.
	\end{enumerate}
	Then the canonical morphism
	\begin{equation}
	\RG(\varphi^{-1}((-\infty,b));F)
	\to
	\RG(\varphi^{-1}((-\infty,a));F)
	\end{equation}
	is an isomorphism.
\end{proposition}

For closed conic subsets $A$ and $B$ of $T^*X$, let us denote by $A+B$
the fiberwise sum of $A$ and $B$, that is,
\begin{equation}
    A+B
     \coloneqq 
    \left\{
        (x;a+b) \; \middle| \; 
        \begin{aligned}
            & x \in \pi(A) \cap \pi(B), \\
            & a \in A \cap \pi^{-1}(x), b \in B \cap \pi^{-1}(x) 
        \end{aligned}
    \right\} 
    \subset T^*X.
\end{equation}

\begin{proposition}[{cf.\ \cite[Prop.~5.4.14]{KS90}}]\label{proposition:SStenshom}
	Let $F, G \in \Db_{/[1]}(X)$.
	\begin{enumerate}
		\item If $\MS(F) \cap \MS(G)^a \subset T^*_XX$,
		then $\MS(F \otimes G) \subset \MS(F)+\MS(G)$.
		\item If $\MS(F) \cap \MS(G) \subset T^*_XX$,
		then $\MS(\cHom(F,G)) \subset \MS(F)^a+\MS(G)$.
	\end{enumerate}
\end{proposition}

Let $\varphi \colon X \to \bR$ be a
$C^\infty$-function and assume that $d\varphi(x) \neq 0$ for any $x
\in \varphi^{-1}(0)$.
Set $U \coloneqq \{x \in X \mid \varphi(x)<0\}$.
For such an open subset $U$ of $X$, we define
\begin{equation}
    N^*(U) \coloneqq \MS(\bfk_U)^a=T^*_XX|_U \cup
    \{(x;\lambda d\varphi(x)) \mid \varphi(x)=0, \lambda \le 0\}.
\end{equation}

\begin{lemma}\label{lemma:interiorboundary}
    Let $U$ be an open subset of $X$ as above, $j \colon U \to X$ be the open embedding, and $Z \coloneqq \overline{U}$. 
	Moreover, let $F\in \Db_{/[1]}(X)$ be a locally tame object. 
	\begin{enumerate}
		\item If $\Supp (F)\subset Z$ and $\MS(F)\cap N^*(U)\subset T^*_{X}X$,
		there exists a natural isomorphism $ F_U \simeq F$.
		\item If $\MS(F)\cap N^*(U)^a\subset T^*_{X}X$, there exists a natural isomorphism $F_Z\simeq Rj_*j^{-1}F$.
	\end{enumerate}
\end{lemma}

\begin{proof}
	\noindent(i) Consider the exact triangle $F_U \to F \to F_{\varphi^{-1}(0)}\toone $,
	where $\MS(F_U)\subset N^*(U)^a+\MS(F)$ and
	$(N^*(U)^a+\MS(F))\cap N^*(U)\subset T^*_XX$.
	By \cref{proposition:properties-ms}(ii), we have $\MS(F_{\varphi^{-1}(0)}) \cap N^*(U)\subset T^*_XX$.
	However, for any locally tame $G \in \Db_{/[1]}(X)$ supported on a closed submanifold $N$ of $X$, 
	the set $\MS(G)$ contains $T^*_NX|_{\Supp(G)}$, 
	which intersects with $N^*(U)$ outside 
	the zero-section unless $\Supp(G)=\emptyset$.
	Hence, we obtain $F_{\varphi^{-1}(0)} \simeq 0$.

	\noindent (ii) We obtain the morphism by applying $Ri_*$ to the morphism $i^{-1}F\to Rj'_*j'^{-1}i^{-1}F$, where $i\colon Z\to X$ and $j'\colon U\to Z$ are the inclusions. 
	The cone of $F_Z\to Rj_*j^{-1}F$ is supported on $\varphi^{-1}(0)$.
	By \cite[5.4.8]{KS90} and the locally tameness of $F$, $\MS (F_Z)\cup \MS (Rj_*j^{-1}F) \subset N^*(U)+\MS(F)$
	and hence the cone is $0$ as in (i).
\end{proof}

\begin{proposition}[{cf.\ \cite[Thm.~6.3.1]{KS90}}]\label{proposition:SSopenimm}
	Let $M$ be a manifold.
	Set $U \coloneqq M \times \bR_{<0}, X \coloneqq M\times \bR$ and denote by $j\colon U\to X$ the open embedding. 
	Moreover, let $G\in \Db_{/[1]}(U)$ and assume that there exists an open covering $\cV$ of the open subset $\{(x, u;\xi, \upsilon)\mid \xi \neq 0 \}$ of $T^*X=T^*(M \times \bR)$ such that
	\begin{enumerate}
        \renewcommand{\labelenumi}{$\mathrm{(\arabic{enumi})}$}
        \item $(x,u ;\xi, \upsilon+\lambda )\in V$ for any $V\in \cV$, $(x,u;\xi,\upsilon)\in V$ and $\lambda\in \bR$, 
        \item $G$ as an object of $\Db_{/[1]}(U;\{(x,u;\xi,\tau)\mid \xi\neq 0, u<0\})$ is microlocally tame with respect to $\{ V\cap T^*U \mid V\in\cV\}$.  
    \end{enumerate}
	Then, letting $\overline{\MS(G)}$ denote the closure of $\MS(G)$ in $T^*X$, one has the following.
	\begin{enumerate}
		\item If $\overline{\MS(G)}\cap N^*(U)^a\subset T^*_XX$,
		then $\MS(Rj_*G) \subset \overline{\MS(G)}+ N^*(U)$.
		\item If $\overline{\MS(G)}\cap N^*(U)\subset T^*_XX$,
		then $\MS(Rj_!G) \subset \overline{\MS(G)}+ N^*(U)^a$.
	\end{enumerate}
\end{proposition}

\subsection{Composition of sheaves}\label{subsection:composition}

We recall the operation called the composition of sheaves.

For $i=1,2,3$, let $X_i$ be a manifold.
We write $X_{ij} \coloneqq X_i \times X_j$ and $X_{123} \coloneqq X_1 \times X_2 \times X_3$ for short.
We denote by $q_{ij}$ the projection $X_{123} \to X_{ij}$.
Similarly, we denote by $p_{ij}$ the projection $T^*X_{123} \to T^*X_{ij}$.
We also denote by $p_{12^a}$ the composite of $p_{12}$ and the antipodal map on $T^*X_2$.

Let $A \subset T^*X_{12}$ and $B \subset T^*X_{23}$.
We set
\begin{equation}\label{equation:compset}
    A \circ B
     \coloneqq 
    p_{13}(p_{12^a}^{-1}A \cap p_{23}^{-1}B) \subset T^*X_{13}.
\end{equation}
We define the composition of sheaves as follows:
\begin{equation}
\begin{aligned}
    \underset{X_2}{\circ} \colon \Db_{/[1]}(X_{12}) \times \Db_{/[1]}(X_{23}) 
    & \to \Db_{/[1]}(X_{13}) \\
    (K_{12},K_{23}) 
    & \mapsto K_{12} \underset{X_2}{\circ} K_{23}
     \coloneqq 
    R{q_{13}}_!\,(q_{12}^{-1}K_{12}\otimes q_{23}^{-1}K_{23}).
\end{aligned}
\end{equation}
If there is no risk of confusion, we simply write $\circ$ instead of
$\underset{X_2}{\circ}$.
By \cref{proposition:SSpushpull}(i) and~(ii) and \cref{proposition:SStenshom}, we have the following.

\begin{proposition}\label{proposition:SScomp}
	Let $K_{ij} \in \Db_{/[1]}(X_{ij})$ and set $\Lambda_{ij} \coloneqq \MS(K_{ij}) \subset T^*X_{ij} \ (ij=12,23)$.
	Assume
	\begin{enumerate}
		\renewcommand{\labelenumi}{$\mathrm{(\arabic{enumi})}$}
		\item $q_{13}$ is proper on $q_{12}^{-1}\Supp(K_{12}) \cap q_{23}^{-1}\Supp(K_{23})$,
		\item $p_{12^a}^{-1}\Lambda_{12} \cap p_{23}^{-1}\Lambda_{23} \cap (T^*_{X_1}X_1 \times T^*X_2 \times T^*_{X_3}X_3) \subset T^*_{X_{123}}X_{123}$.
	\end{enumerate}
	Then
	\begin{equation}
	\MS(K_{12} \underset{X_2}{\circ} K_{23}) \subset
	\Lambda_{12} \circ \Lambda_{23}.
	\end{equation}
\end{proposition}

\subsection{\texorpdfstring{$\mu hom$}{mu hom} functor}

The bifunctor $\mu hom$ is originally defined in \cite[\S4.4]{KS90} for the usual derived category, and then generalized to the case of the triangulated orbit category by \cite{Gu12,Gu19} as $\mu hom \colon \Db_{/[1]}(X)^{\op} \times \Db_{/[1]}(X) \to \Db_{/[1]}(T^*X)$.
Here we recall its properties.

\begin{proposition}[{cf.\ \cite[Cor.~5.4.10 and
		Cor.~6.4.3]{KS90}}]\label{proposition:muhomss}
	Let $F,G \in \Db_{/[1]}(X)$ be microlocally tame.
	Then
	\begin{align}
	\Supp(\mu hom(F,G)) & \subset \MS(F) \cap \MS(G), \\
	\MS(\mu hom(F,G)) & \subset -\bfh^{-1}(C(\MS(G), \MS(F))).
	\end{align}
	(See \cref{subsection:geometric} for $C(S_1,S_2)$ and $\bfh \colon T^*T^*X \simto T T^*X$.) 
\end{proposition}

If $F,G \in \Db_{/[1]}(X;\Omega)$ are microlocally tame, $\mu hom(F,G)|_\Omega \in \Db_{/[1]}(\Omega)$ is locally tame. 

The functor $\mu hom$ gives an enrichment in $\Db_{/[1]}(T^*X)$ to $\Db_{/[1]}(X)$. 
For each $F\in \Db_{/[1]}(X)$, $\id_F \in \Hom(F,F) \simeq \Hom(\bfk_{T^*X},\mu hom(F,F))$ induces a morphism 
\begin{equation}\label{equation:idmu}
    \id_F^{\mu}\colon \bfk_{T^*X}\to \mu hom(F,F).
\end{equation}
For each $F,G,H \in \Db_{/[1]}(X)$, a composition morphism 
\begin{equation}
	\circ_{F,G,H}^\mu \colon \mu hom (G,H)\otimes \mu hom (F,G) \to \mu hom(F,H)
\end{equation}
is defined. 
This composition is unital and associative. 

For an open subset $\Omega$ of $T^*X$, the restriction of $\mu hom$ to $\Omega$ also gives an enrichment in $\Db_{/[1]}(\Omega)$ to $\Db_{/[1]}(X;\Omega)$. 

\begin{definition}\label{definition:cat-muhom}
    Let $\Omega$ be an open subset of $T^*X$.
    \begin{enumerate}
        \item For $F,G \in \Db_{/[1]}(X;\Omega)$, define $\Hom_{\Omega}^\mu(F,G) \coloneqq H^*\RG(\Omega;\mu hom(F,G))$.
        One also defines a new category $\Dmu_{/[1]}(X;\Omega)$ as follows:
        \begin{equation}
        \begin{aligned}
            \Ob(\Dmu_{/[1]}(X;\Omega)) & \coloneqq \Ob(\Db_{/[1]}(X;\Omega)), \\
            \Hom_{\Dmu_{/[1]}(X;\Omega)}(F,G) & \coloneqq \Hom^\mu_\Omega(F,G) \ \text{for $F,G \in \Ob( \Dmu_{/[1]}(X;\Omega) )$}.
        \end{aligned}
        \end{equation}
        \item For $F,G\in \Db_{/[1]}(X;\Omega)$, $m_{F,G}\colon\Hom_{\Db_{/[1]}(X;\Omega)}(F,G) \to \Hom^\mu_\Omega(F,G)$ denotes the natural map, which induces a functor from $\Db_{/[1]}(X;\Omega)$ to $\Dmu_{/[1]}(X;\Omega)$.
        \item For $v\in \Hom_{\Db_{/[1]}(X;\Omega)}(F,G)$, one denotes by $v^\mu$ the corresponding morphism $\bfk_\Omega\to \mu hom (F,G)|_\Omega$ in $\Db_{/[1]}(\Omega)$. 
    \end{enumerate}
\end{definition}

Note that the notation in (iii) is compatible with \eqref{equation:idmu}.
We also use the notation $\End_{\Omega}^{\mu}(F) \coloneqq \Hom_{\Omega}^\mu(F,F)$ for $F \in \Db_{/[1]}(X;\Omega)$.

\subsection{Simple sheaves}\label{subsection:simple}

Let $\Lambda$ be a locally closed conic Lagrangian submanifold of $\rT X$ and $p \in \Lambda$.
Simple sheaves along $\Lambda$ at $p$ are defined in \cite[Def.~7.5.4]{KS90}, which we recall below.
For a $C^\infty$-function $\varphi \colon X \to \bR$ such that $\varphi(\pi(p))=0$ and $\Gamma_{d\varphi} \coloneqq \{ (x;d\varphi(x)) \mid x \in X \}$ intersects $\Lambda$ transversally at $p$, one can define $\tau_{\varphi}=\tau_{p,\varphi} \in \bZ$ 
(see \cite[\S7.5 and A.3]{KS90}).

\begin{proposition}[{\cite[Prop.~7.5.3]{KS90}}]\label{proposition:maslovshft}
For $i=1,2$, let $\varphi_i \colon X \to \bR$ be a
$C^\infty$-function such that $\varphi_i(\pi(p))=0$ and
$\Gamma_{d\varphi_i}$ intersects $\Lambda$ transversally at $p$.
Let $F \in \Db(\bfk_X)$ and assume that $\MS(F) \subset \Lambda$ in a
neighborhood of $p$.
Then
\begin{equation}
\RG_{\{\varphi_1 \ge 0\}}(F)_{\pi(p)}
\simeq
\RG_{\{\varphi_2 \ge
0\}}(F)_{\pi(p)}\left[\tfrac{1}{2}(\tau_{\varphi_2}-\tau_{\varphi_1})\right].
\end{equation}
\end{proposition}

\begin{definition}[{\cite[Def.~7.5.4]{KS90}}]
Let $F \in \Db(\bfk_X)$ such that $\MS(F) \subset \Lambda$ in a neighborhood of $p$.
Then $F$ is said to be \emph{simple} if $\RG_{\{\varphi \ge 0\}}(F)_{\pi(p)} \simeq \bfk[d]$ for some $d \in \bZ$, 
for some (hence for any) $C^\infty$-function $\varphi$ such that $\varphi(\pi(p))=0$ and $\Gamma_{d\varphi}$ intersects $\Lambda$ transversally at $p$.
If $F$ is simple at all points of $\Lambda$, one says that $F$ is
\emph{simple along $\Lambda$}.
\end{definition}

\begin{definition}
	An object $F \in \Db_{/[1]}(\bfk_X)$ is said to be \emph{simple} along $\Lambda$ if for each $p\in \Lambda$ there exist an open neighborhood $U$ of $p$ in $\rT X$ and $G\in \Db(X;U)$ that is simple along $\Lambda \cap U$ such that $\fraki(G)\simeq F$ in $\Db_{/[1]}(X;U)$.
\end{definition}

One can prove that if $F \in \Db_{/[1]}(X)$ is simple along $\Lambda$, then
$\id_{F}^\mu |_{\Lambda} \colon \bfk_{\Lambda}\to \mu hom(F,F)|_{\Lambda}$ is an isomorphism.

\begin{lemma}[{cf.\ \cite[Lem.~6.14]{Gu12}}]\label{lemma:simpleclean-orbit}
    Let $\Lambda_1$ and $\Lambda_2$ be two conic Lagrangian submanifolds of $T^*X$ that intersect cleanly.
    For $i=1,2$, let $F_i \in \Db _{/[1], (\Lambda_i)}(X)$ be simple along $\Lambda_i$. 
    Assume that there exists an open neighborhood $\Omega_i$ of $\Lambda_i$ for $i=1,2$ and an open covering $\cU$ of $\Omega_1\cap \Omega_2$ such that
    \begin{enumerate}
        \renewcommand{\labelenumi}{$\mathrm{(\arabic{enumi})}$}
        \item each connected component of $\Lambda_1 \cap \Lambda_2$ is contained in some element of $\cU$,
        \item $F_1$ and $F_2$ are microlocally tame with respect to $\cU$ as objects of $\Db_{/[1]}(X;\Omega_1\cap \Omega_2)$. 
    \end{enumerate}
	Then $\mu hom(F_1,F_2)|_{\Lambda_1 \cap \Lambda_2} \simeq \bfk_{\Lambda_1 \cap \Lambda_2}$. 
\end{lemma}

\section{Tamarkin category and distance for sheaves}\label{section:Tamarkincat}

In this section, we introduce a modified version of Tamarkin category \cite{Tamarkin} and the translation distance \cite{AI20}. 
Since proofs for the results in this section are almost the same as those of \cite{GS14,AI20}, we omit the details here and discuss more precisely in \Cref{section:tamarkin-precise}.

From now on, until the end of this paper, let $M$ be a non-empty connected manifold
without boundary.
Recall also that $\bfk$ denotes the field $\bF_2=\bZ/2\bZ$.

\subsection{Definition of Tamarkin category}

In this subsection, we define a modified version of Tamarkin category $\cD^P(M)_\theta$, which is a crucial tool to give a better estimate of displacement energy.
For the original definition, see \cite{Tamarkin} (see also \cite{GS14}).
As mentioned in \cref{remark:intro-modification}, the modifications are the following threefold:
\begin{enumerate}
    \item replacing the additive variable space $\bR$ with $\Stheta=\bR/\theta \bZ$ for some $\theta \in \bR_{\ge 0}$,
    \item adding a parameter manifold $P$,
    \item using the triangulated orbit category instead of the usual derived category.
\end{enumerate}

Let $\theta\in\bR_{\geq 0}$ and set $\Stheta \coloneqq \bR /\theta \bZ$. 
Note that $\Stheta=\bR$ when $\theta=0$. 
We denote the image of $t\in \bR$ under the quotient map $\bR\to \Stheta$ by $[t]$ or simply $t$. 
Moreover, let $P$ be a manifold.
Denote by $(x;\xi)$ a local homogeneous coordinate system on $T^*M$,
by $(y;\eta)$ that on $T^*P$, and by $(t;\tau)$ the homogeneous
coordinate system on $T^*\Stheta$ and $T^*\bR$.
We define maps $\tilde{q}_1,\tilde{q}_2,s_\theta \colon M \times P \times \Stheta \times \Stheta \to M \times P \times \Stheta$ by
\begin{equation}
\begin{aligned}
\tilde{q}_1(x,y,t_1,t_2) & = (x,y,t_1), \\
\tilde{q}_2(x,y,t_1,t_2) & = (x,y,t_2), \\
s_\theta(x,y,t_1,t_2) & =(x,y,t_1+t_2).
\end{aligned}
\end{equation}
If there is no risk of confusion, we simply write $s$ for $s_\theta$.
We also set
\begin{equation}
\begin{aligned}
i \colon M \times P \times \Stheta \to M \times P \times \Stheta, \ (x,y,t)
\longmapsto (x,y,-t),\\
\ell\colon M \times P \times \bR \to M \times P \times \Stheta, \ (x,y,t)
\longmapsto (x,y,[t]).
\end{aligned}
\end{equation}

Note also that if $\theta=0$ then $\ell$ is the identity map.
We also write $\ell \colon \bR \to \Stheta, t \mapsto [t]$ by abuse of notation.

\begin{definition}
	For $F,G \in \Db_{/[1]}(M \times P \times \Stheta)$, one sets
	\begin{align}
	F \star G &  \coloneqq Rs_!(\tilde{q}_1^{-1}F \otimes \tilde{q}_2^{-1}G), \\
	\cHom^\star(F,G) &  \coloneqq R\tilde{q}_{1*} \cRHom(\tilde{q}_2^{-1}F,s^!G)
	\\
	& \ \simeq Rs_*\cRHom(\tilde{q}_2^{-1}i^{-1}F,\tilde{q}_1^!G). \notag
	\end{align}
\end{definition}

Note that the functor $\star$ is a left adjoint to $\cHom^\star$.

For a manifold $N$, we set $\Omega_+(N)_\theta \coloneqq T^*N \times \{ (t;\tau) \mid \tau >0\} \subset T^*(N \times \Stheta)$ and write $\Omega_{+} \coloneqq \Omega_+(M\times P)_\theta$ for short.
We define the map
\begin{equation}
\begin{aligned}
\xymatrix@R=10pt{
	\rho \colon \Omega_+ \ar[r] & T^*M \\
	\ (x,y,t;\xi,\eta,\tau) \ar@{|->}[r] \ar@{}[u]|-{\hspace{7pt}
		\rotatebox{90}{$\in$}} & (x;\xi/\tau).
	\ar@{}[u]|-{\rotatebox{90}{$\in$}}
}
\end{aligned}
\end{equation}
We also define an endofunctor $P_l$ of $\Db_{/[1]}(M \times P \times \Stheta)$ by $P_l \coloneqq R\ell_! \bfk_{M \times P \times [0,+\infty)} \star (-)$, which induces the equivalence of categories
\begin{equation}
	P_l \colon \Db_{/[1]}(M \times P \times \Stheta;\Omega_+) 
	\simto {}^\perp \Db_{/[1], \{\tau \le 	0\}}(M \times P \times \Stheta),
\end{equation}
where ${}^\perp (-)$ denotes the left orthogonal.

\begin{definition}
	One defines a category $\cD^P(M)_\theta$ by
	\begin{equation}
	\cD^P(M)_\theta
	 \coloneqq 
	\Db_{/[1]}(M \times P \times \Stheta;\Omega_{+})
	\end{equation}
	and identifies it with the left orthogonal ${}^\perp \Db_{/[1], \{\tau \le 0\}}(M \times P \times \Stheta)$.
	One also sets $\MS_+(F) \coloneqq \MS(F)\cap \Omega_+$ for $F\in \cD^P(M)_\theta$. 
	For a compact subset $A$ of $T^*M$, one defines a full subcategory
	$\cD^P_A(M)_\theta$ by
	\begin{equation}
	\cD^P_A(M)_\theta
	 \coloneqq 
	\Db_{/[1], \rho^{-1}(A) }(M \times P \times \Stheta;\Omega_{+}).
	\end{equation}
    If $P=\pt$, we omit $P$ from the above notation. 
\end{definition}

The bifunctor $\cHom^\star$ induces an internal Hom functor
$\cHom^\star \colon \cD^P(M)^{\op}_\theta \times \cD^P(M)_\theta \to
\cD^P(M)_\theta$ (see \cite{AI20} for the details).
An argument similar to \cite{GS14} proves the following.

\begin{proposition}[{cf.\ \cite[Lem.~4.18]{GS14}}]\label{proposition:morD}
    Let $q \colon M \times P \times \Stheta \to \Stheta$ be the projection.
	For $F,G \in \Db_{/[1]}(M \times P \times \Stheta)$, there are isomorphisms
	\begin{equation}
	\begin{aligned}
    	\Hom_{\cD^P(M)_\theta}(Q(F),Q(G)) 
    	&\simeq \Hom_{\Db_{/[1]}(M \times P \times \Stheta)}(P_l(F),G)\\
    	&\simeq
    	H^* \RG_{[0,+\infty)}(\bR;\ell^! Rq_* \cHom^\star(F,G)),
    \end{aligned}
	\end{equation}
	where $Q \colon \Db_{/[1]}(M \times P \times \Stheta) \to \cD^P(M)_\theta$ is the quotient functor.
	In particular, if $F \in {}^\perp \Db_{/[1], \{\tau \le 	0\}}(M \times P \times \Stheta)$, one has isomorphisms 
	\begin{equation}
	\begin{aligned}
    	\Hom_{\cD^P(M)_\theta}(Q(F),Q(G)) &\simeq \Hom_{\Db_{/[1]}(M \times P \times \Stheta)}(F,G)\\
    	&\simeq
    	H^* \RG_{[0,+\infty)}(\bR;\ell^! Rq_* \cHom^\star(F,G)).
    \end{aligned}
	\end{equation}
\end{proposition}

\subsection{Distance and stability with respect to Hamiltonian deformation}

We introduce a distance on $\cD^P(M)_\theta$ following~\cite{AI20}.
For $c \in \bR$, we define the translation map
\begin{equation}
    T_c \colon M \times P \times \Stheta \to M \times P \times \Stheta, (x,y,t) \mapsto (x,y,t+c).
\end{equation}
For $F \in \cD^P(M)_\theta$ and $c, d\in \bR$ with $c \le d$, there is a canonical morphism $\tau_{c,d}(F) \colon {T_c}_*F \to {T_{d}}_*F$ (see \Cref{subsec:appendix-distance}).
Using the morphism, we define the translation distance $d_{\cD^P(M)_\theta}$
as in \cite{AI20}.

\begin{definition}[{cf.\ \cite[Def.~4.4]{AI20}}]\label{definition:distance}
	 Let $F,G \in \cD^P(M)_\theta$.
	\begin{enumerate}
		\item Let $a,b \in \bR_{\ge 0}$.
		The pair $(F,G)$ is said to be \emph{$(a,b)$-interleaved}
		if there exist morphisms $\alpha, \delta \colon F \to {T_a}_*G$ and
		$\beta, \gamma \colon G \to {T_b}_*F$ satisfying the following conditions:
		\begin{enumerate}
			\renewcommand{\labelenumii}{$\mathrm{(\arabic{enumii})}$}
			\item $F \xrightarrow{\alpha} {T_a}_* G \xrightarrow{{T_a}_*\beta} {T_{a+b}}_*F$
			is equal to $\tau_{0,a+b}(F) \colon F \to {T_{a+b}}_*F$ and 
			\item $G \xrightarrow{\gamma} {T_b}_* F \xrightarrow{{T_b}_*\delta} {T_{a+b}}_*G$
			is equal to $\tau_{0,a+b}(G) \colon G \to {T_{a+b}}_*G$.
		\end{enumerate}
		\item One defines
		\begin{equation}
		d_{\cD^P(M)_\theta}(F,G)
		 \coloneqq 
		\inf
		\{
		a+b \in \bR_{\ge 0}
		\mid
		a,b \in \bR_{\ge 0},
		\text{$(F,G)$ is $(a,b)$-interleaved}
		\},
		\end{equation}
		and calls $d_{\cD^P(M)_\theta}$ the \emph{translation distance}.
	\end{enumerate}
\end{definition}

Now we consider Hamiltonian deformations of sheaves.
Let $I$ be an open interval containing the closed interval $[0,1]$.
Let $H \colon T^*M \times I \to \bR$ be a compactly supported
Hamiltonian function and denote by $\phi^H \colon T^*M \times I \to
T^*M$ the Hamiltonian isotopy generated by $H$.
We set
\begin{equation}
    \| H \|
     \coloneqq 
    \int_0^1 \left(\max_p H_s(p) - \min_p H_s(p) \right) ds.
\end{equation}
Moreover, let $K^H \in \Db(M \times \Stheta \times M \times \Stheta \times I)$ be the
sheaf quantization associated with $\phi^H$, whose existence was proved by Guillermou--Kashiwara--Schapira~\cite{GKS} (see also \Cref{subsection:appendix-sq-hamiltonian}).
For $s \in I$, we set $K^H_s \coloneqq K^H|_{M \times \Stheta \times M \times \Stheta \times \{s\}}$.
Then the composition with $K^H_s$ induces a
functor
\begin{equation}\label{eq:hamiltonian-induced-functor}
    \Phi^H_s \coloneqq K^H_s \circ (-) \colon
    \cD^P(M)_\theta \to \cD^P(M)_\theta,
\end{equation}
which restricts to $\cD^P_A(M) \to \cD^P_{\phi^H_s(A)}(M)$ for any
compact subset $A$ of $T^*M$, by \cref{proposition:SScomp}.
The following proposition is one of the main results of~\cite{AI20}, which is a stability result of the translation distance with respect to Hamiltonian deformation. 
For the outline of the proof, see \Cref{subsec:appendix-distance}. 

\begin{proposition}[{cf.\ \cite[Thm.~4.16]{AI20}}]\label{proposition:distance}
	Let $\phi^H \colon T^*M \times I \to T^*M$ be the Hamiltonian isotopy
	generated by a compactly supported Hamiltonian function $H \colon T^*M
	\times I \to \bR$ and denote by $\Phi^H_1 \colon \cD^P(M)_\theta \to \cD^P(M)_\theta$ the functor associated with $\phi^H_1$.
	Then for $F \in \cD^P(M)_\theta$, one has an inequality $d_{\cD^P(M)_\theta}(F,\Phi^H_1(F)) \le \| H \|$.
\end{proposition}

As explained in the introduction, one can obtain a sheaf-theoretic bound for the displacement energy of two compact subsets, using the proposition above (see \cref{proposition:energyestimate}).

\section{Sheaf quantization of rational Lagrangian immersions}
\label{section:constrution}

In this section, we prove the existence of sheaf quantizations of a certain class of Lagrangian immersions, following the idea of Guillermou~\cite{Gu12, Gu19}.

\subsection{Definitions and statement of the existence result}

First we introduce some notions for Lagrangian immersions.
We assume that $L$ is a compact connected manifold.

\begin{definition}\label{definition:strongly-rational}
	\begin{enumerate}
		\item A Lagrangian immersion $\iota\colon L\to T^*M$
		is said to be \emph{strongly rational} if there exists a non-negative number
		$\theta(\iota) \in \bR_{\geq 0}$ such that the image of the pairing map 
		$\langle \iota^*\alpha, - \rangle \colon H_1(L;\bZ)\to \bR, \gamma \mapsto \int_\gamma \iota^*\alpha$ is $\theta(\iota) \cdot \bZ$.
		We call $\theta(\iota)$ the \emph{period} of $\iota$.
		\item For a strongly rational Lagrangian immersion
		$\iota \colon L \to T^*M$, one defines
		\begin{equation}
		r(\iota)
		 \coloneqq 
		\inf
		\left( \left\{ \int_{l} \iota^*\alpha
		\; \middle| \;
		\begin{aligned}
		& l \colon [0,1] \to L, \\
		& \iota \circ l(0)= \iota \circ l(1)
		\end{aligned}
		\right\}
		\cap \bR_{>0}\right).
		\end{equation}
	\end{enumerate}
\end{definition}
Note that the infimum of the empty set is defined to be $+\infty$.

\begin{notation}\label{notation:conification}
	Let $\iota \colon L \to T^*M$ be a compact strongly rational Lagrangian
	immersion with period $\theta=\theta(\iota)$ and $f \colon L \to \Stheta$
	be a function satisfying $\iota^*\alpha=df$.
	One defines a conic Lagrangian immersion $\wh{\iota}_f$ by 
	\begin{equation}
	\wh{\iota} \coloneqq \wh{\iota}_f\colon L\times \bR_{>0} \to T^*(M\times \Stheta), (y,\tau)\mapsto (\tau \iota (y), (-f(y);\tau)) 
	\end{equation}
	and sets
	\begin{equation}\label{equation:def-conification-lambda}
    	\Lambda
    	 =
    	\Lambda_{\iota,f}
    	 \coloneqq 
    	\left\{ 
    	    (x,t;\xi,\tau) \in T^*(M \times \Stheta) 
    	    \; \middle| \; 
    	    \begin{aligned}
        	    & \text{$\tau >0$, there exists $y \in L$}, \\
        	    & (x;\xi/\tau)=\iota(y), t=-f(y) 
    	    \end{aligned}
    	\right\}. 
	\end{equation}
	One also sets 
	\begin{align*}
	\Lambda_q
	& =
	\Lambda_{\iota,f,q}
	 \coloneqq 
	\{ (x,u,t;\xi,0,\tau) 
	\mid (x,t;\xi,\tau) \in \Lambda_{\iota,f} \}, \\
	\Lambda_r
	& =
	\Lambda_{\iota,f,r}
	 \coloneqq 
	\{ (x,u,t;\xi,-\tau,\tau) \mid (x,t-u;\xi,\tau) \in \Lambda_{\iota,f} \}.
	\end{align*}
\end{notation}

Moreover, we make the following assumption.

\begin{assumption}\label{assumption:embedding}
	There exists no curve $l\colon [0,1]\to L$ with $l(0)\neq l(1)$, $\iota \circ l(0)= \iota \circ l(1)$, and $\int_{l} \iota^*\alpha=0$. 
\end{assumption}

Under \cref{assumption:embedding}, the conic Lagrangian immersion $\wh{\iota}$ is an embedding and we identify it with the conic Lagrangian submanifold $\Lambda$.
Without this assumption \cref{proposition:muhom-self,proposition:EndG} do not hold in general.
Moreover, thanks to the assumption, we can apply the method in \cite{Gu19} for the construction of a sheaf quantization.

The following is the existence result of a sheaf quantization of a strongly rational Lagrangian immersion, which we will prove in the next subsection.

\begin{theorem}\label{theorem:existence-quantization}
	Let $\iota \colon L \to T^*M$ be a compact strongly rational Lagrangian
	immersion with period $\theta=\theta(\iota)$ satisfying \cref{assumption:embedding}. 
	Take a function $f \colon L \to \Stheta$ and define $\Lambda, \Lambda_q, \Lambda_r$ as in \cref{notation:conification}.
	Then for each $a\in (0,r(\iota))$, there exists an object $G_{(0,a)} \in
	{}^\perp \Db_{/[1], \{\tau \le 	0\}}(M \times (0,a) \times \Stheta) \simeq \cD^{(0,a)}(M)_\theta$ satisfying the following conditions:
	\begin{enumerate}
		\renewcommand{\labelenumi}{$\mathrm{(\arabic{enumi})}$}
		\item $\rMS(G_{(0,a)})\subset \left( \Lambda_{q} \cup
		\Lambda_{r}\right) \cap T^*(M\times (0,a)\times \Stheta)$,
    	\item  $G_{(0,a)}$ is simple along $ \Lambda_{q} \cap T^*(M\times (0,a)\times \Stheta)$, 
		\item $F_0 \coloneqq (Rj_*G_{(0,a)})|_{M \times
			\{0\} \times \Stheta }$ is isomorphic to $0$, where $j$ is the inclusion $M \times (0,a) \times
		\Stheta \to M \times \bR \times \Stheta $, 
		\item there is an open covering $\{V_\alpha\}_\alpha$ of $\Omega_+(M)_\theta$ such that $G_{(0,a)}$ is microlocally tame with respect to $\{V_\alpha\times T^*(0,a) \}_\alpha$. 
	\end{enumerate}
	Moreover, the object $G_{(0,a)}$ automatically satisfies $d_{\cD^{(0,a)}(M)_\theta}(G_{(0,a)},0)\leq a$. 
\end{theorem}

For the reason why we take $a \in (0,r(\iota))$, see \cref{remark:construction-r}.

\begin{example}[An object associated with the unit circle]
The embedding $\iota\colon \bR/2\pi\bZ\to T^*\bR, s\mapsto (\cos s,\sin s)$ is strongly rational with $\theta(\iota)=r(\iota)=\pi$. 
The function $f\colon\bR/2\pi\bZ\to S^1_{\pi}, s\mapsto \frac{1}{2}s-\frac{1}{4}\sin 2s$ satisfies $df=\iota^*\alpha$. 
Applying \cref{theorem:existence-quantization} to these $\iota$ and $f$, 
we obtain $G_{(0,a)} \in{}^\perp \Db_{/[1], \{\tau \le 	0\}}(\bR \times (0,a) \times S^1_{\pi})$ for each $a\in (0,\pi)$. 
The boundary set of the support of $G_{(0,a)}$ is 
\begin{equation}
\begin{aligned}
    \pi_{\bR\times (0,a)\times S^1_{\pi}}(\rMS(G_{(0,a)}))
    = \ &\{(\cos s, u, -f(s))\mid s\in \bR/2\pi\bZ , u\in(0,a)\}\\
    &\cup\{(\cos s, u, -f(s)+u)\mid s\in \bR/2\pi\bZ , u\in(0,a)\}\\
    \subset \ &\bR\times (0,a)\times S^1_{\pi}.
\end{aligned}
\end{equation}
The support of $G_{(0,a)}$ is the closure of 
the bounded regions enclosed by the boundary set with respect to the standard metric. 
The support of $G_{(0,a)}$ can also be  written as 
\begin{equation}
\begin{aligned}
    & \Supp (G_{(0,a)}) \\
    & = \{(x, u, t)\in\bR\times (0,a)\times S^1_{\pi}\mid (x-\cos g(-t))(x-\cos g(-t+u))\le 0\},
\end{aligned}
\end{equation}
where 
$g\colon S^1_{\pi}\to \bR/2\pi\bZ$ is the inverse of $f$. 
See \cref{fig:sq_example}.

At each interior point of $\Supp (G_{(0,a)})$, the stalk of $G_{(0,a)}$ is isomorphic to $\bfk$. 
At each boundary point of $\Supp (G_{(0,a)})$, the stalk of $G_{(0,a)}$ is isomorphic to $\bfk$ or $0$. 
Indeed, for any point $(x,u,t)\in\bR\times (0,a)\times S^1_{\pi}$, the stalk of $G_{(0,a)}$ at $(x,u,t)$ is isomorphic to $\bfk$ if and only if $(x,u,t+\varepsilon)\in \Supp (G_{(0,a)})$ for sufficiently small $\varepsilon >0$. 

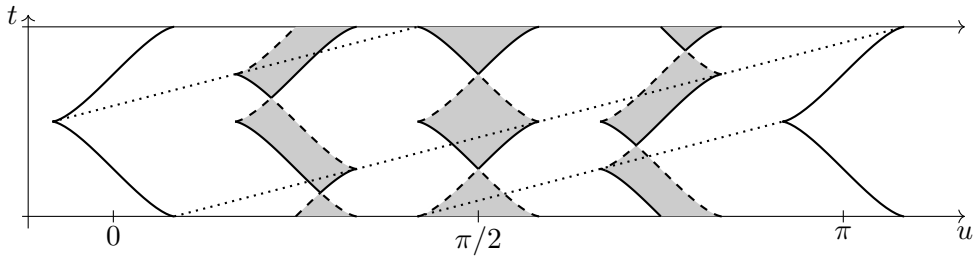
\begin{figure}[ht]
    \centering
    \begin{tikzpicture}[scale=0.8]
    \draw[->] (-1.5,0) -- (14,0) node[below] {$u$};
    \draw[->] (-1.5,pi) -- (14,pi);
    \draw[->] (-1.4,-0.2) -- (-1.4, pi+0.2) node[left] {$t$};
    
    \draw[thick,-] plot[domain=0:2*pi, smooth] ({cos(\x r)},{0.5* \x -0.25 * sin(2*\x r)});

    \path[fill=gray!40] plot[domain=1.98665-0.5*pi:1.98665+0.5*pi] ({cos((\x - 0.5*pi) r)+3},{0.5* \x -0.25 * sin(2* (\x - 0.5*pi) r ))}) plot[domain=1.15494+pi:1.15494] ({cos(\x r)+3},{0.5* \x -0.25 * sin(2*\x r)});
    \path[fill=gray!40] plot (4,0)--(3,0)-- plot[domain=0:1.98665-0.5*pi] ({cos((\x - 0.5*pi) r)+3},{0.5* \x -0.25 * sin(2* (\x - 0.5*pi) r ))}) plot[domain=1.15494:0] ({cos(\x r)+3},{0.5* \x -0.25 * sin(2*\x r)})--cycle;
    \path[fill=gray!40] plot (4,pi)--(3,pi)-- plot[domain=2*pi:1.98665+0.5*pi] ({cos((\x - 0.5*pi) r)+3},{0.5* \x -0.25 * sin(2* (\x - 0.5*pi) r ))}) plot[domain=1.15494+pi:2*pi] ({cos(\x r)+3},{0.5* \x -0.25 * sin(2*\x r)})--cycle;

    \draw[domain=0:1.15494, thick, dashed] plot ({cos(\x r)+3},{0.5* \x -0.25 * sin(2*\x r)});
    \draw[domain=1.15494:pi, thick] plot ({cos(\x r)+3},{0.5* \x -0.25 * sin(2*\x r)});
    \draw[domain=pi:1.15494+pi, thick, dashed] plot ({cos(\x r)+3},{0.5* \x -0.25 * sin(2*\x r)});
    \draw[domain=1.15494+pi:2*pi, thick] plot ({cos(\x r)+3},{0.5* \x -0.25 * sin(2*\x r)});
    
    \draw[domain=0:1.98665-0.5*pi, thick, dashed] plot ({cos((\x - 0.5*pi) r)+3},{0.5* \x -0.25 * sin(2* (\x - 0.5*pi) r ))});
    \draw[domain=1.98665-0.5*pi:0.5*pi, thick] plot ({cos((\x - 0.5*pi) r)+3},{0.5* \x -0.25 * sin(2* (\x - 0.5*pi) r ))});
    \draw[domain=0.5*pi:1.98665+0.5*pi, thick, dashed] plot ({cos((\x - 0.5*pi) r)+3},{0.5* \x -0.25 * sin(2* (\x - 0.5*pi) r ))});
    \draw[domain=1.98665+0.5*pi:1.5*pi, thick] plot ({cos((\x - 0.5*pi) r)+3},{0.5* \x -0.25 * sin(2* (\x - 0.5*pi) r ))});
    \draw[domain=1.5*pi:2*pi, thick, dashed] plot ({cos((\x - 0.5*pi) r)+3},{0.5* \x -0.25 * sin(2* (\x - 0.5*pi) r ))});

    \path[fill=gray!40] plot (5,0)--(7,0)-- plot[domain=0:0.5*pi] ({cos(\x r)+6},{0.5* \x -0.25 * sin(2*\x r)}) plot[domain=0.5*pi:0] ({cos((\x -pi) r)+6},{0.5* \x -0.25 * sin(2*(\x -pi) r)})--cycle;
    \path[fill=gray!40] plot[domain=0.5*pi:1.5*pi] ({cos(\x r)+6},{0.5* \x -0.25 * sin(2*\x r)}) plot[domain=1.5*pi:0.5*pi] ({cos((\x -pi) r)+6},{0.5* \x -0.25 * sin(2*(\x -pi) r)});
    \path[fill=gray!40] plot (5,pi)--(7,pi)-- plot[domain=2*pi:1.5*pi] ({cos(\x r)+6},{0.5* \x -0.25 * sin(2*\x r)}) plot[domain=1.5*pi:2*pi] ({cos((\x -pi) r)+6},{0.5* \x -0.25 * sin(2*(\x -pi) r)})--cycle;
    
    \draw[thick,dashed] plot[domain=0:0.5*pi, smooth] ({cos(\x r)+6},{0.5* \x -0.25 * sin(2*\x r)});
    \draw[thick,-] plot[domain=0.5*pi:pi, smooth] ({cos(\x r)+6},{0.5* \x -0.25 * sin(2*\x r)});
    \draw[thick,dashed] plot[domain=pi:1.5*pi, smooth] ({cos(\x r)+6},{0.5* \x -0.25 * sin(2*\x r)});
    \draw[thick,-] plot[domain=1.5*pi:2*pi, smooth] ({cos(\x r)+6},{0.5* \x -0.25 * sin(2*\x r)});
    \draw[thick,dashed] plot[domain=0:0.5*pi, smooth] ({cos((\x -pi) r)+6},{0.5* \x -0.25 * sin(2*(\x -pi) r)});
    \draw[thick,-] plot[domain=0.5*pi:pi, smooth] ({cos((\x -pi) r)+6},{0.5* \x -0.25 * sin(2*(\x -pi) r)});
    \draw[thick,dashed] plot[domain=pi:1.5*pi, smooth] ({cos((\x -pi) r)+6},{0.5* \x -0.25 * sin(2*(\x -pi) r)});
    \draw[thick,-] plot[domain=1.5*pi:2*pi, smooth] ({cos((\x -pi) r)+6},{0.5* \x -0.25 * sin(2*(\x -pi) r)});

    \path[fill=gray!40] plot (9,0)--(10,0)-- plot[domain=0:1.98665] ({cos(\x r)+9},{0.5* \x -0.25 * sin(2*\x r)}) [domain=0.5*pi:0] plot[domain=1.15494+0.5*pi:0] ({cos((\x -1.5*pi) r)+9},{0.5* \x -0.25 * sin(2*(\x -1.5*pi) r)})--cycle;
    \path[fill=gray!40] plot[domain=1.98665:1.98665+pi] ({cos(\x r)+9},{0.5* \x -0.25 * sin(2*\x r)}) plot[domain=1.15494+1.5*pi:1.15494+0.5*pi] ({cos((\x -1.5*pi) r)+9},{0.5* \x -0.25 * sin(2*(\x -1.5*pi) r)})--cycle;
    \path[fill=gray!40] plot(9,pi)--(10,pi)-- plot[domain=2*pi:1.98665+pi] ({cos(\x r)+9},{0.5* \x -0.25 * sin(2*\x r)}) plot[domain=1.15494+1.5*pi:2*pi] ({cos((\x -1.5*pi) r)+9},{0.5* \x -0.25 * sin(2*(\x -1.5*pi) r)})--cycle;
    
    \draw[thick,dashed] plot[domain=0:1.98665, smooth] ({cos(\x r)+9},{0.5* \x -0.25 * sin(2*\x r)});
    \draw[thick,-] plot[domain=1.98665:pi, smooth] ({cos(\x r)+9},{0.5* \x -0.25 * sin(2*\x r)});
    \draw[thick,dashed] plot[domain=pi:1.98665+pi, smooth] ({cos(\x r)+9},{0.5* \x -0.25 * sin(2*\x r)});
    \draw[thick,-] plot[domain=1.98665+pi:2*pi, smooth] ({cos(\x r)+9},{0.5* \x -0.25 * sin(2*\x r)});
    \draw[thick,-] plot[domain=0:1.15494, smooth] ({cos((\x -1.5*pi) r)+9},{0.5* \x -0.25 * sin(2*(\x -1.5*pi) r)});
    \draw[thick,dashed] plot[domain=1.15494:1.15494+0.5*pi, smooth] ({cos((\x -1.5*pi) r)+9},{0.5* \x -0.25 * sin(2*(\x -1.5*pi) r)});
    \draw[thick,-] plot[domain=1.15494+0.5*pi:1.5*pi, smooth] ({cos((\x -1.5*pi) r)+9},{0.5* \x -0.25 * sin(2*(\x -1.5*pi) r)});
    \draw[thick,dashed] plot[domain=1.5*pi:1.15494+1.5*pi, smooth] ({cos((\x -1.5*pi) r)+9},{0.5* \x -0.25 * sin(2*(\x -1.5*pi) r)});
    \draw[thick,-] plot[domain=1.15494+1.5*pi:2*pi, smooth] ({cos((\x -1.5*pi) r)+9},{0.5* \x -0.25 * sin(2*(\x -1.5*pi) r)});

    \draw[thick,-] plot[domain=0:2*pi, variable=\x, smooth] ({cos(\x r)+12},{0.5* \x -0.25 * sin(2*\x r)});

    \draw[thick, dotted] (-1,0.5*pi) -- (-1+6,pi);
    \draw[thick, dotted] (-1+6,0) -- (-1+12,0.5*pi);
    \draw[thick, dotted] (1,0) -- (13,pi);
    
    \draw (0,0.1)--(0,-0.1);
    \node[below] at (0,0) {$0$};
    \draw (6,0.1)--(6,-0.1);
    \node[below] at (6,0) {$\pi/2$};
    \draw (12,0.1)--(12,-0.1);
    \node[below] at (12,0) {$\pi$};

    \end{tikzpicture}
    \caption{$G_{(0,a)}$ associated with the unit circle}
    \label{fig:sq_example}
\end{figure}

\end{example}

\begin{remark}\label{remark:legendrian}
    For any Legendrian submanifold $\cL$ of $ST^*(M \times \Stheta)$, there exist $n\in \bZ_{\ge 0}$, a strongly rational immersion $\iota \colon L \to T^*M$ with period $n\theta$ satisfying \cref{assumption:embedding}, and a primitive function $f\colon L\to S^1_{n\theta}$ such that $p_n(\Lambda_{\iota,f})/\bR_{>0}=\cL$, 
    where $p_n\colon \rT(M \times S^1_{n\theta})\to \rT(M \times \Stheta)$ is the natural covering map. 
    Thus, \cref{theorem:existence-quantization} gives a sheaf quantization object for a Legendrian submanifold of $ST^*(M \times \Stheta)$.
\end{remark}

\subsection{Construction}\label{subsec:construction}

In this subsection, we prove \cref{theorem:existence-quantization} following the idea of \cite[Thm.~12.2.2 and \S12.3]{Gu19}. 
We prepare some notions introduced by \cite{Gu19}. 

\subsubsection{Kashiwara--Schapira stack}

First we introduce Kashiwara--Schapira stack \cite[Part~10]{Gu19}.
This is an important tool to construct the sheaf quantization $G_{(0,a)}$.

\begin{definition}
	Let $\Lambda$ be a locally closed conic Lagrangian submanifold of $T^*X$.
	\begin{enumerate}
	\item For $V\subset \Lambda$, one defines a category $\muSh_{/[1],\Lambda}^{\mathrm{mt},0}(V)$ as follows:
	\begin{equation}
	\begin{aligned}
	    \Ob\left( \muSh_{/[1],\Lambda}^{\mathrm{mt},0}(V) \right) & \coloneqq \Ob\left( \BD_{/[1],(V)}^{\mathrm{b},\mathrm{mt}}(X) \right), \\
	    \Hom_{\muSh_{/[1],\Lambda}^{\mathrm{mt},0}(V)}(F,G)
	    & \coloneqq \Hom_{\Db_{/[1]}(X;V)}(F,G) 
	    \\
	    & \text{for $F,G \in \Ob\left( \muSh_{/[1],\Lambda}^{\mathrm{mt},0}(V) \right)$}.
	\end{aligned}
	\end{equation}
	The correspondence $V \mapsto \muSh_{/[1],\Lambda}^{\mathrm{mt},0}(V)$ defines a prestack $\muSh_{/[1],\Lambda}^{\mathrm{mt},0}$ on $\Lambda$. 
	
	\item One defines the \emph{Kashiwara--Schapira stack} $\muSh_{/[1],\Lambda}^{\mathrm{mt}}$ on $\Lambda$ as the associated stack with $\muSh_{/[1], \Lambda}^{\mathrm{mt},0}$. 
    The quotient functor gives a functor $\frakm_{\Lambda} \colon 	\BD_{/[1],(\Lambda)}^{\mathrm{b},\mathrm{mt}}(X) \to \muSh_{/[1], \Lambda}^{\mathrm{mt}}(\Lambda)$ (see also \cref{definition:microlocally-tame})
    
	\item An object $\cF \in \muSh_{/[1], \Lambda}^{\mathrm{mt}}(V)$ is said to be \emph{simple} if $\cF$ is obtained by gluing simple objects. 
	\end{enumerate}
\end{definition}

\begin{remark}\label{remark:+mt}
	The stack $\muSh_{/[1], \Lambda}^{\mathrm{mt}}$ defined above is smaller than or equal to the Kashiwara--Schapira stack $\muSh_{/[1]}(\bfk_{\Lambda})$ in \cite{Gu19}. 
	We put the microlocally tameness condition since they are easier to treat. 
\end{remark}

Arguments similar to \cite[\S10.4]{Gu19} show the following proposition.

\begin{proposition}\label{proposition:simpleglobal}
	Let $\Lambda$ be a locally closed conic Lagrangian submanifold of $T^*X$. 
	The category $\muSh_{/[1], \Lambda}^{\mathrm{mt}}(\Lambda)$ has a unique simple object. 
\end{proposition}

\subsubsection{Doubling functor and doubled sheaves}

First we construct $G_{(0,\varepsilon)}$ for a sufficiently small $\varepsilon>0$ so that $G_{(0,\varepsilon)}$ is locally isomorphic to an image of $\Psi_U$, which we define below.   
We introduce a variant of the convolution functor $\star$ in \cref{section:Tamarkincat}. 
Set $\gamma \coloneqq \{(u,t)\mid 0 \leq t <u \} \subset \bR_{>0} \times \bR$. 
For an open subset $U \subset M \times \Stheta$, we define
\begin{equation}
    U_\gamma \coloneqq \{(x,u,t)\in M\times \bR_{>0} \times \Stheta 
    \mid 
    \text{$(x,t-[c]) \in U$ for any $c \in [0,u]$} \}
\end{equation}
We also define a functor $\Psi_U\colon \Db_{/[1]}(U)\to \Db_{/[1]}(U_\gamma)$ by $\Psi_U(F) \coloneqq R{s_U}_!(F \boxtimes \bfk_\gamma)|_{U_\gamma}$, where $s_U \colon U\times \bR_{>0}\times \bR\to M\times \bR_{>0}\times \Stheta$ is $(x,t_1,u, t_2)\mapsto (x,u, t_1+[t_2])$. 

The next lemma follows from \cite[Thm.~11.1.7]{Gu19}. 

\begin{lemma}\label{lemma:F0=0}
	Let $s_{23}$ be the swapping map $M\times \Stheta \times \bR\to M\times \bR\times\Stheta,(x,t,u)\mapsto (x,u,t)$ and $j_U$ be the open embedding $U_\gamma\to U_\gamma\cup s_{23}(U\times \bR_{\leq 0})$. 
	Then $R{j_U}_*\Psi_U(F)|_{s_{23}(U\times \{0\})}\simeq 0$ for any $F\in \Db_{/[1]}(U)$. 
\end{lemma}

\begin{definition}[{cf.\ \cite[Def.~11.4.1]{Gu19}}]\label{definition:adaptedcoverings}
	Let $\Lambda$ be a conic Lagrangian submanifold of $\Omega_+(M)_\theta$ such that $\Lambda/\bR_{>0}$ is compact and $\Lambda/\bR_{>0}\to M$ is finite. 
	A finite family $\cU=\{U_b\}_{b\in B}$ of open subsets of $M\times \Stheta$ is said to be \emph{adapted} to $\Lambda$ if it satisfies the following conditions: 
	\begin{enumerate}
	\renewcommand{\labelenumi}{$\mathrm{(\arabic{enumi})}$}
		\item $\pi_{M\times \Stheta }(\Lambda)\subset \bigcup_{b \in B} U_b$.
		\item For each $b\in B$, there exist an open subset $W_b$ of $M$ and a contractible open subset $I_b$ of $\Stheta$ such that $U_b=W_b\times I_b$ and $\pi_{M\times \Stheta }(\Lambda)\cap U_b\subset W_b\times K$ for some compact subset $K$ of $I_b$.
		\item For any $B_1\subset B$, $\cRHom(\bfk_{U^{B_1}},\bfk_{M\times\Stheta})\simeq \bfk_{\overline{U^{B_1}}}$ 
		where $U^{B_1} \coloneqq \bigcup_{b \in B_1}U_b$. 
		\item
		Setting $\Lambda_+ \coloneqq \Lambda\cup T^*_{M\times \Stheta}{(M\times \Stheta)}$, one has 
		\begin{equation}\label{equation:adapted}
		(\MS(\bfk_{U^{B_1}})\widehat{+}\MS(\bfk_{U^{B_2}})^a)\cap (\Lambda_+\widehat{+}(\Lambda_+)^a)\subset T^*_{M\times \Stheta}{(M\times \Stheta)}
		\end{equation}
		for any $B_1,B_2\subset B$. 
	\end{enumerate}
	See \cite[Def.~6.2.3(v) and Rem.~6.2.8(ii)]{KS90} for the definition of $\widehat{+}$ in \eqref{equation:adapted}. 
\end{definition}

Similarly to \cite[Lemma~11.4.2]{Gu19}, we obtain the following.

\begin{lemma}\label{lemma:genericallyadapted}
	Let $\Lambda$ be a conic Lagrangian submanifold of $\Omega_+(M)_\theta$ such that $\Lambda/\bR_{>0}$ is compact and let $\{\Lambda_j\}_{j\in J}$ be a finite open covering of $\Lambda$ by conic subsets. Then there exist
	\begin{enumerate}
	    \item a homogeneous Hamiltonian isotopy $\wh{\phi}$ on $\Omega_+(M)_\theta$, as closed to $\id$ as desired and
	    \item a finite family $\{U_b\}_{b\in B}$ of open subsets of $M \times \Stheta$ that is adapted to $\wh{\phi}(\Lambda)$
	\end{enumerate}
	 such that for each $b \in B$, each connected component of $\wh{\phi}(\Lambda)\cap T^*U_b$ is contained in $\wh{\phi}(\Lambda_j)$, for some $j\in J$.
\end{lemma}

\begin{definition}\label{definition:doubledsheaves}
	Let $\Lambda$ and $\cU=\{U_b\}_{b \in B}$ be as in \cref{definition:adaptedcoverings}. 
	Let $U$ be an open subset of $M \times \Stheta$. 
	We denote by $\BD^{\mathrm{dbl}}_{/[1],\Lambda, \cU}(U)$ the subcategory of $\Db_{/[1]}({U\times \bR_{>0}})$ formed by $F$ such that, for sufficiently small $\varepsilon>0$ 
	\begin{enumerate}
	\renewcommand{\labelenumi}{$\mathrm{(\arabic{enumi})}$}
		\item $\Supp(F)\cap (U\times (0,\varepsilon)) \subset \{(x,t+[c],u)\in U\times (0,\varepsilon)\mid (x,t)\in \pi(\Lambda), c\in [0,u] \}$,
		\item every point of $U$ has a neighborhood $W$ such that $\pi_0(\Lambda\cap T^*W)=\{\Lambda_i\}_i$ is finite, and for each $\Lambda_i$ there exist a subset $B_i\subset B$ and a microlocally tame object $F_i \in \Db_{/[1],\Lambda_i}(\bfk_W)$ with $\rMS (F_i)=\Lambda_i$ such that
		\begin{equation}
		    F|_{W^\varepsilon}\simeq \bigoplus_{i} \Psi_{W}(\RG_{U^{B_i}}(F_i))|_{W^\varepsilon},
		\end{equation}
		where $W^\varepsilon \coloneqq W_\gamma\cap W\times (0,\varepsilon)$ and $U^{B_i} \coloneqq \bigcup_{b \in B_i}U_b$. 
	\end{enumerate}
\end{definition}

For $F \in \BD^{\mathrm{dbl}}_{/[1],\Lambda, \cU}(U)$, 
there exists a well-defined open subset $\MS^{\mathrm{dbl}}(F)$ of $\Lambda\cap T^*U$ locally defined by $\MS^{\mathrm{dbl}}(F)\cap T^*W  \coloneqq \bigcup_i (\Lambda_i \cap T^*U^{B_i})$ with the notation of \cref{definition:doubledsheaves}. 
For an object $F \in \BD^{\mathrm{dbl}}_{/[1],\Lambda,\cU}(M\times \Stheta)$ satisfying $\MS^{\mathrm{dbl}}(F)=\Lambda$, the functor $\frakm_{\Lambda} \colon \BD_{/[1],(\Lambda)}^{\mathrm{b},mt}(U) \to \muSh_{/[1], \Lambda}^{\mathrm{mt}}(\Lambda \cap T^*U)$ defines $\frakm_{\Lambda}^{\mathrm{dbl}}(F) \in \muSh_{/[1],\Lambda}^{\mathrm{mt}}(\Lambda)$
so that 
$\frakm_{\Lambda}^{\mathrm{dbl}}(F)|_{\Lambda_i} \simeq \frakm_{\Lambda_i}(F_i)|_{\Lambda_i}$, again with the notation of \cref{definition:doubledsheaves}. 
See \cite[Part~12]{Gu19} for more details.
Arguing similarly to \cite{Gu19} with some $\bR$'s replaced by $\Stheta$'s, one can prove the following. 

\begin{proposition}[{cf.\ \cite[Thm.~12.2.2]{Gu19}}]\label{proposition:existence-doubled}
	Let $\Lambda$ and $\cU=\{U_b\}_{b \in B}$ be as in \cref{definition:adaptedcoverings}. 
	For any object $\cF \in \muSh_{/[1],\Lambda}^{\mathrm{mt}}(\Lambda)$ there exists an object $F \in \BD^{\mathrm{dbl}}_{/[1],\Lambda,\cU}(M\times \Stheta)$ such that $\MS^{\mathrm{dbl}}(F)=\Lambda$ and $\frakm_{\Lambda}^{\mathrm{dbl}}(F)\simeq \cF$. 
\end{proposition}

Now we give a proof of \cref{theorem:existence-quantization}.

\begin{proof}[Proof of \cref{theorem:existence-quantization}]
	Since the conditions in \cref{theorem:existence-quantization} are preserved by the action of a homogeneous Hamiltonian isotopy on $\Omega_+(M)_\theta$, we may assume that there exists a family $\cU$ adapted to $\Lambda$ by \cref{lemma:genericallyadapted}. 
	By \cref{proposition:simpleglobal} and \cref{proposition:existence-doubled}, there exists $F\in \BD^{\mathrm{dbl}}_{/[1],\Lambda,\cU}(M\times \Stheta)$ such that $\MS^{\mathrm{dbl}}(F)=\Lambda$ and $\frakm_{\Lambda}^{\mathrm{dbl}}(F)$ is simple. 
	For a sufficiently small $a>0$, the object $F|_{M\times (0,a)\times \Stheta}\in \Db_{/[1]}(M\times (0,a)\times \Stheta)$ satisfies the conditions (1)--(4), where (3) follows from \cref{lemma:F0=0} and (4) is verified by the microlocally tameness of $F_i$'s in \cref{definition:doubledsheaves}. 
	
	The construction for a larger $a\in (0,r(\iota))$ is parallel to that in  \cite[\S12.3]{Gu19}. 
	We use a homogeneous Hamiltonian isotopy $\tl{\phi}=(\tl{\phi}_u)_{u\in (0,a)}\colon \rT(M\times \Stheta)\times (0,a)\to\rT(M\times \Stheta)$ such that 
	\begin{enumerate}
		\renewcommand{\labelenumi}{$\mathrm{(\alph{enumi})}$}
		\item $\tl{\phi}_\varepsilon=\id$ for some $\varepsilon\in (0,a)$, 
		\item $\tl{\phi}_u$ is identity on $\Lambda$ for any $u \in (0,a)$, 
		\item for $\varepsilon\in (0,a)$ given in (a), any $u\in (0,a)$, and any $(x,t;\xi,\tau)\in \Lambda$, one has $\tl{\phi}_u(x,t+\varepsilon; \xi,\tau)=(x,t+u; \xi,\tau)$. 
	\end{enumerate}
	Such $\tl{\phi}$ exists since $\Lambda $ is disjoint from $T'_u \Lambda$ for any $u\in (0,a)$, where $T'_u\colon \rT(M\times \Stheta)\to \rT(M\times \Stheta)$ is the lift of $T_u\colon M\times \Stheta\to M\times \Stheta$. 
	Hence we can apply a parallel argument to obtain an object $G'_{(0,a)}$ on $M\times (0,a) \times \Stheta$.
	Since $a$ is strictly smaller than $r(\iota)$, near the $\{u=a\}$-part, the boundedness of the quantization of $\tl{\phi}$ is guaranteed. 
	Defining $G_{(0,a)} \coloneqq P_l(G'_{(0,a)}) \in {}^\perp \Db_{/[1], \{\tau \le 0\}}(M \times (0,a) \times \Stheta)$, one can check the conditions (1)--(4) for $G_{(0,a)}$ by using those of the object corresponding to a smaller $a$. 
	
	Let us show $d_{\cD^{(0,a)}(M)_\theta}(G_{(0,a)},0)\leq a$.
	Take $\tilde{a}\in (a,r)$ and $G_{(0,\tilde{a})}\in \cD^{(0,\tilde{a})}_L(M)_\theta$ satisfying the conditions in \cref{theorem:existence-quantization} with $a$ replaced by $\tilde{a}$ so that $G_{(0,\tilde{a})}|_{M\times (0,a)\times \Stheta}$ is isomorphic to $G_{(0,a)}$. 
	Define $D_{\tilde{a}} \coloneqq \{ (u,s)\in \bR^2\mid 0<u<s<\tilde{a}\}$ and denote by $p \colon M \times D_{\tilde{a}} \times \Stheta \to M \times (0, \tilde{a}) \times \Stheta, (x,u,s,t) \mapsto (x,u,t)$ the projection.
	Moreover, we set $\cG \coloneqq p^{-1}G_{(0,\tilde{a})}\in \cD^{D_{\tilde{a}}}_L(M)_\theta$. 
	Define the map 
	\begin{equation}
	\begin{aligned}
	    k \colon M\times D_{\tilde{a}}\times \Stheta & \to M\times (-\infty , a) \times (-\infty, \tilde{a}) \times \Stheta, \\
	    (x,u,s,t) & \mapsto (x,u-s+a,s,t)
	\end{aligned}
	\end{equation}
	and set $\cG' \coloneqq Rk_!\cG|_{M\times (0,a)\times (-\infty,\tilde{a})\times \Stheta}$.
	Then it satisfies $\cG'|_{\{s=0\}}\simeq 0$, $\cG'|_{\{s=a\}}\simeq G_{(0,a)}$ and $\MS(\cG')\subset T^*(M\times (0,a))\times \{ 0 \leq \sigma \leq \tau \}_{T^*((-\infty,\tilde{a})\times \Stheta)}$. 
	Thus we obtain the result by \cref{lemma:torhtpy}. 
\end{proof}

\begin{remark}
    In our situation, we cannot apply the above construction for $a>r(\iota)$ since $T'_u \Lambda$ may intersect $\Lambda$ if $u\ge r(\iota)$ and an isotopy $\tl{\phi}$ as above does not exist. 
    This is one of the differences from the construction of \cite{Gu19}.
    We also remark that a conic half-line in the intersection $\Lambda \cap T'_u \Lambda$ corresponds to a Reeb chord of $\Lambda/\bR_{>0}$ in the cosphere bundle. 
    
\end{remark}

\begin{remark}\label{remark:construction-r}
    One could take $a=r(\iota)$ to obtain an object $G_{(0,r(\iota))}$. 
    For construction, we need to use the sheaf quantization of a homogeneous Hamiltonian isotopy $\tl{\phi}=(\tl{\phi}_u)_{u\in (0,r(\iota))}$ that diverges at $u=r(\iota)$. 
    Hence the boundedness of the quantization of $\tl{\phi}$ 
    and the well-definedness of $G_{(0,r(\iota))}$ are unclear. 
    This is why we take $a \in (0,r(\iota))$ and construct $G_{(0,a)}$. 
    Moreover, it gets more complicated to obtain similar results to \cref{section:rational-immersions} with the possibly constructed $G_{(0,r(\iota))}$.
\end{remark}

\section{Intersection of rational Lagrangian immersions}
\label{section:rational-immersions}

In this section, using the refined version of Tamarkin category introduced in \cref{section:Tamarkincat} and the sheaf quantization constructed in \cref{section:constrution}, we give explicit estimates for the displacement energy and the number of the intersection points of rational Lagrangian immersions (\cref{theorem:betti-estimate,theorem:cuplength} below).
These are proved by a purely sheaf-theoretic method and partial generalizations of results of Chekanov~\cite{Chekanov98}, Akaho~\cite{Akaho}, and Liu~\cite{Liu} (see \cref{remark:comparisontoFloer} for more details).

\subsection{Statements of theorems}

First we give the definition of rational Lagrangian immersions.
Here again we assume that $L$ is a compact connected manifold.

\begin{definition}\label{definition:rational}
	\begin{enumerate}
		\item A Lagrangian immersion $\iota\colon L\to T^*M$ is said to be \emph{rational} if there exists $\sigma(\iota) \in \bR_{\ge 0}$ such that
		\begin{equation}
		\left\{ \int_{D^2}v^*\omega \; \middle| \; (v,\bar{v}) \in \Sigma(\iota)
		\right\}
		=
		\sigma(\iota) \cdot \bZ,
		\end{equation}
		where
		\begin{equation}
		\Sigma(\iota)
		 \coloneqq 
		\left\{ (v,\bar{v})
		\; \middle| \;
		\begin{aligned}
		& v \colon D^2 \to T^*M, \bar{v} \colon
		\partial D^2 \to L, \\
		& v|_{\partial D^2}=\iota \circ \bar{v}
		\end{aligned}
		\right\}.
		\end{equation}
		We call $\sigma(\iota)$ the \emph{rationality constant} of $\iota$.

		\item For a rational Lagrangian immersion $\iota \colon L \to T^*M$, one defines
		\begin{equation}
		e(\iota)
		 \coloneqq 
		\inf \left( \left\{ \int_{D^2}v^*\omega \; \middle| \; (v,\bar{v}) \in E(\iota)\amalg \Sigma(\iota) \right\} \cap \bR_{>0} \right),
		\end{equation}
		where
		\begin{equation}
		E(\iota)
		 \coloneqq 
		\left\{ (v,\bar{v})
		\; \middle| \;
		\begin{aligned}
		& v \colon D^2 \to T^*M, \bar{v} \colon
		[0,1] \to L, \\
		& \bar{v}(0) \neq \bar{v}(1),
		\iota \circ \bar{v}(0) = \iota \circ \bar{v}(1), \\
		& v|_{\partial D^2}\circ \exp(2\pi \sqrt{-1} (-))=\iota \circ \bar{v}
		\end{aligned}
		\right\}. 
		\end{equation} 
	\end{enumerate}
\end{definition}

\begin{remark}
	A strongly rational Lagrangian immersion is rational.
	Indeed, for a strongly rational Lagrangian immersion $\iota$ with period $\theta(\iota)$, one can check that it is rational with rationality constant $n \theta(\iota)$ for some $n \in \bZ_{\ge 0}$.
	However, the converse is not true.
	For example, the graph of any closed $1$-form $\beta$ on a compact connected manifold $M$ is rational with rationality constant $0$, but this embedding has a period $\theta(\iota)$ if and only if there exists a primitive element $b \in H^1(M;\bZ)$ such that $[\beta]=\theta(\iota) \cdot b \in H^1(M;\bR)$.
\end{remark}

We make the following assumption, which we will use in the reduction to the case of a strongly rational Lagrangian immersion with \cref{assumption:embedding} in the next subsection (see \cref{lemma:reduction}).

\begin{assumption}\label{assumption:no0disk}
    There exists no $(v,\bar{v}) \in E(\iota)$ with $\int_{D^2}v^*\omega =0$. 
\end{assumption}

Our results are the following: the first one is an estimate for the number of Lagrangian intersection for immersions by the total Betti number of $L$ under the transverse assumption, and the second is an estimate by the cup-length of $L$.

\begin{theorem}\label{theorem:betti-estimate}
	Let $\iota \colon L \to T^*M$ be a compact rational Lagrangian
	immersion satisfying \cref{assumption:no0disk}.
	If $\|H\| <e(\iota)$ and $\iota \colon L \to T^*M$ intersects $\phi^H_1
	\circ \iota \colon L \to T^*M$ transversally, then
	\begin{equation}
	\# \left\{ (y,y') \in L \times L \; \middle| \; \iota(y)=\phi^H_1 \circ
	\iota(y') \right\}
	\ge
	\sum_{i=0}^{\dim L} b_i(L).
	\end{equation}
\end{theorem}

\begin{theorem}\label{theorem:cuplength}
	Let $\iota \colon L \to T^*M$ be a compact rational Lagrangian immersion satisfying \cref{assumption:no0disk}.
	If a Hamiltonian function $H$ satisfies $\|H\| < \min \left( \{e(\iota)\} \cup (\{ \sigma(\iota)/2 \}\cap \bR_{>0})\right)$,
	then
	\begin{equation}\label{equation:cuplength}
		\# \left\{ (y,y') \in L \times L \; \middle| \; \iota(y)=\phi^H_1 \circ
		\iota(y') \right\}
		\ge
		\cl (L)+1,
	\end{equation}
	where
	$\cl(L)$ denotes the cup-length of $L$ over $\bF_2$ (see \cref{subsubsec:cuplength} for the definition).
\end{theorem}

\begin{remark}\label{remark:eneqsigma}
    If $e(\iota) \neq \sigma(\iota)$ then $\min \left( \{e(\iota)\} \cup (\{ \sigma(\iota)/2 \}\cap \bR_{>0})\right)=e(\iota)$. 
\end{remark}

\begin{remark}\label{remark:comparisontoFloer}
    Our theorems are partial generalizations of results of Chekanov~\cite{Chekanov98}, Akaho~\cite{Akaho}, and Liu~\cite{Liu} in the following sense. 
    Their results hold on any compact symplectic manifold and do \emph{not} require \cref{assumption:no0disk}.
    Remark that they all used Floer-theoretic methods to prove the results and our method is independent from theirs.
    \begin{enumerate}
        \item The result of Chekanov~\cite{Chekanov98} is \cref{theorem:betti-estimate} for a rational Lagrangian \emph{embedding}~$\iota$ (i.e.,\ rational Lagrangian submanifold) with rationality constant $\sigma(\iota)>0$. 
        Remark that $e(\iota)=\sigma(\iota)$ in this case.
        
        \item The result of Akaho~\cite{Akaho} is \cref{theorem:betti-estimate} for an \emph{exact} Lagrangian immersion~$\iota$, which corresponds to $\sigma(\iota)=0$, under the condition that the non-injective points are transverse.
        
        \item Liu~\cite{Liu} proved that for a rational Lagrangian \emph{embedding} $\iota$ with rationality constant $\sigma(\iota)>0$, if $\|H\|<e(\iota)=\sigma(\iota)$ then the estimate~\eqref{equation:cuplength} holds.
        This estimate is \cref{theorem:cuplength} with the bound $\min \left( \{e(\iota)\} \cup (\{ \sigma(\iota)/2 \}\cap \bR_{>0})\right)=\sigma(\iota)/2$ replaced by $\sigma(\iota)$.
    \end{enumerate}
\end{remark}

\begin{remark}\label{remark:comparisontoFOOO}
    In the Floer-theoretic approach, one can study the cases $\|H\|\geq e(\iota)$ using bounding cochains in the sense of \cite{FOOO09,FOOO092,AJ10}. 
    Fukaya--Oh--Ohta--Ono~\cite[Thm.~J]{FOOO09,FOOO092} and \cite[Thm.~6.1]{FOOO13} gave an estimate for the number of the intersection points of Lagrangian submanifolds.
    Moreover, they proved a common generalization of the results of Chekanov~\cite{Chekanov98} and \cite[Thm.~J]{FOOO09,FOOO092} with a slight modification of the definition of bounding cochains \cite[Thm.~6.5.47]{FOOO09}. 
\end{remark}

\subsection{Reduction to strongly rational case}\label{subsec:reduction}

In this subsection, we reduce the problem to the strongly rational case.

\begin{notation}
    For a Lagrangian immersion $\iota \colon L \to T^*M$, one defines $F(\iota)$ to be the set of non-injective points: $F(\iota) \coloneqq \{(y,y') \in L \times L \mid y \neq y', \iota(y)=\iota(y')\}$. 
\end{notation}

\begin{lemma}\label{lemma:stronglyrational}
	Let $\iota\colon L\to T^*M$ be a compact connected rational Lagrangian immersion with rationality constant $\sigma(\iota)$.
	Assume that $\pi_1(\pi\circ \iota) \colon \pi_1(L)\to \pi_1(M)$ is surjective.
	Then, there exists a closed $1$-form $\beta$ on $M$ such that the immersion
	$\iota +\beta \colon L\to T^*M, y \mapsto \iota(y)+\beta(\pi\circ\iota(y))$
	is strongly rational with $\theta(\iota+\beta)=\sigma(\iota+\beta)=\sigma(\iota)$ and $r({\iota +\beta})=e({\iota+\beta})=e(\iota)$.
\end{lemma}

\begin{proof}
	Since $\pi_1(L)\to \pi_1(M)$ is surjective, so is the induced homomorphism of groups $[\pi_1(L),\pi_1(L)] \to [\pi_1(M),\pi_1(M)]$.
	Consider the commutative diagram of groups
	\begin{equation}
	\begin{aligned}
		\xymatrix@R=15pt{
			&&1\ar[d]&1\ar[d]&\\
			& &[\pi_1(L),\pi_1(L)]\ar@{^{(}->}[d]\ar@{->>}[r]&[\pi_1(M),\pi_1(M)]\ar@{^{(}->}[d]\ar[r]&1\\
			1\ar[r]&\Ker (\pi_1(\pi\circ\iota))\ar@{^{(}->}[r]\ar[d]&\pi_1(L)\ar@{->>}[r]\ar@{->>}[d]&\pi_1(M)\ar[r]\ar@{->>}[d]&1\\
			0\ar[r]&\Ker (H_1(\pi\circ\iota;\bZ))\ar@{^{(}->}[r]&H_1(L;\bZ)\ar[d]\ar@{->>}[r]&H_1(M;\bZ)\ar[r]\ar[d]&0 \\
			& & 0 & \ 0. &
		}
	\end{aligned}
	\end{equation}
	By the nine lemma for groups, we find that $\Ker (\pi_1(\pi\circ\iota)) \to \Ker (H_1(\pi\circ\iota;\bZ))$
	is surjective.
	
	Choose a section $w \colon H_1(M;\bR)\to H_1(L;\bR)$ of $H_1(\pi\circ\iota;\bR)\colon H_1(L;\bR)\to H_1(M;\bR)$ and take a closed $1$-form $\beta$ on $M$ such that $[\beta]=\langle w(-), [-\iota^* \alpha]\rangle \in H^1(M;\bR)\simeq \Hom(H_1(M;\bR),\bR)$.
	Since $(\iota+\beta)^*\alpha =\iota^*\alpha+(\pi\circ\iota)^*\beta$, the pairing with $[(\iota+\beta)^*\alpha ]$ vanishes on the image of $w$.  
	Thus we obtain
	\begin{equation}
	\begin{aligned}
		\left\{ \int_\gamma (\iota+\beta)^*\alpha \; \middle| \; \gamma \in H_1(L;\bZ) \right\}
		& =
		\left\{ \int_\gamma (\iota+\beta)^*\alpha \; \middle| \; \gamma \in \Ker(H_1(\pi\circ\iota;\bZ)) \right\} \\
		& =
		\left\{ \int_\gamma \iota^*\alpha \; \middle| \; \gamma \in \Ker(H_1(\pi\circ\iota;\bZ)) \right\} \\
		& =
		\left\{ \int_\gamma \iota^*\alpha \; \middle| \; \gamma \in \Ker(\pi_1(\pi\circ\iota)) \right\}
		=\sigma(\iota) \cdot \bZ,
	\end{aligned}
	\end{equation}
	where the third equality follows from the surjectivity of $\Ker (\pi_1(\pi\circ\iota)) \to \Ker (H_1(\pi\circ\iota;\bZ))$.
	This proves that $\iota+\beta$ is strongly rational and $\theta(\iota+\beta)=\sigma(\iota)=\sigma(\iota+\beta)$.
	
	For each $(y_0,y_1)\in F({\iota+\beta})=F(\iota)$, there exists an element $(v,\bar{v})\in E(\iota)$ such that $\bar{v}(i)=y_i \ (i=0,1)$. We obtain such a pair $(v,\bar{v})$ as follows. 
	Take a path connecting $y_0$ and $y_1$ in $L$.
	Composing $\pi \circ \iota$ with this path gives an element of $\pi_1(M,\pi\circ\iota(y_0))$. 
	We can take a preimage of the element in $\pi_1(L,y_0)$ by the surjectivity of $\pi_1(L)\to \pi_1(M)$. 
	Concatenating a representative path of the inverse of the preimage to the original path on $L$, we obtain a path $\bar{v}$ connecting $y_0$ and $y_1$ that bounds a disk $v$ in $T^*M$. 
	By the existence of such $(v,\bar{v}) \in E(\iota)$, we conclude $r({\iota+\beta})=e({\iota+\beta})=e(\iota)$.
\end{proof}

\begin{lemma}\label{lemma:reduction}
	Assume that \cref{theorem:betti-estimate,theorem:cuplength}
	hold for any strongly rational Lagrangian immersion
	$\iota \colon L \to T^*M$ satisfying \cref{assumption:embedding}, 
	$\theta(\iota)=\sigma(\iota)$, and $r(\iota)=e(\iota)$.
	Then \cref{theorem:betti-estimate,theorem:cuplength}
	hold for any rational Lagrangian immersion.
\end{lemma}

\begin{proof}
	Take the covering $p\colon \tl{M}\to M$ corresponding to $\iota_*(\pi_1(L))\subset \pi_1(M)$. 
	Then a lift $\tl{\iota} \colon L \to T^*\tl{M}$ of $\iota$ induces a surjection on the fundamental groups and $\sigma(\tl{\iota})=\sigma(\iota)$. 
	By the construction of $p$, a non-injective point $(y,y')\in F(\iota)$ of $\iota$ is one for $\tl{\iota}$ if and only if there exists $(v,\bar{v})\in E(\iota)$ with $(y,y')=(\bar{v}(0),\bar{v}(1))$. 
	Hence $e({\tl{\iota}})=e(\iota)$ and \cref{assumption:no0disk} for $\iota$ is equivalent to \cref{assumption:no0disk} for $\tl{\iota}$. 
	
	Take a closed $1$-form $\beta$ on $\tl{M}$ satisfying the conclusion of \cref{lemma:stronglyrational} for $\tl{\iota}$. 
	By the surjectivity of $\pi_1(\tl{\iota})$, \cref{assumption:embedding} for $\tl{\iota}+\beta$ is equivalent to \cref{assumption:no0disk} for $\tl{\iota}+\beta$, which is equivalent to \cref{assumption:no0disk} for $\iota$. 
	Furthermore, for a Hamiltonian function $H$ on $T^*M$, setting $\tl{H}$ to be the composite of $H$ and the projection $T^*\tl{M} \to T^*M$, we get $\|\tl{H}\|=\|H\|$ and 
	\begin{equation}
	\begin{aligned}
	& \left\{ (y,y') \in L \times L \; \middle| \; \iota(y)=\phi^H_1 \circ
	\iota(y') \right\} \\
	\supset & 
	\left\{ (y,y') \in L \times L \; \middle| \; \tl{\iota}(y)=\phi^{\tl{H}}_1 \circ
	\tl{\iota}(y') \right\} \\
	= & 
	\left\{ (y,y') \in L \times L \; \middle| \; (\tl{\iota}+\beta )(y)=\phi^{\tl{H}}_1 \circ
	(\tl{\iota }+\beta )(y') \right\}.
	\end{aligned}
	\end{equation}
	Thus \cref{theorem:betti-estimate,theorem:cuplength} for $\iota$ are reduced to those for $\tl{\iota}+\beta$. 
\end{proof}

\subsection{Proof for strongly rational cases}\label{subsection:proof}

This subsection is devoted to the proofs of \cref{theorem:betti-estimate,theorem:cuplength} for a strongly rational Lagrangian immersion.
In what follows, we assume the following.

\begin{assumption}\label{assumption:special-stongly-rational}
	An immersion $\iota \colon L \to T^*M$ is
	a strongly rational Lagrangian immersion
	satisfying \cref{assumption:embedding}, $\theta(\iota)=\sigma(\iota)$, and $r(\iota)=e(\iota)$.
\end{assumption}

We write $\theta=\theta(\iota)$ for simplicity.
Set $\Lambda=\Lambda_{\iota,f}$ as in \eqref{equation:def-conification-lambda}.
Let $a \in (0,r(\iota))$ and $G_{(0,a)} \in {}^\perp \Db_{/[1], \{\tau \le 0\}}(M \times (0,a) \times \Stheta)$ be an object given in \cref{theorem:existence-quantization}.
We denote by $j_a \colon M \times (0,a)\times \Stheta\to M \times \bR\times \Stheta$ the open embedding and define
\begin{equation}
	F_{(0,a)} \coloneqq {j_a}_!G_{(0,a)}, \ 
	F_{[0,a]} \coloneqq R{j_a}_*G_{(0,a)} \in \Db_{/[1]}(M \times \bR \times \Stheta).
\end{equation}
Note that $F_{(0,a)} \in {}^\perp \Db_{/[1], \{\tau \le	0\}}(M \times \bR \times \Stheta)$ and 
\begin{equation}
    \Hom(F_{(0,a)}, F_{[0,a]}) \simeq H^* \RG_{[0,+\infty)}(\bR;\ell^! Rq_* \cHom^\star(F_{(0,a)},F_{[0,a]}))
\end{equation}
by \cref{proposition:morD}, where $F_{[0,a]}$ also means $Q(F_{[0,a]}) \in \cD^\bR(M)_\theta$ by abuse of notation and $\Hom(-,-)$ denotes $\Hom_{\cD^{\bR}(M)_\theta}(-,-)$ unless otherwise specified hereafter.
Using the fact that $F_{(0,a)}$ is in the left orthogonal, we also find that $\Hom(F_{(0,a)},F_{[0,a]})$ is naturally isomorphic to $\End(G_{(0,a)})$ by the adjunction ${j_a}_! \dashv {j_a}^{-1}$.

\begin{notation}
	One defines subsets of $T^*\bR$ by 
	\begin{align*}
		\bfc (a)& \coloneqq \{(0;\upsilon )\mid -1\leq \upsilon \leq 0\}\cup \{(u;\upsilon)\mid 0\leq u \leq a,\upsilon =0,-1 \} ,\\
		\bfd (a)& \coloneqq \bfc (a)\cup \{ (a;\upsilon) \mid \upsilon \geq -1 \},\\
		\bfq (a)& \coloneqq \bfc (a)\cup \{ (a;\upsilon) \mid \upsilon \leq 0 \},\\
		\bfl (a)& \coloneqq \{ (a;\upsilon)\mid -1<\upsilon <0 \} .
	\end{align*}
	One also defines their ``conifications'', which are subsets of $\Omega_+(\bR)_\theta$, by
	\begin{align*}
		\wh{\bfc} (a)& \coloneqq \{(u, -u\upsilon;\tau\upsilon,\tau)\mid (u;\upsilon) \in \bfc (a)\},\\
		\wh{\bfd} (a)& \coloneqq \wh{\bfc} (a) \cup \{(a,0;\upsilon,\tau)\mid \upsilon>0 \}
		\cup \{(a,[a];\upsilon,\tau)\mid \upsilon>-\tau \},\\
		\wh{\bfq} (a)& \coloneqq \wh{\bfc} (a)\cup \{(a,0;\upsilon,\tau)\mid \upsilon<0 \}
		\cup \{(a,[a];\upsilon,\tau)\mid \upsilon<-\tau \},\\
		\wh{\bfl} (a)& \coloneqq \{(a,[a];\upsilon, \tau)\mid -\tau<\upsilon<0 \}.
	\end{align*}
\end{notation}

\begin{notation}
	For cones $C_1\subset \Omega_+(M)_\theta$ and $C_2\subset \Omega_+(P)_\theta$, 
	one defines 
	\begin{equation}
    	C_1\boxplus C_2 
    	 \coloneqq 
    	\left\{
    	    (x,y,t_1+t_2;\xi,\eta,\tau)
    	    \; \middle| \; 
    	    \begin{aligned}
    	        & (x,t_1;\xi,\tau) \in C_1, \\
    	        & (y,t_2;\eta,\tau)\in C_2 
    	    \end{aligned}
    	\right\}
    	\subset 
    	\Omega_+.
	\end{equation}
\end{notation}

\begin{lemma}\label{lemma:SSF0a}
\begin{enumerate}
	\item One has $\MS_+(F_{(0,a)}) \subset \Lambda \boxplus {\wh{\bfd}(a)}$ and $\MS_+(F_{[0,a]})\subset \Lambda \boxplus {\wh{\bfq}(a)}$.
	In particular, $F_{(0,a)}, F_{[0,a]}$ are objects of $\cD_{\iota(L)}^\bR(M)_\theta$.
	\item One has 
	\begin{align}
	    \rMS(F_{(0,a)}) \cap \{\tau =0\}_{T^*(M\times\bR\times \Stheta)} 
	    & \subset \{u=a,\xi=0, \upsilon \geq 0 \}_{T^*(M\times\bR\times \Stheta)}, \\
	    \rMS(F_{[0,a]})\cap \{\tau =0\}_{T^*(M\times\bR\times \Stheta)} 
	    & \subset \{u=a, \xi=0, \upsilon\leq0 \}_{T^*(M\times\bR\times \Stheta)}. \label{equation:ringss-estimate}
	\end{align}
\end{enumerate}
\end{lemma}

\begin{proof}
    In this proof, $\{ - \}$ denotes $\{ - \}_{T^*(M\times\bR\times \Stheta)}$.
	Since $\Lambda_q\cup \Lambda_r\subset \{ \upsilon  =0,-\tau\}$ and $N^*(M\times (0,a)\times \Stheta)\subset \{ \tau=0 \}$, we find that $\overline{\MS(G_{(0,a)})}\cap N^*(M\times (0,a)\times \Stheta)$ and 
	$\overline{\MS(G_{(0,a)})}\cap N^*(M\times (0,a)\times \Stheta)^a$ are contained in the zero-section. 
	By \cref{proposition:SSopenimm}, we obtain $\MS(F_{(0,a)})\subset \overline{\MS(G_{(0,a)})}+ N^*(M\times (0,a)\times \Stheta)^a$ and $\MS(F_{[0,a]})\subset\overline{\MS(G_{(0,a)})}+ N^*(M\times (0,a)\times \Stheta)$. 
	Fiberwise computations show
	\begin{align*}
	\MS(F_{(0,a)})\cap \{u=0\}
	\subset & \ 
	\{(x,t;\xi,\tau)\in \Lambda, \upsilon \leq 0 \} \cup\{\tau=0,\xi=0, \upsilon \leq 0 \}, \\
	\MS(F_{[0,a]}) \cap \{u=0\}
	\subset & \ 
	\{(x,t;\xi,\tau)\in \Lambda, \upsilon \geq -\tau \} \cup\{\tau=0,\xi=0, \upsilon \geq 0 \},\\
	\MS(F_{(0,a)})\cap \{u=a\}
	\subset & \ 
	\{(x,t;\xi,\tau)\in \Lambda, \upsilon \geq 0 \} \cup \{(x,t-a;\xi,\tau)\in \Lambda, \upsilon \geq -\tau \} \\ 
	& \ \cup\{\tau=0,\xi=0, \upsilon \geq 0  \},\\
	\MS(F_{[0,a]})\cap \{u=a\}
	\subset & \ 
	\{(x,t;\xi,\tau)\in \Lambda, \upsilon \leq 0 \} \cup \{(x,t-a;\xi,\tau)\in \Lambda, \upsilon \leq -\tau \}\\ 
	& \ \cup\{\tau=0,\xi=0, \upsilon \leq 0  \}.
	\end{align*}
	The cone of the natural morphism $F_{(0,a)}\to F_{[0,a]}$ is supported on $M\times \{ a\}\times \Stheta$, since $F_{[0,a]}|_{M\times  \{0\}\times \Stheta}$ is isomorphic to $0$ by \cref{theorem:existence-quantization}(3). 
	Hence, by the triangle inequality (\cref{proposition:properties-ms}(ii)), we obtain 
	\begin{equation}
    	\rMS(F_{(0,a)})\cap \{u=0\}=\rMS(F_{[0,a]})\cap \{u=0\}\subset \{(x,t;\xi,\tau)\in \Lambda, -\tau \leq \upsilon \leq 0 \},
	\end{equation}
	which proves the results.
\end{proof}

Now let $H \colon T^*M \times I \to \bR$ be a compactly supported Hamiltonian function and denote by $\phi^H=(\phi^H_s)_{s \in I} \colon T^*M \times I \to T^*M$ the Hamiltonian isotopy generated by $H$.

\begin{notation}
    \begin{enumerate}
        \item One sets $\iota^H \coloneqq \phi^H_1\circ \iota$ and $f^H \coloneqq f-\ell \circ h\colon L\to \Stheta$ with $h \coloneqq u_1 \circ \iota \colon L \to \bR$ (see \Cref{subsection:appendix-sq-hamiltonian} for the definition of $u_1$).
        \item One also sets $C(\iota,H) \coloneqq \{ (y,y') \in L\times L \mid \iota(y)=\iota^H (y') \}$ and defines $g \colon C(\iota,H) \to \Stheta$ by $g(y,y') \coloneqq f^H(y')-f(y)$.
        \item For $y \in L$, one sets $l(y) \coloneqq \{ (x,-f(y);\xi,\tau)\in \Lambda\mid (x;\xi/\tau) =\iota(y)\} \subset T^*(M\times \Stheta)$.
        \item One defines $\Lambda^H \coloneqq \wh{\phi}^H_1(\Lambda)$ (see \cref{subsection:appendix-sq-hamiltonian} for $\wh{\phi}^H_1$), namely 
        \begin{equation}
            \Lambda^H=
            \left\{ 
            (x,t;\xi,\tau) \in T^*(M \times \Stheta) 
            \; \middle| \; 
            \begin{aligned}
        	    & \text{$\tau >0$, there exists $y \in L$}, \\
        	    & (x;\xi/\tau)=\iota^H(y), t=-f^H(y) 
            \end{aligned}
            \right\}. 
        \end{equation}
    \end{enumerate}
\end{notation}

We denote by $\Phi^H_1 \colon \cD^P(M)_\theta \to \cD^P(M)_\theta$ the functor associated with $\phi^H_1$ (see \eqref{eq:hamiltonian-induced-functor}) and set $F_{[0,a]}^H \coloneqq \Phi^H_1(F_{[0,a]})$.
Note that $\MS_+(F_{[0,a]}^H) \subset \Lambda^H \boxplus {\wh{\bfq}(a)}$ and \eqref{equation:ringss-estimate} also holds with $F_{[0,a]}$ replaced by $F_{[0,a]}^H$.

We denote by $q \colon M \times \bR \times \Stheta \to \Stheta$ the projection and by $\ell$ the quotient map $\bR \to \Stheta$. 
From now on, for simplicity we write $T_c$ instead of ${T_c}_*$ for $c \in \bR$. 

\begin{proposition}\label{proposition:interleavingF0a}
	\begin{enumerate}
		\item One has $d_{\cD(M)_\theta}(F_{[0,a]},F^H_{[0,a]}) \le \|H\|$.
		In particular, for any $a'>\|H\|$, there exist $b \in [0,a']$ and
		morphisms $\alpha \colon F_{[0,a]} \to T_b F^H_{[0,a]},
		\beta \colon F^H_{[0,a]} \to T_{a'-b} F_{[0,a]}$ such that
		$\tau_{0,a'} \colon F_{[0,a]} \to T_{a'} F_{[0,a]}$ is equal to $T_b \beta \circ \alpha$. 
		\item One has
		\begin{equation}\label{equation:SSvalue}
		\pi (\rMS(\ell^! Rq_*\cHom^\star(F_{(0,a)},F^H_{[0,a]}) ))
		\subset
		\left\{
		c \in \bR
		\; \middle| \;
		\begin{aligned}
		& \text{there exists $(y,y') \in C(\iota,H)$}, \\
		& \text{$g(y,y') \equiv -c \mod \theta$ or} \\
		& g(y,y') \equiv -c-a \mod \theta
		\end{aligned}
		\right\}.
		\end{equation}
		
		\item If $a'<a$, then
		$\tau_{0,a'} \colon \Hom (F_{(0,a)},F_{[0,a]}) \to \Hom (F_{(0,a)},T_{a'}F_{[0,a]})$
		is an isomorphism.
	\end{enumerate}
\end{proposition}

\begin{proof}
	(i) It follows from \cref{proposition:distance} and \cref{definition:distance}.

	\noindent (ii) 
	Let $T_{c}'$ be the translation map $\Omega_{+}\to\Omega_{+}$ or $\Omega_{+}(\bR)_\theta\to \Omega_{+}(\bR)_\theta$ that is the lift of $T_{c}$. 
	By \cref{lemma:SSF0a} and the above remark on $F_{[0,a]}^H$, we have $\rMS(i^{-1}F_{(0,a)}) \cap \rMS(F^H_{[0,a]})=\emptyset$. 
	Hence by \cref{proposition:SStenshom}(ii) and \cref{proposition:SSpushpull}, 
	\begin{equation}\label{equation:SSdq}
	\begin{aligned}
	    \pi(\rMS(\ell^! Rq_*\cHom^\star(F_{(0,a)},F^H_{[0,a]}) )
	&\subset \{-c \mid \rMS(F_{(0,a)})\cap T'_{c}(\rMS(F^H_{[0,a]}))\neq \emptyset \}\\
	&\subset \{-c \mid \Lambda \boxplus\wh{\bfd}(a)\cap T'_{c}(\Lambda^H\boxplus\wh{\bfq}(a))\neq \emptyset \}. 
	\end{aligned}
	\end{equation}
	If $(x,u,t;\xi,\upsilon,\tau)\in \Lambda \boxplus\wh{\bfd}(a)\cap T'_{c}(\Lambda^H\boxplus\wh{\bfq}(a))$, then there exist $t_1,t_2,t_3,t_4\in \Stheta$ with $t=t_1+t_2=t_3+t_4+[c]$ such that $(x,t_1;\xi,\tau)\in \Lambda, (u,t_2;\upsilon,\tau)\in \wh{\bfd}(a),
	(x,t_3;\xi,\tau)\in \Lambda^H, (u,t_4;\upsilon,\tau)\in \wh{\bfq}(a)$.
	Then $(u,t_2;\upsilon,\tau)\in \wh{\bfd}(a)\cap T'_{t_2-t_4}(\wh{\bfq}(a))$, where we use $T'_{c'}$ for $c' \in \Stheta$ by abuse of notation.
	Since
	\begin{equation}\label{eqn:intdq}
	\wh{\bfd}(a)\cap T'_{c'}(\wh{\bfq}(a))=
	\begin{cases}
	\wh{\bfc}(a)&(c'=0)\\
	\overline{~\wh{\bfl}(a)} &(c'=[a])\\
	\emptyset&(\text{otherwise}),
	\end{cases}
	\end{equation}
	we have $t_2-t_4=0,[a]$.
	By definition, there exist $y,y'\in L$ such that $\iota(y)=\iota^H(y')=(x,\xi/\tau), t_1=-f(y)$, and $t_3=-f^H(y')$.
	Then $g(y,y')=-t_3+t_1=[c]-(t_2-t_4)$, which implies that the sets in \eqref{equation:SSdq} are contained in the right-hand side of \eqref{equation:SSvalue} and the assertion holds. 
    	
	\noindent(iii) By \cref{proposition:morD}, we have
	\begin{equation}
	\Hom (F_{(0,a)},T_c F_{[0,a]})\simeq H^*\RG_{[-c,+\infty)}(\bR;\ell^!Rq_*\cHom^\star (F_{(0,a)},F_{[0,a]}))
	\end{equation}
	for any $c \in \bR$.
	Set $\cH=\ell^!Rq_*\cHom^\star (F_{(0,a)},F_{[0,a]})$.
	Applying (ii) to the case $H$ is a constant function, we get $[-a',0)\cap \pi(\MS_+(\cH))=\emptyset$. 
	Hence by the microlocal Morse lemma (\cref{proposition:microlocalmorse}) for $\cH$ and the five lemma, we find that $H^*\RG_{[0,+\infty)}(\bR;\cH) \to H^*\RG_{[-a',+\infty)}(\bR;\cH)$ is an isomorphism, which proves the result.
\end{proof}

Now we assume that the Hamiltonian function $H$ satisfies $\|H\| < r(\iota)$.
Moreover, we fix $a,a' \in \bR_{>0}$ such that $\|H \|< a'<a<r(\iota)$.
By \cref{proposition:interleavingF0a}(i) and (iii), the isomorphism $\tau_{0,a'}$ factors as
\begin{equation}\label{equation:factorization}
    \Hom(F_{(0,a)},F_{[0,a]}) \to \Hom(F_{(0,a)},T_b F^H_{[0,a]}) \to \Hom(F_{(0,a)},T_{a'}F_{[0,a]})
\end{equation}
for some $b \in [0,a']$.
We also fix such $b$ in what follows.

In order to study the second object in \eqref{equation:factorization}, we set 
\begin{equation}
\begin{aligned}
    \cH_b
     \coloneqq \ & \ell^!Rq_*\cHom^\star(F_{(0,a)},T_b F^H_{[0,a]}) \\
    \simeq \ & \ell^!Rq_* T_b \cHom^\star(F_{(0,a)},F^H_{[0,a]}).
\end{aligned}
\end{equation}
Note that $H^*\RG_{[c,+\infty)}(\bR;\cH_b) \simeq \Hom(F_{(0,a)},T_{b-c} F^H_{[0,a]})$ for $c \in \bR$.
We also define 
\begin{equation}
\begin{aligned}
    W_{c}
	& \coloneqq 
	H^*\RG_{[c,+\infty)}(\cH_b)_{c} \\
	& \simeq 
	H^*\RG_{[0,+\infty)}(\ell^!Rq_*\cHom^\star(F_{(0,a)},T_{b-c} F^H_{[0,a]}))_{0}
\end{aligned}
\end{equation}
for $c \in \bR$.

The following proposition is an essential tool for the proofs of theorems below, which decomposes the information of $\Hom(F_{(0,a)},T_b F^H_{[0,a]})$ into that of $W_c$'s.

\begin{proposition}\label{proposition:exact-triagle}
    In the situation above we have the following.
	\begin{enumerate}
		\item Assume that $c \in \bR$ is not an accumulation point of 
		$\pi (\MS_+(\cH_b))$.
		Take $d,d'\in \bR$ satisfying 
		    $\mathrm{(1)}$ $d\le c < d'$ and
			$\mathrm{(2)}$ $\pi (\MS_+(\cH_b)) \cap [d,d'] \subset\{c\}$
		and define
		\begin{align}
    		A_{c} & \coloneqq \Coker (\Hom(F_{(0,a)},T_{b-d'}F^H_{[0,a]})
    		\to \Hom(F_{(0,a)},T_{b-d}F^H_{[0,a]})), \\
    		B_{c} & \coloneqq \Ker (\Hom(F_{(0,a)},T_{b-d'}F^H_{[0,a]})
    		\to \Hom(F_{(0,a)},T_{b-d}F^H_{[0,a]})).
		\end{align}
		Then there exists a short exact sequence of right $\End( G_{(0,a)})$-modules
		\begin{equation}\label{eq:w-exact-seq}
		    0 \to A_{c} \to W_{c} \to B_{c} \to 0.
		\end{equation}

		\item Assume that
		$\pi (\MS_+(\cH_b))$
		is a discrete set and let
		\begin{equation}
    		\pi (\MS_+(\cH_b)) \cap [-a,0) 
    		= \{c_1,\dots, c_n\}
		\end{equation}
		with $c_1 < \dots < c_n$.
		Take $d_1,\dots,d_{n-1} \in \bR$ satisfying
		$
		c_1 < d_1 < c_2 < \dots < d_{n-1} < c_n 
		$
		and set $d_0=-a, d_n=0$.
		Define
		\begin{equation}
	    	V_d \coloneqq \Image (\Hom (F_{(0,a)},T_bF^H_{[0,a]}) \to \Hom (F_{(0,a)},T_{b-d}F^H_{[0,a]}))
		\end{equation}
		for $d\in [-a, 0]$.
		Then for any $i=0,\dots,n$
		there exists a submodule $\tl{B}_{c_{i}}$ of $ B_{c_i}$
		and a short exact sequence of right $\End( G_{(0,a)})$-modules
		\begin{equation}\label{equation:microstalk-exact-seq}
		0\to \tl{B}_{c_{i}} \to V_{d_i} \to V_{d_{i-1}} \to 0.
		\end{equation}
        Moreover, $V_{d_0} \simeq 0$.
        
		\item Assume that
		$\pi (\MS_+(\cH_b))$
		is a discrete set and let
		\begin{equation}
    		\pi (\MS_+(\cH_b)) \cap [0,a)
    		= \{c_{n+1},\dots, c_{n+m}\}
		\end{equation}
		with $c_{n+1} < \dots < c_{n+m}$.
		Take $d_{n+1},\dots,d_{n+m-1} \in \bR$ satisfying
		$
    		c_{n+1} < d_{n+1} < c_{n+2} < \dots < d_{n+m-1} < c_{n+m}
		$
		and set $d_n=0, d_{n+m}=a$.
		Define
		\begin{equation}
    		V_d \coloneqq \Image (\Hom(F_{(0,a)},T_{b-d}F^H_{[0,a]}) \to \Hom  (F_{(0,a)},T_{b}F^H_{[0,a]}))
    	\end{equation}
		for $d\in [0,a]$.
		Then for any $i=n+1,\dots,n+m$
		there exists a quotient module $\tl{A}_{c_i}$ of $A_{c_i}$
		and a short exact sequence of right $\End( G_{(0,a)})$-modules
		\begin{equation}\label{equation:microstalk-exact-seq2}
    		0 \to V_{d_i} \to V_{d_{i-1}} \to \tl{A}_{c_{i}}\to 0.
		\end{equation}
      Moreover, $V_{d_{n+m}} \simeq 0$.
	\end{enumerate}
\end{proposition}

\begin{proof}
    (i) 
	Consider the exact triangle 
	\begin{equation}
    	\RG_{[d',+\infty)}(\bR;\cH_b) \to \RG_{[d,+\infty)}(\bR;\cH_b) \to \RG_{[d,d')}(\bR;\cH_b) \toone, 
	\end{equation}
	where the third object is isomorphic to $\RG_{[c,+\infty)}(\cH_b)_c$ by the microlocal Morse lemma (\cref{proposition:microlocalmorse}).
	Note that in the triangulated orbit category $\Db_{/[1]}(\bfk)$ the shift functor $[1]$ is naturally isomorphic to the identity functor. 
	Hence, applying the cohomological functor $H^*=\Hom_{\Db_{/[1]}(\bfk)}(\bfk,-)$ gives the long exact sequence 
	\begin{equation}
	\begin{aligned}
	    \xymatrix{
	        & & \cdots \ar[r] & W_c \ar `[rd] `[l] `[dlll] `^r[dll] [dll] & \\
	        & \Hom(F_{(0,a)},T_{b-d'}F^H_{[0,a]}) \ar[r]
    	    & \Hom(F_{(0,a)},T_{b-d}F^H_{[0,a]}) \ar[r]
        	& W_{c} \ar `[rd] `[l] `[dlll] `^r[dll] [dll] & \\
        	& \Hom(F_{(0,a)},T_{b-d'}F^H_{[0,a]}) \ar[r]
    	    & \Hom(F_{(0,a)},T_{b-d}F^H_{[0,a]}) \ar[r]
    	    & \cdots, &
	    }
	\end{aligned}
	\end{equation}
	which induces the short exact sequence~\eqref{eq:w-exact-seq} of $\bfk$-vector spaces.
	Through the natural ring homomorphism
	\begin{equation}
    	\End(G_{(0,a)})^{\op}\ \xrightarrow{{j_a}_!} \End(F_{(0,a)})^{\op} \to \End(\cH_b),
	\end{equation}
	we find that the exact sequence is that of right $\End(G_{(0,a)})$-modules.

    \noindent (ii)(iii)
    Defining $\tl{B}_{c_i} \coloneqq B_{c_i}\cap V_{d_i}$, we obtain the exact sequence \eqref{equation:microstalk-exact-seq} of $\bfk$-vector spaces.
    
    The induced morphism $A_{c_i}\to \Coker(V_{d_i}\to V_{d_{i-1}})$ is surjective and $\Coker(V_{d_i}\to V_{d_{i-1}})$ is isomorphic to a quotient module $\tl{A}_{c_i}$ of $A_{c_i}$.
    This gives the exact sequence \eqref{equation:microstalk-exact-seq2} of $\bfk$-vector spaces.

	The constructions above are natural with respect to $\cH_b$ and the exact sequences are those of right $\End(G_{(0,a)})$-modules.
	
	Let us prove $V_{d_0}\simeq 0$. 
	The proof for $V_{d_{n+m}} \simeq 0$ is similar.
    Since $\tau_{0,a}(F_{(0,a)})={j_a}_!\tau_{0,a}(G_{(0,a)})$, it is enough to show $\tau_{0,a}(G_{(0,a)})=0$.
    Applying the microlocal Morse lemma (\cref{proposition:microlocalmorse}) to $\cHom^\star(F_{(0,a)},F_{[0,a]})$, for a sufficiently small $\varepsilon>0$, we find that 
    \begin{equation}
	\begin{aligned}
	    & \Image (\Hom_{\cD^{(0,a)}(M)_\theta} (G_{(0,a)},G_{(0,a)})\to \Hom_{\cD^{(0,a)}(M)_\theta} (G_{(0,a)},T_{a}G_{(0,a)})) \\
	    \simeq \
	    & \Image (\Hom_{\cD^{\bR}(M)_\theta} (F_{(0,a)},F_{[0,a]})\to \Hom_{\cD^{\bR}(M)_\theta} (F_{(0,a)},T_{a}F_{[0,a]})) \\
	    \simeq \
	    & \Image (\Hom_{\cD^{\bR}(M)_\theta} (F_{(0,a)},F_{[0,a]})\to \Hom_{\cD^{\bR}(M)_\theta} (F_{(0,a)},T_{a+\varepsilon}F_{[0,a]}))\\
	    \simeq \
	    & \Image (\Hom_{\cD^{(0,a)}(M)_\theta} (G_{(0,a)},G_{(0,a)})\to \Hom_{\cD^{(0,a)}(M)_\theta} (G_{(0,a)},T_{a+\varepsilon}G_{(0,a)})).
	\end{aligned}
	\end{equation}
    Therefore the result follows from the assertion $d_{\cD^{(0,a)}(M)_\theta}(G_{(0,a)},0)\leq a$ in \cref{theorem:existence-quantization}.
\end{proof}

\subsubsection{Study of \texorpdfstring{$\mu hom$}{mu hom} between sheaf quantizations}

In this subsection, we compute $\End(G_{(0,a)})$ and $W_c$ using $\mu hom$ functor.

\begin{lemma}\label{lemma:muhom-isom}
	\begin{enumerate}
		\item For $c \in \bR$, there is an isomorphism
		\begin{equation}
		\begin{aligned}
		    W_c
    		& \simeq
	    	H^*\RG (\Omega_+; \mu hom(F_{(0,a)}, {T_{b-c}}F^H_{[0,a]})) \\
	    	& =
	    	\Hom_{\Omega_+}^\mu(F_{(0,a)}, {T_{b-c}}F^H_{[0,a]})).
		\end{aligned}
		\end{equation}
		\item For $c \in \bR$, 
		\begin{equation}
		\Supp(\mu hom(F_{(0,a)}, {T_{c}}F^H_{[0,a]})|_{\Omega_+})
		\subset C_1(a,c) \cup \overline{C_2(a,c)},
		\end{equation}
		where $C_1(a,c) \coloneqq \bigcup_{g(y,y')=[c]} l(y)\boxplus{\wh{\bfc}(a)}$
		and $C_2(a,c) \coloneqq \bigcup_{g(y,y')+[a]=[c]}l(y)\boxplus{\wh{\bfl}(a)}$. 
	\end{enumerate}
\end{lemma}

\begin{proof}
	(i) The proof is essentially the same as that of \cite[\S4.3]{Ike19}. 
	The only and slight difference appears in checking that $\delta \colon (M \times \bR) \times \Stheta \times \Stheta \to (M \times \bR) \times (M \times \bR) \times \Stheta \times \Stheta$ is non-characteristic, which can also be verified easily.

	\noindent (ii) 
	By \cref{proposition:muhomss,lemma:SSF0a}, we get 
	\begin{equation}
    \begin{aligned}
        \Supp(\mu hom(F_{(0,a)}, {T_{c}}F^H_{[0,a]})|_{\Omega_+}) 
        & \subset \MS_+(F_{(0,a)})\cap T'_{c}(\MS_+(F^H_{[0,a]})) \\
        & \subset \Lambda \boxplus\wh{\bfd}(a) \cap T'_{c}(\Lambda^H\boxplus\wh{\bfq}(a)),
    \end{aligned}
	\end{equation}
    where $T_{c}' \colon \Omega_{+}\to\Omega_{+}$ is the lift of $T_{c}$. 
	The equality $\Lambda \boxplus\wh{\bfd}(a) \cap T'_{c}(\Lambda^H\boxplus\wh{\bfq}(a))= C_1(a,c) \cup \overline{C_2(a,c)}$ has been checked in the proof of \cref{proposition:interleavingF0a}(ii).
\end{proof}

\begin{proposition}\label{proposition:muhom-self}
	\begin{enumerate}
		\item The object $F_{(0,a)}$ (resp.\ $F_{[0,a]}$) is simple along $\Lambda \boxplus ({\wh{\bfd}(a)}\setminus {\wh{\bfc}(a)})$ (resp.\ $\Lambda \boxplus ({\wh{\bfq}(a)}\setminus {\wh{\bfc}(a)})$). 
		
		\item There exists an isomorphism $\mu hom(F_{(0,a)},F_{(0,a)})|_{\Omega_+}\simeq \bfk_{\Lambda \boxplus {\wh{\bfd}(a)}}$ such that the diagram 
				\begin{equation}
		\begin{aligned}
		\xymatrix{
			\bfk_{\Omega_+}\ar[r] \ar[rd]_-{\id^\mu_{F_{(0,a)}}}& \bfk_{\Lambda \boxplus {\wh{\bfd}(a)}} \ar[d]^-{\rotatebox{90}{$\sim$}} \\
			&\mu hom(F_{(0,a)},F_{(0,a)})|_{\Omega_+}
		}
		\end{aligned}
		\end{equation}
		commutes. 
		In particular, $\End^\mu_{\Omega_{+}}(F_{(0,a)})\simeq H^*(L)=\bigoplus_{i \in \bZ} H^i(L;\bfk)$ as $\bfk$-vector spaces. 
		Moreover, $\circ_{F_{(0,a)},F_{(0,a)},F_{(0,a)}}^\mu$ induces the cup product on $H^*(L)$ through this isomorphism.
		Similarly, there is an isomorphism $\mu hom(F_{[0,a]},F_{[0,a]})|_{\Omega_+}\simeq \bfk_{\Lambda \boxplus {\wh{\bfq}(a)}}$.
		
		\item There exists an isomorphism $\mu hom(F_{(0,a)},F_{[0,a]})|_{\Omega_+}\simeq \bfk_{\Lambda \boxplus {\wh{\bfc}(a)}}$ such that the diagram
		\begin{equation}
		\begin{aligned}
		\xymatrix@C=100pt{
			\bfk_{\Lambda \boxplus {\wh{\bfd}(a)}} \ar[d]_-{\rotatebox{90}{$\sim$}} \ar[r]& \bfk_{\Lambda \boxplus {\wh{\bfc}(a)}}\ar[d]^-{\rotatebox{90}{$\sim$}}\\
			\mu hom(F_{(0,a)},F_{(0,a)})|_{\Omega_+}\ar[r]^-{\mu hom(F_{(0,a)}, \psi)|_{\Omega_+}}&\mu hom(F_{(0,a)},F_{[0,a]})|_{\Omega_+}
		}
		\end{aligned}
		\end{equation}
		commutes, where $\psi \colon F_{(0,a)}\to F_{[0,a]}$ is the canonical morphism and the left vertical arrow is the isomorphism given in (ii). 
		In particular, $\Hom^\mu_{\Omega_{+}}(F_{(0,a)},F_{[0,a]})\simeq H^*(L)$ as $\bfk$-vector spaces. 
		Moreover, $\circ_{F_{(0,a)},F_{(0,a)},F_{[0,a]}}^\mu$ induces the usual right $H^*(L)$-module structure on $H^*(L)$ through these isomorphisms. 
	\end{enumerate}
\end{proposition}
\begin{proof}
    (i) It follows from (ii) and \cref{theorem:existence-quantization}(4). 
    
	\noindent (ii) 
	Since $\Supp(\mu hom(F_{(0,a)},F_{(0,a)})|_{\Omega_{+}}) \subset \Lambda \boxplus {\wh{\bfd}(a)}$ by \cref{proposition:muhomss} and \cref{lemma:SSF0a}, 
	the morphism $\id_{F_{(0,a)}}^\mu\colon \bfk_{\Omega_{+}}\to \mu hom(F_{(0,a)},F_{(0,a)})|_{\Omega_{+}}$ factors through $\bfk_{\Lambda \boxplus {\wh{\bfd}(a)}}$. 
	We define $\cE$ to be the cone of the morphism $\bfk_{\Lambda \boxplus {\wh{\bfd}(a)}}\to \mu hom(F_{(0,a)},F_{(0,a)})|_{\Omega_+}$. 
	We will show that $\cE \simeq 0$.
	
	Again by \cref{proposition:muhomss} and \cref{lemma:SSF0a},
	setting $C_{\wh{\bfd},\wh{\bfd}} \coloneqq -\bfh^{-1}(C(\Lambda\boxplus\wh{\bfd}(a),\Lambda \boxplus\wh{\bfd}(a)))$ we have 
	\[\MS(\mu hom(F_{(0,a)},F_{(0,a)})|_{\Omega_{+}})\subset C_{\wh{\bfd},\wh{\bfd}}.\] 
    Since $\MS (\bfk_{\Lambda \boxplus {\wh{\bfd}(a)}}) \subset C_{\wh{\bfd},\wh{\bfd}}$, by the triangle inequality (\cref{proposition:properties-ms}(ii)), 
	\begin{equation}
	    \MS(\cE) 
	    \subset C_{\wh{\bfd},\wh{\bfd}}
    	\subset 
	    -\bfh^{-1}((d\rho)^{-1}C (\bfd(a),\bfd(a))). 
	\end{equation}

	We decompose $\wh{\bfd}(a)$ into the following nine parts 
	\begin{alignat*}{2}
	D_1& \coloneqq \{(a,0;\upsilon,\tau)\mid \upsilon>0 \}, & \quad  D_2& \coloneqq \{(a,0;0,\tau) \},\\ D_3& \coloneqq \{(u,0;0,\tau) \mid 0<u<a \}, & D_4& \coloneqq \{(0,0;0,\tau) \}, \\
	D_5& \coloneqq \{(0,0; \upsilon,\tau)\mid -\tau<\upsilon <0 \}, & D_6& \coloneqq \{(0,0; -\tau,\tau) \},\\
	D_7& \coloneqq \{(u,[u]; -\tau,\tau)\mid 0<u<a\}, & D_8& \coloneqq \{(a,[a]; -\tau,\tau)\},\\
	D_9& \coloneqq \{(a,[a];\upsilon,\tau)\mid \upsilon>-\tau \}.&
	\end{alignat*}
	Let $p\colon \Lambda\boxplus \wh{\bfd}(a) \to\Lambda$ be the unique continuous map satisfying $p(x,u,t;\xi,\upsilon,\tau)=(x,t';\xi,\tau)$ for any $(x,u,t;\xi,\upsilon,\tau)\in \Lambda\boxplus \wh{\bfd}(a)$ and some $t'$. 
	Define $\Lambda_i \coloneqq \Lambda\boxplus{D_i}$ for $i=1,\dots, 9$.
	Let $p_i$ be the projection $\Lambda_i\to \Lambda$ that  is the restriction of $p$.
    For even $i$, $p_i$ is bijective.
	For odd $i$, we define $q_i$ that is an extension of $p_i$ by  
	\begin{align*}
	q_1 & \colon \Omega_1  \coloneqq  \{u=a,\upsilon>0 \}_{\Omega_+} \to \Omega_{+}(M)_\theta, (x,a,t;\xi,\upsilon,\tau)\mapsto (x, t;\xi,\tau), \\
	q_3&\colon \Omega_3 \coloneqq  \{0<u<a, \upsilon =0\}_{\Omega_{+}}\to \Omega_{+}(M)_\theta, (x,u,t;\xi,0,\tau)\mapsto (x, t;\xi,\tau),\\
	q_5&\colon \Omega_5 \coloneqq   \{u=0, -\tau <\upsilon <0\}_{\Omega_{+}}\to \Omega_{+}(M)_\theta, (x,0,t;\xi,\upsilon,\tau)\mapsto  (x, t;\xi,\tau) \\
	q_7&\colon \Omega_7 \coloneqq \{0<u<a, \upsilon =-\tau\}_{\Omega_{+}}\to \Omega_{+}(M)_\theta,(x,u,t;\xi,-\tau ,\tau)\mapsto (x,t-u;\xi,\tau),\\
	q_9&\colon \Omega_9 \coloneqq \{u=a, -\tau<\upsilon \}_{\Omega_{+}}\to \Omega_{+}(M)_\theta, (x,a,t;\xi,\upsilon ,\tau)\mapsto (x,t-a;\xi,\tau).
	\end{align*}
	The image of $(q_i)_d$ contains $\MS(\cE|_{\Omega_i})$ for each odd $i$. 
	Here we show it in the case $i=7$, for example. 
	We denote by $(x,u,t,\xi,\upsilon, \tau; \momx, \momu, \momt, \momxi,\momupsilon,\momtau)$ a homogeneous coordinate system of $T^*\Omega_+$.
	Let $i_7\colon \Omega_7\to \Omega_+$ be the inclusion. It is enough to check $\MS(\cE)\cap T^*\Omega_+|_{\Omega_7}$ is contained in $((i_7)_d)^{-1}(\Image (q_7)_d)$. 
	A direct computation shows 
	\begin{equation}
	    ((i_7)_d)^{-1}(\Image (q_7)_d)=\{0<u<a, \upsilon=-\tau, \momu=-\momt \}_{T^*\Omega_+}.
	\end{equation}
	On the other hand, $C_{\wh{\bfd},\wh{\bfd}} \cap T^*\Omega_+|_{\Omega_7}$ is the conormal bundle of $\Lambda_7$ and hence contained in $\{\momu=-\momt \}_{T^*\Omega_{+}}$. Hence the image of $(q_7)_d$ contains $\MS(\cE|_{\Omega_7})$.
	
	For each odd $i$, by \cref{proposition:SSpushpull}(iii) and \cref{theorem:existence-quantization}(4), there exists an $E_i \in \Db_{/[1]}(\Omega_{+}(M)_\theta)$ satisfying $\Supp(E_i)\subset \Lambda$ and $\cE|_{\Omega_i}\simeq q_i^{-1}E_i$. 
	We also define $E_i \coloneqq \cE|_{\Omega_{+}(M)_\theta \boxplus D_i}$ for even $i$.
	By \cref{theorem:existence-quantization}(2), we have $\cE|_{\Omega_{+}(M)_\theta \boxplus D_3} \simeq 0$. 
	On a neighborhood of $\Lambda_2$, the set $C_{\wh{\bfd},\wh{\bfd}}$ does not intersect $\{ \momu \momupsilon <0\}_{T^*\Omega_{+}}$.
	Using \cref{lemma:interiorboundary}(ii) for $\phi=\pm ( u+\frac{\upsilon}{\tau}-a)$, we find that $E_1\simeq E_2\simeq E_3$ and $\cE$ is of the form $p^{-1}E_1$ on this neighborhood.
	By similar arguments for $\Lambda_4,\Lambda_6$ and $\Lambda_8$, we get $\cE|_{\Lambda\boxplus \widehat{\bfd}(a)} \simeq p^{-1}(E_3|_\Lambda) \simeq 0$, which proves the first assertion.
	
	The last assertion is proved in a parallel way. 
	
	Let us prove the second assertion. 
	Denote by 
	\[
	v\in \Hom \left(\left(\bfk_{\Lambda \boxplus {\wh{\bfd}(a)}}\right)^{\otimes 2}, \bfk_{\Lambda \boxplus {\wh{\bfd}(a)}}\right) 
	\]
	the morphism corresponding to $\circ_{F_{(0,a)},F_{(0,a)},F_{(0,a)}}^\mu$ through the isomorphism proved above and by $v' \colon H^*(L)^{\otimes 2} \to H^*(L)$ the induced morphism.
	Consider the canonical isomorphism $\zeta \colon \bfk_{\Lambda \boxplus {\wh{\bfd}(a)}}^{\otimes 2} \simto \bfk_{\Lambda \boxplus {\wh{\bfd}(a)}}$ that induces the cup product $\cup \colon H^*(L)^{\otimes 2} \to H^*(L)$.
	The morphism $w \colon H^*(L) \to H^*(L)$ corresponding to $v \circ \zeta^{-1}$ satisfies $w(\beta)= w(1) \cup \beta$ for any $\beta \in H^*(L)$.
	By construction, $v'(\alpha_1 \otimes \alpha_2)=w(\alpha_1 \cup \alpha_2)=w(1) \cup \alpha_1 \cup \alpha_2$ for any $\alpha_1, \alpha_2 \in H^*(L)$, which also shows $w(1)=v'(1 \otimes 1)$.
	The morphism $\id_{F_{(0,a)}}^\mu$ corresponds to $1 \in H^*(L)$ and hence by the unitality $v'(1\otimes 1)=1$, which proves the result.
    
	\noindent (iii) The morphism $\mu hom(F_{(0,a)}, \psi)|_{\Omega_+}$ factors through $\bfk_{\Lambda \boxplus {\wh{\bfc}(a)}}$. Since $\psi|_{\{u<a\}_{\Omega_+}}$ is an isomorphism, $\bfk_{\Lambda \boxplus {\wh{\bfc}(a)}}\to \mu hom(F_{(0,a)}, F_{[0,a]})$ is also isomorphic on $\{u<a\}_{\Omega_{+}}$. 
	The cone $\cE'$ of $\bfk_{\Lambda \boxplus {\wh{\bfc}(a)}}\to \mu hom(F_{(0,a)}, F_{[0,a]})|_{\Omega_+}$ is supported on $\{u=a\}_{\Omega_{+}}$. Since the microsupports of both $\bfk_{\Lambda \boxplus {\wh{\bfc}(a)}}$ and $\mu hom(F_{(0,a)}, F_{[0,a]})|_{\Omega_+}$ are contained in $-\bfh^{-1}C(\Lambda\boxplus{\wh{\bfq}(a)},\Lambda\boxplus{\wh{\bfd}(a)})$, 
	$\MS(\cE')$ does not intersect $\{\momu>0 \}_{T^*\Omega_{+}}$. These two properties require $\cE'\simeq 0$. 
	
	The composition morphism 
	\begin{equation*}
		\mu hom(F_{(0,a)},F_{[0,a]})|_{\Omega_{+}}\otimes \mu hom(F_{(0,a)},F_{(0,a)})|_{\Omega_{+}}\to\mu hom(F_{(0,a)},F_{[0,a]})|_{\Omega_{+}}
	\end{equation*}
	is also determined by the unitality as in (ii). 
\end{proof}

\begin{proposition}\label{proposition:EndG}
    There is an isomorphism of rings 
    \begin{equation}
        \End( G_{(0,a)}) \simeq H^*(L)=\bigoplus_{i \in \bZ}H^i(L; \bfk).
    \end{equation}
\end{proposition}

\begin{proof}
    By the functoriality of $m_{-,-} \colon \Db_{/[1]}(X;\Omega) \to \Dmu_{/[1]}(X;\Omega)$ (see \cref{definition:cat-muhom}(ii)) and \cref{proposition:muhom-self}, we obtain the ring homomorphism 
	\begin{equation}
	\End (G_{(0,a)}) \xrightarrow{{j_a}_!} \End(F_{(0,a)})\xrightarrow{m_{F_{(0,a)},F_{(0,a)}}} \End^\mu_{\Omega_+}(F_{(0,a)})\simeq H^*(L).
	\end{equation}
	We check that this morphism is an isomorphism of modules. 
	
	For any $0< \varepsilon <r(\iota)-a$, there is an exact triangle of $\End (G_{(0,a)})$-modules 
	\begin{equation}\label{equation:exact-seq-ring-isom}
	\begin{aligned}
	    & \Hom (F_{(0,a)},T_{-\varepsilon}F_{[0,a]})\to \Hom (F_{(0,a)},F_{[0,a]}) \\ 
	    & \hspace{20pt} \to H^*\RG_{[0,+\infty)}(\ell^!Rq_*\cHom^\star(F_{(0,a)},F_{[0,a]}))_{0} \toone
	\end{aligned}
	\end{equation}
	The second module is isomorphic to $\End( G_{(0,a)})$.
	Moreover, the third one is isomorphic to $\Hom^\mu_{\Omega_+}(F_{(0,a)},F_{[0,a]})$ by \cref{lemma:muhom-isom}(i) (the case $H \equiv 0$ and $c=b$), which is isomorphic to $H^*(L)$ by \cref{proposition:muhom-self}(iii). 
	Thus by the commutativity of the following diagram, it is enough to prove the first module in \eqref{equation:exact-seq-ring-isom} is $0$:
	\begin{equation}\label{equation:diagram-ring-isom}
	\begin{aligned}
	\xymatrix{
		\End(G_{(0,a)}) \ar[d]_-{{j_a}_!} \ar[rd]&\\
		\End(F_{(0,a)}) \ar[d]_-{m_{F_{(0,a)},F_{(0,a)}}}\ar[r]^-{\psi \circ-} & \Hom (F_{(0,a)},F_{[0,a]}) \ar[d]^-{m_{F_{(0,a)},F_{[0,a]}}} \\
		\End^\mu_{\Omega_+}(F_{(0,a)}) \ar[r] \ar[d] & \Hom^\mu_{\Omega_{+}}(F_{(0,a)},F_{[0,a]}) \ar[d] \\
		H^*(L)\ar[r]&H^*(L),
	}
	\end{aligned}
	\end{equation}
	where $\psi \colon F_{(0,a)} \to F_{[0,a]}$ is the canonical morphism.
	All the arrows in the diagram are morphisms of right $\End(G_{(0,a)})$-modules and the three arrows in the left column are ring homomorphisms. 
	Note that the unlabeled arrows are all isomorphisms. 
	
	If $a <r(\iota)/2$, we can choose $0<\varepsilon_1, \varepsilon_2<r(\iota)-a$ so that $\varepsilon_2-\varepsilon_1>a$. 
	The isomorphism $\Hom (F_{(0,a)},T_{-\varepsilon_2}F_{[0,a]})\to \Hom (F_{(0,a)},T_{-\varepsilon_1}F_{[0,a]})$ is induced by $\tau_{-\varepsilon_2, -\varepsilon_1}(G_{(0,a)})$, which is the zero morphism since $d_{\cD^{(0,a)}(M)_\theta}(G_{(0,a)},0)\leq a$ by \cref{theorem:existence-quantization}.
	
	Now assume $a>r(\iota)/2$. 
	In this case, we take $a<\tilde{a}<r(\iota)$ and construct an object $\cH \in \cD^{(0, \tilde{a})}(\pt)_\theta$ such that $\cH|_{\Stheta \times \{u\}}\simeq {Rq}_*\cHom^\star(F_{(0,u)}, F_{[0,u]})$ for any $u\in(0,\tilde{a})$ as follows.
    Let $D_{\tilde{a}}$ and $\cG \in \cD^{D_{\tilde{a}}}_L(M)_\theta$ be as in the proof of \cref{theorem:existence-quantization}.
	Set 
	\begin{equation}
	\begin{aligned}
	    & \cF_{1} \coloneqq Rj_!\cG, \ 
	    \cF_{2} \coloneqq Rj_*\cG \in \cD^{\bR \times (0,\tilde{a})}(M)_\theta, \\ 
	    & \cH \coloneqq Rq'_* \cHom^\star (\cF_{1},\cF_{2}) \in \cD^{ (0,\tilde{a})}(\pt )_\theta,
	\end{aligned}
	\end{equation} 
	where $j\colon M \times D_{\tilde{a}}\times \Stheta\to M\times \bR\times (0,\tilde{a})\times \Stheta$ is the inclusion and $q' \colon M\times \bR\times (0,\tilde{a})\times \Stheta \to  (0,\tilde{a})\times \Stheta$ is the projection. 
	Then the object $\cH|_{(0,\tilde{a}) \times \{\varepsilon\}}$ is locally constant on $(0,\tilde{a})$ for $0< \varepsilon <r(\iota)-\tilde{a}$. 
	This shows that $\Hom (F_{(0,u)},T_{-\varepsilon}F_{[0,u]})$ does not depend on $u\in(0,\tilde{a})$ and the result follows from the case $a<r(\iota)/2$.
\end{proof}

By \cref{proposition:EndG}, the modules $W_c, A_c, B_c$, etc.\ that appeared in \cref{proposition:exact-triagle} are equipped with a $H^*(L)$-action.

\subsubsection{Betti number estimate: Proof of \texorpdfstring{\cref{theorem:betti-estimate}}{Theorem~5.4}}\label{subsubsec:betti}

In this subsection, in order to prove \cref{theorem:betti-estimate} for a strong rational Lagrangian immersion satisfying \cref{assumption:special-stongly-rational}, we assume the following.

\begin{assumption}
	The strongly rational Lagrangian immersion $\iota \colon L \to T^*M$ and
	the Hamiltonian function $H \colon T^*M \times I \to \bR$ satisfy the following conditions:
	\begin{enumerate}
		\renewcommand{\labelenumi}{$\mathrm{(\arabic{enumi})}$}
		\item $\| H\| < r(\iota)$, 
		\item $\iota$ and $\iota^H=\phi^H_1 \circ \iota$ intersect transversally.
	\end{enumerate}
\end{assumption}

Under the assumption, $\pi (\MS_+(\cH_b))$ is discrete by \cref{proposition:interleavingF0a}(ii).

\begin{lemma}\label{lemma:betti-morse-ineq}
    With the notation of \cref{proposition:exact-triagle}, for $t \in \Stheta$, take any $\tilde{t}\in \ell^{-1}(t)$ and set $W_{t} \coloneqq W_{\tilde{t}}, A_{t} \coloneqq A_{\tilde{t}}$, and $B_{t} \coloneqq B_{\tilde{t}}$, which do not depend on the choice of $\tl{t}$.
    Let $c_1,\dots, c_{n+m}$ be as in \cref{proposition:exact-triagle}(ii) and (iii).
	Then
	\begin{equation}
	\begin{aligned}
		& \sum_{t\in \Stheta} \dim B_{t}
		\ge
		\sum_{i=1}^{n} \dim B_{c_i}
		\ge
		\dim H^*(L), \\
		& \sum_{t\in {\Stheta}} \dim A_{t}
		\ge
		\sum_{i=n+1}^{n+m} \dim A_{c_i}
		\ge
		\dim H^*(L).
	\end{aligned}
	\end{equation}
	In particular,
	\begin{equation}
	\sum_{t\in \Stheta} \dim W_{t}
	\ge
	2\dim H^*(L).
	\end{equation}
\end{lemma}

\begin{proof}
	Since the composite \eqref{equation:factorization} is an isomorphism and $\Hom (F_{(0,a)},F_{[0,a]})\simeq H^*(L)$ by \cref{proposition:EndG}, we have
	\begin{equation}
	\dim \Hom (F_{(0,a)},T_b F^H_{[0,a]}) \ge \dim H^*(L).
	\end{equation}
	By \cref{proposition:exact-triagle}(ii), noticing that $V_{d_0} \simeq 0$ and
	$V_{d_n} \simeq \Hom (F_{(0,a)},T_b F^H_{[0,a]})$,
	we get
	\begin{equation}
	\sum_{i=1}^{n} \dim B_{c_i}
	\ge
	\sum_{i=1}^{n} \dim \tl{B}_{c_i}=
	\dim \Hom(F_{(0,a)},T_b F^H_{[0,a]}).
	\end{equation}
	Similarly by \cref{proposition:exact-triagle}(iii), noticing that $V_{d_{n+m}} \simeq 0$ and
	$V_{d_n} \simeq \Hom (F_{(0,a)},T_b F^H_{[0,a]})$,
	we get
	\begin{equation}
	\sum_{i=n+1}^{n+m} \dim A_{c_i}
	\ge
	\dim \Hom(F_{(0,a)},T_b F^H_{[0,a]}).
	\end{equation}
	Moreover, by \cref{proposition:exact-triagle}(i),
	$\dim B_{t}+\dim A_{t}=\dim W_t$.
	Combining these inequalities, we obtain the result.
\end{proof}

\begin{proposition}\label{proposition:transmuhom}
    For $c \in \bR$, there is an isomorphism
	 \begin{equation}
	 	\mu hom(F_{(0,a)}, {T_c}F^H_{[0,a]})|_{\Omega_+} \simeq
	 	\bigoplus_{g(y,y')=[c]} \bfk_{l(y)\boxplus{\wh{\bfc}(a)}}\oplus \bigoplus_{g(y,y')+[a]=[c]}\bfk_{l(y)\boxplus{\wh{\bfl}(a)} }.
	 \end{equation}
\end{proposition}

\begin{proof}
    We argue similarly to \cref{proposition:muhom-self} and use the same notation as in its proof and \cref{lemma:muhom-isom}.
	Moreover, we set $\cF \coloneqq \mu hom(F_{(0,a)},F^H_{[0,a]})|_{\Omega_{+}} \in \Db_{/[1]}(\Omega_+), D_{10} \coloneqq \{(a,[a];0,\tau)\}, D_{11} \coloneqq \wh{\bfl}(a)$, and 
	\begin{equation}
	\begin{aligned}
	    q_{11}\colon \Omega_{11} \coloneqq \{u=a, -\tau<\upsilon<0 \}_{\Omega_+}\to \Omega_{+}(M)_\theta, (x,a,t;\xi,\upsilon, \tau) \mapsto (x,t-a;\xi,\tau).
	\end{aligned}
	\end{equation}
	
	By \cref{lemma:muhom-isom}(ii), $\Supp(\cF)\subset C_1(a,c) \cup \overline{C_2(a,c)}$. 
	Since $C_1(a,c)$ and $\overline{C_2(a,c)}$ are disjoint, $\cF$ admits a direct sum decomposition of the form $\cF\simeq \cF'\oplus\cF''$ with $\Supp(\cF')\subset C_1(a,c)$ and $\cF''\subset \overline{C_2(a,c)}$. 
	By \cref{proposition:muhomss} and \cref{lemma:SSF0a}, we have 
	\begin{equation}
	    \MS(\cF)\subset -\bfh^{-1}(C(\Lambda\boxplus\wh{\bfd}(a),\Lambda^H\boxplus\wh{\bfq}(a)))
	\subset 
	-\bfh^{-1}((d\rho)^{-1}C (\bfd(a),\bfq(a))).
	\end{equation}
	
	The image of $(q_i)_d$ contains $\MS(\cF'|_{\Omega_i})$ for each odd $i$. 
	Therefore, by \cref{proposition:SSpushpull}(iii), there exists a locally tame object $F_i' \in \Db_{/[1]}(\Omega_{+}(M)_\theta)$ with $\Supp(F_i')\subset \Lambda$ and $\cF'|_{\Omega_i}\simeq q_i^{-1}F_i'$ for any odd $i$. 
	We also define $F_i' \coloneqq \cF'|_{\Omega_{+}(M)_\theta \boxplus D_i}$ for even $i$.
	By \cref{theorem:existence-quantization,lemma:simpleclean-orbit}, we have $F_3'\simeq \bigoplus_{g(y,y')=c} \bfk_{l(y)\boxplus D_3}$. 
	On a neighborhood of $\Lambda_2$, the set $-\bfh^{-1}C(\Lambda\boxplus{\wh{\bfd}(a)},\Lambda\boxplus{\wh{\bfq}(a)})$ does not intersect $\{ \momu \momupsilon <0\}_{T^*\Omega_+}$.
	Using \cref{lemma:interiorboundary}(ii) for $\phi= u-\frac{\upsilon}{\tau}-a$, we find that $F_2'\simeq F_3'$ and $\cF'|_{\Lambda \boxplus \wh{\bfc}(a)}$ is of the form $p^{-1}F_3'$ on this neighborhood.
	By similar arguments for $\Lambda_4,\Lambda_6$ and $\Lambda_8$, we get $\cF'\simeq \bigoplus_{g(y,y')=[c]} \bfk_{l(y)\boxplus{\wh{\bfc}(a)}}$.
	
	By \cref{proposition:muhom-self}(i) and \cref{lemma:simpleclean-orbit}, $\cF''|_{\Omega_{11}}\simeq \bigoplus_{g(y,y')+[a]=[c]} \bfk_{l(y)\boxplus \wh{\bfl}(a)}|_{\Omega_{11}}$. 
	Using \cref{lemma:interiorboundary}(i) for $\phi=u-\frac{\upsilon}{\tau}-a,-(u+\frac{\upsilon}{\tau}-a)-1$, we obtain $\cF''\simeq \bigoplus_{g(y,y')+[a]=[c]} \bfk_{l(y)\boxplus \wh{\bfl}(a)}$. 
	
	Combining these isomorphisms, we obtain the result. 
\end{proof}

\begin{proof}[Proof of \cref{theorem:betti-estimate}]
    For $t \in \Stheta$, take $\tl{t} \in \ell^{-1}(t)$ as in \cref{lemma:betti-morse-ineq}.
	By \cref{lemma:muhom-isom}(i) and \cref{proposition:transmuhom}, 
	\begin{equation}
	\begin{aligned}
		\dim W_{t}
		= \ &
		\dim H^*\RG (\Omega_+; \mu hom(F_{(0,a)}, {T_{b-\tl{t}}}F^H_{[0,a]})) \\
		= \ &
		\# \left\{ (y,y') \in C(\iota,H) \mid g(y,y')=[b]-t \right\} \\
		& +
		\# \left\{(y,y') \in C(\iota,H) \mid g(y,y')+[a]=[b]-t \right\}
	\end{aligned}
	\end{equation}
	for any $t \in \Stheta$.
	Hence, we get
	$\sum_{t\in \Stheta} \dim W_{t}=2 \#C(\iota,H)$.
	Combining this equality with \cref{lemma:betti-morse-ineq},
	we obtain the theorem.
\end{proof}

\subsubsection{Cup-length estimate: Proof of \texorpdfstring{\cref{theorem:cuplength}}{Theorem 5.5}}\label{subsubsec:cuplength}

In this subsection, we give a proof of \cref{theorem:cuplength} for a strongly rational Lagrangian immersion satistying \cref{assumption:special-stongly-rational}.

First we introduce an algebraic counterpart of cup-length and study some properties.

\begin{definition}
	Let $R$ be an associative (not necessarily commutative nor unital) algebra over $\bfk$.
	For a right $R$-module $A$,
	define
	\begin{equation}
	\cl_R(A)
	 \coloneqq 
	\inf
	\left\{
	k-1
	\; \middle| \;
	\begin{aligned}
		& k \in \bZ_{\ge 0}, \ a_0 \cdot r_1\cdots r_k = 0 \\
		& \text{for any $(r_i)_{i} \in R^k$ and for any $a_0 \in A$}
	\end{aligned}	
	\right\}
	\in \bZ_{\geq -1} \cup \{\infty\}.
	\end{equation}
	Note that
	\begin{enumerate}
		\item $\cl_R(A)=-1$ if and only if $A=0$.
		\item $\cl_R(A)=0$ if and only if $A \neq 0$ and $ar=0$
		for any $a \in A$ and any $r \in R$.
	\end{enumerate}
	If there is no risk of confusion, we simply write $\cl(A)$ for $\cl_R(A)$.
\end{definition}

By definition, one can easily show the following two lemmas.

\begin{lemma}\label{lemma:cup-exact}
	For an exact sequence $A\to B\to C$ of right $R$-modules,
	$\cl(B)\leq \cl(A)+\cl(C)+1$.
\end{lemma}

\begin{lemma}\label{lemma:cup-factorize}
	Let $R, R'$ be rings and $A$ be a right $R$-module.
    If $R'$ is a non-zero unital ring and the action of $R$ on $A$ factors as $R \to R' \to \End(A)^{\op}$, then $\cl_{R}(A) \le \cl_{R}(R')$.
\end{lemma}

The usual notion of cup-length is related to the above definition as follows.
Let $X$ be a manifold.
We define the ring $R_X \coloneqq \bigoplus_{i\geq 1} H^i(X;\bfk)$ equipped with
the cup product and $\cl(X) \coloneqq \cl_{R_X}(H^*(X;\bfk))$.
The number $\cl(X) \in \bZ_{\geq -1} \cup \{\infty\}$ is called the \emph{cup-length} of $X$.

\medskip

Now we start the proof of \cref{theorem:cuplength}.
In what follows, we assume the following.

\begin{assumption}\label{assumption:cup-length}
	The Hamiltonian function $H$ satisfies $\| H\| < \min(r(\iota), \theta(\iota)/2)$.
\end{assumption}

Take $a \in \bR$ satisfying $\| H \|<a < \min(r(\iota), \theta(\iota)/2)$.
From now on, until the end of this subsection, set $R \coloneqq R_L=\bigoplus_{i \ge 1} H^i(L;\bfk)$.
Recall again that by \cref{proposition:EndG} the right $\End(G_{(0,a)})$-modules that appeared in \cref{proposition:exact-triagle} can be regarded as $H^*(L)$-modules and hence $R$-modules.

\begin{lemma}\label{lemma:cup-length-ineq}
	Assume that
	$\pi ( \MS_+(\cH_b))$ is a discrete set and let $c_1, \dots, c_n$ be as in \cref{proposition:exact-triagle}(ii).
	Then
	\begin{equation}
		n+
		\sum_{i=1}^{n} \cl(W_{c_i})
		\ge
		\cl(L)+1.
	\end{equation}
\end{lemma}

\begin{proof}
    First recall that the composite \eqref{equation:factorization} is an isomorphism and $\Hom (F_{(0,a)},F_{[0,a]}) \simeq H^*(L)$ as $R$-modules by \cref{proposition:EndG}.
	Hence we have
	\begin{equation}
	    \cl( \Hom (F_{(0,a)},T_b F^H_{[0,a]})) \ge \cl( H^*(L) ).
	\end{equation}
	Applying \cref{lemma:cup-exact} to the exact sequences~\eqref{equation:microstalk-exact-seq} and \eqref{eq:w-exact-seq} of right $R$-modules, we have
	\begin{equation}
		\cl(V_{d_i})
		\le
		\cl(W_{c_i}) + \cl(V_{d_{i-1}})+1.
	\end{equation}
	Noticing that $V_{d_0} \simeq 0$ and
	$V_{d_n} \simeq \Hom (F_{(0,a)},T_b F^H_{[0,a]})$,
	by induction we obtain
	\begin{equation}
	n+
	\sum_{i=1}^{n} \cl(W_{c_i})
	\ge
	\cl(\Hom (F_{(0,a)},T_b F^H_{[0,a]}))+1,
	\end{equation}
	which proves the result.
\end{proof}

It remains to see the action of $R$ on each $W_{c_i} \simeq \Hom_{\Omega_+}^\mu(F_{(0,a)},T_{b-c_i}F_{[0,a]}^H)$.
Recall that for any $c \in \bR$, 
\begin{equation}
    \Hom_{\rho^{-1}(U)}^\mu(F_{(0,a)},T_{c}F_{[0,a]}^H) \simeq H^*\RG(\rho^{-1}(U); \mu hom(F_{(0,a)}, {T_{c}} F^H_{[0,a]}))
\end{equation}
admits a right $\End(G_{(0,a)})$-module structure.
Hence it is equipped with a right $R$-module structure through the ring homomorphism $R \hookrightarrow H^*(L)\simeq \End(G_{(0,a)})$.

\begin{proposition}\label{proposition:muhom-cl}
	Let $U$ be an open subset of $T^*M$.
	Then for any $c \in \bR$,
	\begin{equation}\label{equation:muhom-cl}
		\cl \left(
			\Hom_{\rho^{-1}(U)}^\mu(F_{(0,a)},T_{c}F_{[0,a]}^H)
			\right)
		\le
		\cl (H^*(\iota^{-1}(U))).
	\end{equation}
\end{proposition}

\begin{proof}
	By the functoriality of $m_{-,-}$ (see \cref{definition:cat-muhom}(ii)), the action of $\End(F_{(0,a)})$ on $\Hom_{\rho^{-1}(U)}^\mu(F_{(0,a)},T_{c}F_{[0,a]}^H)$ factors through $\End_{\rho^{-1}(U)}^\mu(F_{(0,a)})$, and so does the action of $R$.
	The ring $\End_{\rho^{-1}(U)}^\mu(F_{(0,a)})$ is isomorphic to $H^*\RG(\rho^{-1}(U);\bfk_{\Lambda \boxplus d(a) \cap \rho^{-1}(U)})\simeq H^*(\iota^{-1} (U))$ by \cref{proposition:muhom-self}(ii). 
	Hence the assertion follows from \cref{lemma:cup-factorize} if $\iota^{-1}(U)$ is non-empty. 
	
	If $\iota^{-1}(U)$ is empty, $\Hom_{\rho^{-1}(U)}^\mu(F_{(0,a)},T_{c}F_{[0,a]}^H)$ is zero by \cref{lemma:muhom-isom}(ii). Hence both sides of \eqref{equation:muhom-cl} are $-1$.
\end{proof}

\begin{proof}[Proof of \cref{theorem:cuplength}]
	We may assume that $C(\iota,H)$ is discrete and let $c_1,\dots,c_n$ be as in \cref{lemma:cup-length-ineq}.
	Since $a<\theta/2$, for any $(y,y') \in C(\iota,H)$, the set
	\begin{equation}\label{equation:singleton}
		\left\{
		c \in \bR
		\; \middle| \;
		\begin{aligned}
		& \text{$g(y,y') \equiv -c+b \mod \theta$ or} \\
		& g(y,y') \equiv -c+b-a \mod \theta
		\end{aligned}
		\right\}
		\cap [-a,0)
	\end{equation}
	is a singleton or empty.
	Hence, we have $\# C(\iota,H) \ge n$.

	Let $c$ be any of $c_1,\dots,c_n$ and set 
	\begin{equation}
	    \{ (y_{1},y'_{1}), \dots, (y_{k},y'_{k}) \} 
	     \coloneqq \left\{ (y,y') \in C(\iota,H) 
	    \; \middle| \;
	    \begin{aligned}
	        & g(y,y') \equiv -c+b \mod \theta \\
	        & \text{or $g(y,y') \equiv -c+b-a \mod \theta$}
	    \end{aligned} 
	    \right\}. 
	\end{equation}
	Take a sufficiently small contractible open neighborhood $U_j$ of $\iota(y_j)=\iota^H(y'_j)$ in $T^*M$ for $j=1,\dots,k$ and set $U \coloneqq \bigcup_{j=1}^{k} U_j$.
	Then, by \cref{lemma:muhom-isom}, we obtain
	\begin{equation}
	\begin{aligned}
		W_{c}
		& \simeq
		H^*\RG(\Omega_+; \mu hom(F_{(0,a)}, {T_{b-c}}F^H_{[0,a]})) \\
		& \simeq
		H^*\RG(\rho^{-1}(U);\mu hom(F_{(0,a)},{T_{b-c}}F^H_{[0,a]})).
	\end{aligned}
	\end{equation}
	Therefore, by \cref{proposition:muhom-cl},
	we have $\cl(W_{c}) \le \cl(H^*(\iota^{-1}(U)))=0$, which proves the theorem by \cref{lemma:cup-length-ineq}.
\end{proof}

\begin{remark}
    The quantity $\cl+1$ in the proof of \cref{theorem:cuplength} and $\dim$ in the proof of \cref{theorem:betti-estimate} play similar but a bit different roles in the following sense. 
    For a short exact sequence $0 \to A \to B \to C \to 0$ of right $R$-modules, $\cl(B)+1$ can be strictly smaller than $\cl(A)+1+\cl(C)+1$ while $\dim B= \dim A+ \dim C$ always holds. 
    Because of this difference, the proof of \cref{lemma:betti-morse-ineq} with $\dim$ replaced by $\cl +1$ does not proceed in the same way.
\end{remark}

\begin{remark}
    If $\min \left( \{r(\iota)\} \cup (\{ \theta(\iota)/2 \}\cap \bR_{>0})\right)\neq r(\iota)$, then $r(\iota)=\theta(\iota)$ as remarked in \cref{remark:eneqsigma}. 
    In such a case, for $a\in (\theta/2, r(\iota))$, the set \eqref{equation:singleton} may have two elements and our estimate for $\sigma(\iota)/2\leq \|H\|<e(\iota)$ can be worse than \eqref{equation:cuplength}. 
    
    A possible way to prove \eqref{equation:cuplength} for $\sigma(\iota)/2\leq \|H\|<e(\iota)$ is to take $a=r(\iota)$ as mentioned in \cref{remark:construction-r}.
    In this case, the set \eqref{equation:singleton} is a singleton for each $(y,y')\in C(\iota ,H)$. 
    However, in the case $a=r(\iota)$, \cref{lemma:muhom-isom} does not hold and proofs for \cref{proposition:EndG,proposition:muhom-cl} become more complicated. 
\end{remark}

Note that we have proved $\# g(C(\iota,H))\geq \cl(L)+1$ assuming $C(\iota,H)$ is discrete in the proof of \cref{theorem:cuplength}. 
More generally, we obtain the following.
We denote by $\operatorname{pr}_1 \colon C(\iota,H) \to L$ the first projection.

\begin{proposition}\label{proposition:cuplengthgeneral}
    Assume that the number of the values of $g$ is finite.
    Let $g(C(\iota,H))=\{t_1,\dots ,t_l\}$ and set $T_i \coloneqq g^{-1}(t_i)$ for $i=1,\dots,l$. 
    Moreover, let $V_i$ be an open neighborhood of $\operatorname{pr}_1(T_i)$ in $L$ for $i=1,\dots,l$.
    Then
    \begin{equation}
        l + \sum_{i=1}^{l} \cl(H^*(V_i)) \ge \cl(L) + 1.
    \end{equation}
\end{proposition}

\begin{proof}
    Let $c_1,\dots,c_n$ be as in \cref{lemma:cup-length-ineq}. 
    Then $l \ge n$ and we obtain the result by \cref{lemma:cup-length-ineq} and a slight modified version of \cref{proposition:muhom-cl}.
\end{proof}

We can also deduce a similar statement without mentioning $g$. 

\begin{proposition}\label{proposition:cuplengthgeneral2}
    Assume that the number of the path-connected components of $C(\iota,H)$ is finite and let $\{C_1,\dots ,C_m\}$ be the set of the path-connected components.
    Moreover, let $U_i$ be an open neighborhood of $\operatorname{pr}_1(C_i)$ in $L$ for $i=1,\dots,m$.
    Then
    \begin{equation}
        m + \sum_{i=1}^{m} \cl(H^*(U_i)) \ge \cl(L)+ 1.
    \end{equation}
\end{proposition}

\begin{proof}
    Since $g$ is constant on each path-connected component, the assumption of \cref{proposition:cuplengthgeneral} is satisfied. 
    We use the same notation as in \cref{proposition:cuplengthgeneral}. 
    We define $\kappa \colon \{1,\dots, m \}\to\{1,\dots,l\}$ so that $C_i\subset D_{\kappa(i)}$ for each $i=1,\dots,m$. 
    Set $V_j \coloneqq \bigcup_{i\in \kappa^{-1}(j)}U_i$. 
    It is enough to check that 
    \begin{equation}
        \sum_{i\in \kappa^{-1}(j)} (\cl(H^*(U_i))+ 1) \ge \cl(H^*(V_j))+ 1
    \end{equation}
    for each $j\in \{1,\dots ,l\}$. 
    This is obtained by iterative use of \cref{lemma:cup-length-mv} below.
\end{proof}

\begin{lemma}\label{lemma:cup-length-mv}
    For open subsets $W_1$ and $W_2$ of $L$, 
    \begin{equation}
         \cl(H^*(W_1))+ \cl(H^*(W_2))+1\ge \cl(H^*(W_1\cup W_2)).
    \end{equation}
\end{lemma}

\begin{proof}
    By the Mayer–Vietoris sequence and \cref{lemma:cup-exact}, we get 
    \begin{equation} \cl(H^*(W_1\cup W_2))\le \cl(H^*(W_1\cap W_2))+\cl(H^*(W_1)\oplus H^*(W_2))+1.
    \end{equation}
    Since $\cl(H^*(W_1\cap W_2))\le \min \{ \cl(H^*(W_1)), \cl(H^*(W_1))\}$ and $\cl(H^*(W_1)\oplus H^*(W_2))=\max \{\cl(H^*(W_1)),\cl(H^*(W_2))\}$, the assertion holds. 
\end{proof}

Although in \cref{proposition:cuplengthgeneral2} we state the result for a strongly rational Lagrangian immersion satisfying \cref{assumption:cup-length}, 
we can show that the statement holds for any rational Lagrangian immersion as in \cref{theorem:cuplength} by an argument similar to \cref{subsec:reduction}.

\appendix

\section{Modified Tamarkin category and energy estimate}
\label{section:tamarkin-precise}

In this section, we give a more detailed exposition on the modified version of Tamarkin category $\cD^P(M)_\theta$.
We continue to use the notation in \cref{section:Tamarkincat}.

\subsection{Separation theorem}\label{subsection:appendix-separation}

First noticing that $\ell \colon M \times P \times \bR \to M \times P \times \Stheta$ is a covering map, we obtain the following.

\begin{lemma}\label{lemma:bRtoS1}
	\begin{enumerate}
		\item Let $G\in \Db_{/[1]}(M \times P \times \Stheta)$.
		If $\ell^{!}G\simeq 0$ then $G\simeq 0$.
		\item
		The functor $\ell^{!}\colon \Db_{/[1]} (M \times P \times \Stheta)
		\to \Db_{/[1]} (M \times P \times \bR)$ is conservative.
		That is, a morphism $f$ in $\Db_{/[1]}(M \times P \times
		\Stheta)$ is an isomorphism if and only if so is $\ell^{!}f$.
	\end{enumerate}
\end{lemma}

Applying the proper base change and the projection formula, one can prove the following.

\begin{lemma}\label{lemma:S1star}
	For $F,G \in \Db_{/[1]}(M \times P \times \bR)$,
	there is an natural isomorphism
	\begin{equation}
	R\ell_!F\star R\ell_!G \simeq R\ell_!(F\star G),
	\end{equation}
	where $\star$ in the right-hand side stands for the convolution $\star$ for the case $\theta=0$.
\end{lemma}

We define endofunctors $P_l$ and $P_r$ of $\Db_{/[1]}(M \times P \times \Stheta)$ by $P_l \coloneqq R\ell_! \bfk_{M \times P \times [0,+\infty)} \star (-)$ and $P_r \coloneqq \cHom^\star(R\ell_! \bfk_{M \times P \times [0,+\infty)}, -)$. 
Using \cref{lemma:bRtoS1,lemma:S1star} and arguing similarly to \cite{GS14}, we can show the equivalence of categories 
\begin{align*}
	P_l	\colon \Db_{/[1]}(M 	\times P \times \Stheta;\Omega_+) 
	& \simto {}^\perp \Db_{/[1], \{\tau \le 0\}}(M \times P \times \Stheta), \\
	P_r	\colon \Db_{/[1]}(M \times P \times \Stheta;\Omega_+) 
	& \simto \Db_{/[1], \{\tau \le 0\}}(M \times P \times \Stheta)^\perp,
\end{align*}
where ${}^\perp (-)$ (resp.\ $(-)^\perp$) denotes the left (resp.\ right) orthogonal.

For an object $F \in \cD^P(M)$, we take the canonical representative $P_l(F) \in
{}^\perp \Db_{\{\tau \le 0\}}(M \times P \times \bR)$ unless otherwise specified.
The support of an object $F \in \cD^P(M)$ is defined to be that of $P_l(F)$.
For a compact subset $A$ of $T^*M$ and $F \in \cD^P_A(M)$, the canonical representative $P_l(F) \in {}^\perp \Db_{\{\tau \le 0\}}(M
\times P \times \bR)$ satisfies $\MS(P_l(F)) \subset \overline{\rho^{-1}(A)}$.

The following is a slight generalization of Tamarkin's separation theorem.

\begin{proposition}[{see also \cite[Thm.~3.2 and Lem.~3.8]{Tamarkin} and \cite[Thm.4.28]{GS14}}]\label{proposition:separation}
    Let $q$ denote the projection $M \times P\times \Stheta \to \Stheta$.
	Let $A,B$ be compact subsets of $T^*M$ and $F \in \cD^P_{A}(M)_\theta,\allowbreak G \in \cD^P_{B}(M)_\theta$.
	Assume 
	\begin{enumerate}
		\renewcommand{\labelenumi}{$\mathrm{(\arabic{enumi})}$}
		\item $A \cap B=\emptyset$,
		\item $q$ is proper on $\Supp(F) \cup \Supp(G)$.
	\end{enumerate}
	Then $R{q}_* \cHom^\star(F,G) \simeq 0$.
\end{proposition}

\subsection{Sheaf quantization of Hamiltonian isotopies}\label{subsection:appendix-sq-hamiltonian}

In this subsection, we briefly recall the existence theorem of a sheaf quantization of a Hamiltonian isotopy due to Guillermou--Kashiwara--Schapira~\cite{GKS}, 
with a slight modification so that it can be applied to our setting.

Let $I$ be an open interval containing the closed interval $[0,1]$.
Let $H \colon T^*M \times I \to \bR$ be a compactly supported Hamiltonian function and denote by $X_s$ the associated Hamiltonian vector field on $T^*M$ defined by $d\alpha(X_s,-)=-dH_s$.
We also denote by $\phi^H \colon T^*M \times I \to T^*M$ the Hamiltonian isotopy generated by $X_s$.	
We consider the conification of $\phi^H$ as follows.
Define $\wh{H} \colon T^*M \times \rT \Stheta \times I \to \bR$ by $\wh{H}_s(x,t;\xi,\tau) \coloneqq \tau \cdot H_s(x;\xi/\tau)$.
Note that $\wh{H}$ is homogeneous of degree $1$, that is, $\wh{H}_s(x,t;c\xi,c\tau)=c \cdot \wh{H}_s(x,t;\xi,\tau)$ for any $c \in \bR_{>0}$.
The Hamiltonian isotopy $\wh{\phi} \colon T^*M \times \rT \Stheta \times I \to T^*M \times \rT \Stheta$ associated with $\wh{H}$
makes the following diagram commute (recall that we have set $\rho \colon \Omega_+ \to T^*M, (x,t;\xi,\tau) \mapsto (x;\xi/\tau)$):
\begin{equation}\label{diag:homog}
\begin{aligned}
\xymatrix{
	\Omega_+ \times I \ar[r]^-{\wh{\phi}} \ar[d]_-{\rho \times \id} & \Omega_+ \ar[d]^-{\rho} \\
	T^*M \times I \ar[r]_-{\phi^H} & T^*M.
}
\end{aligned}
\end{equation}
Defining a $C^\infty$-function $u=(u_s)_{s \in I} \colon T^*M \times I \to \bR$ by $u_s(p) \coloneqq \int_0^s (H_{s'}-\alpha(X_{s'}))(\phi_{s'}^H(p))ds'$, we find that 
\begin{equation}
\wh{\phi}_s(x,t;\xi,\tau)
=
(x',t+[u_s(x;\xi/\tau)];\xi',\tau),
\end{equation}
where $(x';\xi'/\tau)=\phi^H_s(x;\xi/\tau)$.
By construction, $\wh{\phi}$ is a homogeneous Hamiltonian isotopy: $\wh{\phi}_s(x,t;c\xi,c\tau)=c \cdot \wh{\phi}_s(x,t;\xi,\tau)$ for any $c \in \bR_{>0}$.
We define a conic Lagrangian submanifold $\Lambda_{\wh{\phi}} \subset T^*M \times \rT \Stheta \times T^*M \times \rT \Stheta \times T^*I$ by
\begin{equation}\label{equation:deflambdahatphi}
\Lambda_{\wh{\phi}}
 \coloneqq 
\left\{
\left(
\wh{\phi}_s(x,t;\xi,\tau), (x,t;-\xi,-\tau), (s;-\wh{H}_s \circ \wh{\phi}_s(x,t;\xi,\tau)) \right)
\; \middle| \;
\begin{aligned}
(x;\xi) & \in T^*M,  \\
(t;\tau) & \in \rT \Stheta, \\
s & \in I
\end{aligned}
\right\}.
\end{equation}
By construction, we have
\begin{equation}
\wh{H_s} \circ \wh{\phi}_s(x,t;\xi,\tau)
=
\tau \cdot (H_s \circ \phi^H_s(x;\xi/\tau)).
\end{equation}
Note also that
\begin{equation}
\begin{aligned}
\Lambda_{\wh{\phi}} \circ T^*_sI
&=
\left\{\left(\wh{\phi}_s(x,t;\xi,\tau), (x,t;-\xi,-\tau)\right) \; \middle| \; (x,t;\xi,\tau) \in T^*M \times \rT \Stheta \right\} \\
& \subset T^*M \times \rT \Stheta \times T^*M \times \rT \Stheta
\end{aligned}
\end{equation}
for any $s \in I$ (see \eqref{equation:compset} for the definition of $A \circ B$).
The following was proved by Guillermou--Kashiwara--Schapira~\cite{GKS}.

\begin{theorem}[{cf.\ \cite[Thm.~4.3]{GKS}}]\label{thm:GKS}
	In the preceding situation, there exists a unique object $K^H \in \Db(M \times \Stheta \times M \times \Stheta \times I)$ satisfying the following conditions:
	\begin{enumerate}
		\renewcommand{\labelenumi}{$\mathrm{(\arabic{enumi})}$}
		\item $\rMS(K^H) \subset \Lambda_{\wh{\phi}}$,
		\item $K^H|_{M \times \Stheta \times M \times \Stheta \times \{0\}} \simeq \bfk_{\Delta_{M \times \Stheta}}$, where $\Delta_{M \times \Stheta}$ is the diagonal of $M \times \Stheta \times M \times \Stheta$.
	\end{enumerate}
	Moreover both projections $\Supp(K^H) \to M \times \Stheta \times I$ are proper. 
\end{theorem}

The object $K^H$ is called the \emph{sheaf quantization} of $\wh{\phi}$ or associated with~$\phi^H$.

\subsection{Hamiltonian deformation for sheaves and translation distance}\label{subsec:appendix-distance}

In this subsection, we give the outline of the proof of \cref{proposition:distance}.
	
Let $F \in \cD^P(M)_\theta$.
Then the canonical morphism $R\ell_! \bfk_{M \times P \times [0,+\infty)} \star F \to F$ is an isomorphism.
Moreover, for any $c \in \bR$, we have an isomorphism ${T_c}_*(R\ell_! \bfk_{M \times P \times [0,+\infty)} \star F) \simeq R\ell_! \bfk_{M \times P \times [c,+\infty)} \star F$.
Hence, for any $c,d \in \bR$ with $c \le d$, the canonical morphism $\bfk_{M \times P \times [c,+\infty)} \to \bfk_{M \times P \times [d,+\infty)}$ induces a morphism $\tau_{c,d}(F) \colon {T_c}_*F \to {T_{d}}_*F$ in $\cD^P(M)_\theta$.
Using the morphism, we define the translation distance as in \cref{definition:distance}.

The following is a modified version of the key lemma in \cite{AI20}, which we used once in the proof of \cref{theorem:existence-quantization}.

\begin{lemma}[{cf.\ \cite[Prop.~4.3]{AI20}}]\label{lemma:torhtpy}
    Denote by $q \colon M \times P\times \Stheta \times I \to M \times P \times \Stheta$ the projection.
	Let $\cH \in \Db_{\{\tau \ge 0\}}(M\times P\times \Stheta \times I)$ and $s_1<s_2$ be in $I$.
	Assume that there exist $a,b,r \in \bR_{>0}$ satisfying
	\begin{equation}
	\MS(\cH) \cap \pi^{-1}(M \times P \times \Stheta \times (s_1-r,s_2+r)) \subset T^*(M \times P) \times (\Stheta \times I) \times \gamma_{a,b},
	\end{equation}
	where 
	$\gamma_{a,b} \coloneqq \{(\tau,\sigma) \in \bR^2 \mid -a \tau \le \sigma \le b \tau \} \subset \bR^2$.
	Then
	\begin{enumerate}
		\item $d_{\cD^P(M)_\theta}(Rq_*(\cH_{M \times P \times \Stheta \times [s_1,s_2)}),0) \le a(s_2-s_1)$,
		\item $d_{\cD^P(M)_\theta}(Rq_*(\cH_{M \times P \times \Stheta \times (s_1,s_2]})),0) \le b(s_2-s_1)$,
		\item $d_{\cD^P(M)_\theta}(\cH|_{M \times P \times \Stheta \times \{s_1\}}, \cH|_{M \times P \times \Stheta \times \{s_2\}}) \le (a+b)(s_2-s_1)$.
	\end{enumerate}
\end{lemma}

\begin{proof}[Outline of the proof]
	We can prove (i) and (ii) similarly to that of \cite[Prop.~4.3]{AI20}, 
	using \cref{lemma:appendix-cutoff} below		
	instead of the usual microlocal cut-off lemma.
	Similarly to \cite[Lem.~4.14]{AI20}, 
	we can show that if $F \to G \to H \toone$ is an exact triangle in 
	$\Db_{\{\tau \ge 0\}}(M\times P\times \Stheta)$ 
	and $d_{\cD^P(M)_\theta}(F,0) \le c$ with $c \in \bR_{\ge 0}$, 
	then $d_{\cD^P(M)_\theta}(G,H) \le c$.
	Hence, applying it to the exact triangles 
	\begin{equation}
	\begin{aligned}
		& Rq_*(\cH_{M \times P \times \Stheta \times (s_1,s_2]})
		\to
		Rq_*(\cH_{M  P \times \Stheta \times [s_1,s_2]})
		\to
		\cH|_{M \times P \times \Stheta \times \{s_1\}}
		\toone, \\
		& Rq_*(\cH_{M \times P \times \Stheta \times [s_1,s_2)})
		\to
		Rq_*(\cH_{M \times P \times \Stheta \times [s_1,s_2]})
		\to
		\cH|_{M \times P \times \Stheta \times \{s_2\}}
		\toone,
	\end{aligned}
	\end{equation}
	we obtain (iii) by the triangle inequality for $d_{\cD^P(M)_\theta}$.
\end{proof}

\begin{lemma}\label{lemma:appendix-cutoff}
	Define
	\begin{align*}
		\bar{s} \colon M\times P\times \Stheta\times \bR \times \bR\times \bR
		& \to M\times P \times \Stheta\times \bR, \\
		(x,y,t_1,s_1,t_2,s_2) & \mapsto (x,y,t_1+[t_2],s_1+s_2), \\
		\bar{q}_1
	\colon M \times P\times \Stheta\times \bR \times \bR\times \bR & \to M \times P\times \Stheta \times \bR, \\
	(x,y,t_1,s_1,t_2,s_2) & \mapsto (x,y,t_1,s_1), \\
	\bar{q}_2\colon M\times P\times \Stheta \times \bR \times \bR\times \bR & \to M\times P\times \bR\times \bR, \\
	(x,y,t_1,s_1,t_2,s_2) & \mapsto (x,y,t_2,s_2).
	\end{align*}
	Let $\gamma$ be a closed convex cone in $\bR^2$ with $0 \in \gamma$
	and $\cH \in \Db(M \times P \times \Stheta \times \bR)$.
	Then the canonical morphism
	$R\bar{s}_*(\bar{q}_1^{-1} \cH \otimes \bar{q}_2^{-1}\bfk_{M\times P\times \gamma}) \to \cH$ is an isomorphism if and only if
	$\MS(\cH) \subset T^*(M\times P) \times (\Stheta \times \bR) \times \gamma^\circ$, where $\gamma^\circ$ denotes the polar cone of $\gamma$.
\end{lemma}

\begin{proof}[Outline of the proof of \cref{proposition:distance}]
	Let $K^H$ be the sheaf quantization associated with $\phi^H$.
	Define $\cH \coloneqq K^H \circ F \in \Db(M \times P \times \Stheta \times I)$.
	Then we have $\cH|_{M \times P \times \Stheta \times \{0\}} \simeq F$ 
	and $\cH|_{M \times P \times \Stheta \times \{1\}} \simeq \Phi^H_1(F)$.
	By \cref{proposition:SScomp} and \eqref{equation:deflambdahatphi}, we get
	\begin{equation}
	\MS(\cH) \subset T^*(M \times P) \times \left\{ (t,s;\tau,\sigma) \;\middle|\; -\max_p H_s(p) \cdot \tau \le \sigma \le -\min_p H_s(p) \cdot \tau \right\}.
	\end{equation}
	Using \cref{lemma:torhtpy}(iii) 
	and arguing similarly to \cite[Prop.~4.15]{AI20}, 
	for any $n \in \bZ_{\ge 0}$ we obtain 
	\begin{equation}\label{equation:appendix-riemannsum}
		d_{\cD^P(M)_\theta}(F,\Phi^H_1(F)) 
		\le 
		\sum_{k=0}^{n-1} \frac{1}{n} \cdot 
		\left(
		\max_{s \in \left[ \frac{k}{n}, \frac{k+1}{n} \right]} f(s)
		+
		\max_{s \in \left[ \frac{k}{n}, \frac{k+1}{n} \right]} g(s)
		\right),
	\end{equation}
	where $f(s)=\max_p H_s(p)$ and $g(s)=-\min_p H_s(p)$.
	For any $\varepsilon \in \bR_{>0}$, there exists $n \in \bZ_{\ge 0}$ 
	such that the right-hand side of \eqref{equation:appendix-riemannsum}
	is less than $\|H\|+\varepsilon$, which proves the result.		
\end{proof}

As an application, we give a sheaf-theoretic bound for the displacement energy of two compact subset of $T^*M$.
For compact subsets $A$ and $B$ of $T^*M$, we define their displacement energy $e(A,B)$ by 
\begin{equation}
    e(A,B)
     \coloneqq 
    \inf
    \left\{
        \| H \| \; \middle| \;
        \begin{aligned}
            & \text{$H \colon T^*M \times I \to \bR$ with compact support}, \\
            & A \cap \phi^H_1(B)=\emptyset
        \end{aligned}
    \right\}
\end{equation}

Using $\cHom^\star$ and the translation distance on $\cD(\pt)_\theta$, we introduce a sheaf-theoretic energy.

\begin{definition}[{cf.\ \cite[Def.~4.17]{AI20}}]
    Let $q$ denote the projection $M \times P\times \Stheta \to \Stheta$.
    One defines
	\begin{equation}
	\begin{aligned}
		e_{\cD^P(M)_\theta}(F,G)
		 \coloneqq &
		d_{\cD(\pt)_\theta} \left(R{q}_*\cHom^\star(F,G),0 \right)\\
		=&\inf \left\{c \in \bR_{\ge 0} \; \middle| \; \tau_{0,c}(R{q}_*\cHom^\star(F,G))=0 \right\}.
	\end{aligned}
	\end{equation}
\end{definition}

Note that by \cref{proposition:morD}, we have
\begin{equation}
	e_{\cD^P(M)_\theta}(F,G)
	\ge
	\inf
	\left\{
    	c \in \bR_{\ge 0}
    	\; \middle| \;
    	\text{$\Hom_{\cD^P(M)_\theta}(F,G) \to \Hom_{\cD^P(M)_\theta}(F,{T_c}_*G)$ is zero}
	\right\}.
\end{equation}
		 
Combining \cref{proposition:separation} with \cref{proposition:distance},
we obtain the following refined version of the main theorem of \cite{AI20}. 
Note that we do not use this result in the previous sections, since we need more precise arguments for the estimates of the number of the intersection points.

\begin{proposition}[{cf.\ \cite[Thm.~4.18]{AI20}}]\label{proposition:energyestimate}
    Let $q$ denote the projection $M \times P\times \Stheta \to \Stheta$.
	Moreover, let $A$ and $B$ be compact subsets of $T^*M$.
	Then for any $F \in \cD^P_{A}(M)_\theta$ and $G \in \cD^P_{B}(M)_\theta$ such that $q$ is proper on $\Supp(F) \cup \Supp(G)$,
	\begin{equation}
	e(A,B) \ge e_{\cD^P(M)_\theta}(F,G).
	\end{equation}
	In particular, for such $F$ and $G$,
	\begin{equation}
	e(A,B)
	\ge
	\inf \{c \in \bR_{\ge 0} \mid \text{$\Hom_{\cD^P(M)_\theta}(F,G) \to
		\Hom_{\cD^P(M)_\theta}(F,{T_c}_*G)$ is zero} \}.
	\end{equation}
\end{proposition}

\newcommand{\etalchar}[1]{$^{#1}$}
\def\cprime{$'$}

\end{document}